\documentclass[11pt]{amsart}
\usepackage[leqno]{amsmath}
\usepackage{amsfonts, amssymb,srcltx, amsopn, color}
\usepackage{amsthm,thmtools,thm-restate}
\usepackage{hyperref}
\usepackage{amsrefs}

\usepackage{MnSymbol}

\usepackage{graphicx}

\usepackage{epsfig}
\usepackage{color}
\usepackage{amsmath}
\usepackage{amssymb}
\usepackage{graphicx}
\usepackage{xspace}
\usepackage{enumerate}

\usepackage{epsf}
\usepackage{psfrag}

\declaretheorem[numberwithin=section]{lemma}
\declaretheorem{proposition}
\declaretheorem{corollary}

\theoremstyle{definition}
\declaretheorem[sibling=lemma]{definition}
\declaretheorem[numbered=unless unique]{example}

\declaretheorem{claim}

\usepackage{color}
\newcommand{\commentout}[1]{}

\newcommand{\ta}{\ensuremath{\widetilde{a}\xspace}}
\newcommand{\tb}{\ensuremath{\widetilde{b}\xspace}}
\newcommand{\tc}{\ensuremath{\widetilde{c}\xspace}}
\newcommand{\td}{\ensuremath{\widetilde{d}\xspace}}

\newcommand{\tu}{\ensuremath{\widetilde{u}\xspace}}
\newcommand{\tv}{\ensuremath{\widetilde{v}\xspace}}
\newcommand{\tw}{\ensuremath{\widetilde{w}\xspace}}
\newcommand{\tx}{\ensuremath{\widetilde{x}\xspace}}
\newcommand{\ty}{\ensuremath{\widetilde{y}\xspace}}
\newcommand{\tz}{\ensuremath{\widetilde{z}\xspace}}
\newcommand{\tB}{\ensuremath{\widetilde{B}\xspace}}
\newcommand{\tS}{\ensuremath{\widetilde{S}\xspace}}
\newcommand{\tG}{\ensuremath{\widetilde{G}\xspace}}
\newcommand{\tX}{\ensuremath{\widetilde{\mathbf X}\xspace}}

\newcommand{\mx}{\ensuremath{x^{*}\!}}
\newcommand{\my}{\ensuremath{y^{*}\!}}
\newcommand{\mz}{\ensuremath{z^{*}\!}}

\newcommand{\bX}{\ensuremath{\mathbf{X}\xspace}}
\newcommand{\bY}{\ensuremath{\mathbf{Y}\xspace}}
\newcommand{\btX}{\ensuremath{\mathbf{\widetilde{X}}\xspace}}

\newcommand{\ovG}{\overline G}

\begin{document}

\thispagestyle{empty}
\centerline{\Large\bf Bucolic complexes}

\vspace{10mm}

\centerline{{\sc B. Bre\v sar$^{\small 1}$, J. Chalopin$^{\small 2}$, V.
Chepoi$^{\small 2}$,} {\sc T. Gologranc}$^{\small 3}$, and {\sc D.
Osajda$^{\small {4,5}}$} }

\vspace{3mm}

\date{\today}

\medskip
\begin{small}
\centerline{$^1$Faculty of Natural Sciences and Mathematics, University of Maribor,}
\centerline {Koro\v ska cesta 160, SI-2000 Maribor, Slovenia}
\centerline{\texttt{bostjan.bresar@uni-mb.si}}

\medskip
\centerline{$^{2}$Laboratoire d'Informatique Fondamentale, Aix-Marseille Universit\'e and CNRS,}
\centerline{Facult\'e des Sciences de Luminy, F-13288 Marseille Cedex 9, France}

\centerline{\texttt{\{jeremie.chalopin, victor.chepoi\}@lif.univ-mrs.fr}}

\medskip
\centerline{$^3$Institute of Mathematics, Physics and Mechanics,}
\centerline {Jadranska 19, SI-1000 Ljubljana, Slovenia}
\centerline{\texttt{tanja.gologranc@imfm.si}}

\medskip
\centerline{$^{4}$Universit\"at Wien, Fakult\"at f\"ur Mathematik}
\centerline{Nordbergstra\ss e 15, 1090 Wien, Austria}

\medskip
\centerline{$^{5}$Instytut Matematyczny, Uniwersytet Wroc{\l}awski (on leave),}
\centerline{pl. Grunwaldzki 2/4, 50-384 Wroc{\l}aw, Poland}

\centerline{\texttt{dosaj@math.uni.wroc.pl}}

\end{small}

\bigskip\bigskip\noindent {\footnotesize {\bf Abstract.}  We introduce and investigate bucolic complexes, a common generalization of systolic complexes and of CAT(0) cubical complexes. They are defined as simply connected prism complexes satisfying some local combinatorial conditions. We study various approaches
to bucolic complexes: from graph-theoretic and topological perspective, as well as from the point of view of geometric group theory.
In particular, we characterize bucolic complexes by some properties of their $2$--skeleta and $1$--skeleta (that we call bucolic graphs),
by which several known results are generalized. We also show that
locally-finite bucolic complexes are contractible, and satisfy some nonpositive-curvature-like properties.
}

\section{Introduction}
\label{s:intro}
CAT(0) cubical complexes and systolic (simplicial) complexes constitute two classes of polyhedral complexes
that have been intensively explored over last decades. Both CAT(0) cubical and systolic complexes exhibit various
properties typical for spaces with different types of nonpositive curvature. Hence groups of isomorphisms of
such complexes provide numerous examples of groups with interesting properties. Both  CAT(0) cubical complexes
and systolic complexes can be nicely characterized via their 1- and 2-skeleta.  It turns out that their
1-skeleta --- median graphs and bridged graphs --- are intensively studied in various areas of discrete mathematics
(see Section \ref{s:related_work} for related results and references).

In this article we introduce a notion of \emph{bucolic complexes} --- polyhedral complexes being a common generalization of
CAT(0) cubical~\cite{Gr,Sag}, systolic~\cite{Ch_CAT,Hag,JaSw}, and weakly systolic~\cite{Osajda} complexes, and we initiate a regular study of them. Analogously to CAT(0) cubical and systolic complexes,
bucolic complexes are defined as simply connected prism complexes satisfying some local combinatorial conditions
(see Subsection~\ref{s:bucdef} for details).
Our main result on bucolic complexes is the following characterization via their 1- and 2-skeleta (see Section~\ref{prel} for explanations of all the notions involved).

\begin{restatable}{theorem}{thcomplexes}\label{theorem0} For a prism complex $\bX$, the following conditions are equivalent:

\begin{itemize}
\item[(i)] $\bX$ is a bucolic complex;
\item[(ii)] the $2$-skeleton $\bX^{(2)}$ of $\bX$ is a connected and simply connected triangle-square flag
complex satisfying the wheel, the 3-cube, and the 3-prism conditions;
\item[(iii)] the $1$-skeleton $G(\bX)$ of $\bX$ is a connected weakly
  modular graph that does not contain induced subgraphs of the form
  $K_{2,3},$ $W_4$, and $W_4^-$, i.e., $G(\bX)$ is a
  bucolic graph not containing infinite hypercubes.
\end{itemize}
Moreover, if ${\bf X}$ is a connected flag prism complex satisfying the wheel, the cube, and the prism conditions,
 then the universal cover $\widetilde{\bf X}$ of $\bf X$ is bucolic.
\end{restatable}

As an immediate corollary we obtain an analogous characterization (Corollary~\ref{cor-strong-buc-complexes} in Section~\ref{proof_triangle-square-complex_bridged}) of \emph{strongly bucolic complexes} --- the subclass of bucolic
complexes containing products of systolic complexes but not all weakly systolic complexes (see Subsection~\ref{s:bucdef} for details).
The condition (iii) in the above characterization is a global condition --- weak modularity concerns balls of arbitrary radius;
cf.\ Section~\ref{prel}. Thus the theorem --- and in particular the last assertion --- may be seen as a local-to-global result
concerning polyhedral complexes. It is an analogue of the Cartan-Hadamard theorem appearing in various contexts of
non-positive-curvature: CAT(0) spaces \cite{BaHe}, Gromov hyperbolic spaces \cite{Gr}, systolic and weakly systolic complexes \cites{JaSw,Osajda}.

The $1$--skeleta of CAT(0) cubical complexes are exactly the median graphs~\cite{Ch_CAT,Ro} which constitute a central graph class in metric graph theory (see \cite{BaCh_survey} and the references therein). In the literature there are numerous structural and other characterizations of median graphs. In particular, median graphs are the retracts of hypercubes~\cite{Ba_retract},
and can be obtained via so-called iterated gated amalgamations from cubes~\cites{Is,vdV1}. The general framework of fiber-complemented graphs was introduced in~\cite{Cha1,Cha2} and allows to prove such decomposition and retraction results.
From this perspective, \emph{bucolic graphs} are  exactly the
fiber-complemented graphs whose elementary gated subgraphs are weakly-bridged; more precisely, the $1$--skeleta of bucolic complexes admit the following characterization.

\begin{restatable}{theorem}{thgraphs}\label{theorem1} For a graph $G=(V,E)$ not containing  infinite cliques, the following conditions are equivalent:

\begin{itemize}
\item[(i)] $G$ is a retract of the (weak) Cartesian product of weakly
  bridged (respectively, bridged) graphs;
\item[(ii)] $G$ is a weakly modular graph not containing induced
  $K_{2,3},$ $W_4,$ and $W_4^-$ (respectively, $K_{2,3}$, $W_4^-$, $W_4$, and
  $W_5$), i.e., $G$ is a bucolic (respectively, strongly bucolic) graph;
\item[(iii)] $G$ is a  weakly modular graph not containing $K_{2,3}$
  and $W_4^-$ in which all elementary (or prime) gated subgraphs are edges or 2-connected
  weakly bridged (respectively, bridged) graphs.
\end{itemize}
\noindent
Moreover, if $G$ is finite, then the conditions (i)-(iii) are equivalent to the following condition:
\begin{itemize}
\item[(iv)] $G$ can be obtained by successive applications of gated
amalgamations from Cartesian products of 2-connected weakly bridged (respectively, bridged) graphs.
\end{itemize}
\end{restatable}

Theorem~\ref{theorem1} allows us to show further non-positive-curvature-like properties of bucolic complexes. The first one completes the analogy with the Cartan-Hadamard theorem.

\begin{restatable}{theorem}{thcontractible}\label{contractible} Locally-finite bucolic complexes are contractible.
\end{restatable}

In particular, the above theorem provides local conditions (the ones appearing in the definition of bucolic complexes) for a prism complex  implying asphericity. Note that this is one of not so many local asphericity criterions known --- most of them appears in the case of some nonpositive curvature.

Similarly to the case of CAT(0) cubical groups and weakly systolic groups, we think that groups acting on bucolic complexes form an important class and deserve further studies. We believe that they have similar properties as groups acting on nonpositively curved spaces, and that they may provide many interesting examples. In the current paper we indicate two basic results on such groups.

\begin{restatable}{theorem}{thfixedprism}\label{fixed_prism}  If $\bX$ is a locally-finite bucolic complex and  $F$ is a finite
group acting by cell automorphisms on $\bX$, then there exists a prism
$\pi$ of $\bX$ which is invariant under the action of $F.$ The center
of the prism $\pi$ is a point fixed by $F$.
\end{restatable}

A standard argument gives
the following immediate consequence of Theorem \ref{fixed_prism}.

\begin{corollary}
\label{conj}  Let $F$ be a group acting geometrically
by automorphisms on a locally-finite bucolic complex $\bX$. Then $F$ contains
only finitely many conjugacy classes of finite subgroups.
\end{corollary}
\medskip

We prove Theorems~\ref{contractible} and \ref{fixed_prism} only for
locally-finite bucolic complexes.  We do not know whether these theorems
hold for non-locally-finite bucolic complexes and we leave this as an open question.

\medskip
\noindent
{\bf Article's structure.}
In the following Section \ref{prel} we introduce all the notions used later on. In Section \ref{s:related_work},
we review the related work on which our paper is based or which is generalized in our paper.
In Section~\ref{STh12}, we provide the characterization of bucolic
graphs (Theorem~\ref{theorem1}). A proof of the main characterization of bucolic
complexes (Theorem~\ref{theorem0}) is presented in
Section~\ref{proof_triangle-square-complex_bridged}.  In
Section~\ref{pf56}, we prove the contractibility and the fixed point
result for locally-finite bucolic complexes (Theorems~\ref{contractible}
and~\ref{fixed_prism}). In Section~\ref{moorability}, we complete the proof of Theorem \ref{theorem1}
in  the  non-locally-finite case.
\medskip

\noindent
{\bf Acknowledgments.} We wish to thank the anonymous referee for a careful
reading of the first version of the manuscript, his useful critical suggestions
about the presentation, and for inciting us to consider the non-locally-finite case. This research was supported
by the French-Slovenian Egide
PROTEUS project ``Distances, structure and products of graphs". B.B.
is also with the Institute of
Mathematics, Physics and Mechanics, Ljubljana, and
was also supported by the Ministry of Science and Technology of
Slovenia under the grants J1-2043 and P1-0297. V.C. was
also supported by the ANR Projets THEOMATRO (grant ANR-10-
BLAN 0207) and GGAA (grant ANR-10-BLAN 0116).
D.O. was partially supported by MNiSW grant
N N201 541738. The authors
would like to acknowledge the participants of the ANR project GRATOS for organizing
the workshop ``Journ\'ees Graphes et Structures Topologiques" in a {\sf bucolic} place
in the C\'evennes and for inviting three
of us to participate.

\section{Preliminaries}
\label{prel}

\subsection{Graphs}

All graphs $G=(V,E)$ occurring in this paper are undirected,
connected, without loops or multiple edges, but not
necessarily finite or locally-finite.  For two vertices $u$ and $v$ of a graph $G$, we
will write $u\sim v$ if $u$ and $v$ are adjacent and $u\nsim v$,
otherwise. We will use the notation $v\sim A$ to note that a vertex
$v$ is adjacent to all vertices of a set $A$ and the notation $v\nsim
A$ if $v$ is not adjacent to any of the vertices of $A$.  For a subset
$A\subseteq V,$ the subgraph of $G=(V,E)$  {\it induced by} $A$
is the graph $G(A)=(A,E')$ such that $uv\in E'$ if and only if $uv\in E$.
$G(A)$ is also called a {\it full subgraph} of $G$. We will say
that a graph $H$ is {\it not an induced subgraph} of $G$ if $H$ is not isomorphic
to any induced subgraph $G(A)$ of $G$.

By an $(a,b)$-{\it path} in a graph $G$ we
mean a sequence of vertices $P=(x_0=a,x_1, \ldots, x_{k-1},x_k=b)$ such that
any two consecutive vertices $x_i$ and $x_{i+1}$ of $P$ are different and
adjacent (notice that in general we may have $x_i=x_j$ if $|i-j|\ge 2$). If $k=2,$
then we call $P$ a {\it 2-path} of $G$. If $x_i\ne x_j$ for $|i-j|\ge 2$, then $P$ is
called a {\it simple $(a,b)$-path}.
A graph $G=(V,E)$ is {\it 2-connected} if any two vertices $a,b$ of $G$
can be connected by two vertex-disjoint $(a,b)$-paths. Equivalently, a
graph $G$ is 2-connected if $G$ has at least 3 vertices and $G(V\setminus\{ v\})$ is connected for any
vertex $v\in V$, i.e., $G$ remains connected after removing from $G$ any vertex $v$
and the edges incident to $v$.

The {\em wheel} $W_k$ is a graph obtained by connecting a single
vertex -- the {\em central vertex} $c$ -- to all vertices of the
$k$-cycle $(x_1,x_2, \ldots, x_k,x_1)$; the {\it almost wheel}
$W_k^{-}$ is the graph obtained from $W_k$ by deleting a spoke (i.e.,
an edge between the central vertex $c$ and a vertex $x_i$ of the
$k$-cycle), see Figure~\ref{fig-interdits}. The \emph{extended
  $5$-wheel} $\widehat{W}_5$ is a $5$-wheel $W_5$ plus a $3$-cycle
$(a,x_1,x_2,a)$ such that $a \nsim c,x_3,x_4,x_5$.

\begin{figure}[t]
\begin{center}
\includegraphics{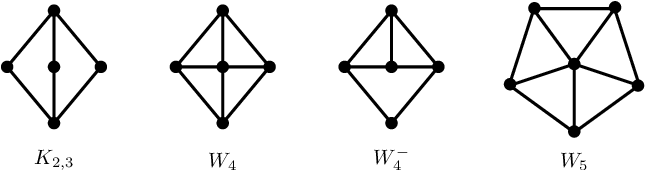}%
\end{center}
\caption{$K_{2,3}$, the wheel $W_4$, the almost-wheel $W_4^-$, and the
wheel $W_5$.}
\label{fig-interdits}
\end{figure}

The {\it  distance}
$d(u,v)=d_G(u,v)$ between two vertices $u$ and $v$ of a graph $G$ is the
length of a shortest $(u,v)$-path.  For a vertex $v$ of $G$ and an integer $r\ge 1$, we will denote  by $B_r(v,G)$ the \emph{ball} in $G$
(and the subgraph induced by this ball)  of radius $r$ centered at  $v$, i.e.,
$B_r(v,G)=\{ x\in V: d(v,x)\le r\}.$ More generally, the $r$-{\it ball  around a set} $A\subseteq V$
is the set (or the subgraph induced by) $B_r(A,G)=\{ v\in V: d(v,A)\le r\},$ where $d(v,A)=\mbox{min} \{ d(v,x): x\in A\}$.
As usual, $N(v)=B_1(v,G)\setminus\{ v\}$ denotes the set of neighbors of a vertex
$v$ in $G$.
The {\it interval}
$I(u,v)$ between $u$ and $v$ consists of all vertices on shortest
$(u,v)$-paths, that is, of all vertices (metrically) {\it between} $u$
and $v$: $I(u,v)=\{ x\in V: d(u,x)+d(x,v)=d(u,v)\}.$ An induced
subgraph of $G$ (or the corresponding vertex set $A$) is called {\it
  convex} if it includes the interval of $G$ between any pair of its
vertices. The smallest convex subgraph containing a given subgraph $S$
is called the {\it convex hull} of $S$ and is denoted by conv$(S)$. An induced
subgraph of $G$ (or the corresponding vertex set $A$) is called {\it
locally  convex} if it includes the interval of $G$ between any pair of its
vertices at distance two having a common neighbor in $A$. A graph $G=(V,E)$
is {\it isometrically embeddable} into a graph $H=(W,F)$
if there exists a mapping $\varphi : V\rightarrow W$ such that $d_H(\varphi (u),\varphi
(v))=d_G(u,v)$ for all vertices $u,v\in V$.

\begin{definition}[Retraction]
A {\it retraction}
$\varphi$ of a graph $G$ is an idempotent nonexpansive mapping of $G$ into
itself, that is, $\varphi^2=\varphi:V(G)\rightarrow V(G)$ with $d(\varphi
(x),\varphi (y))\le d(x,y)$ for all $x,y\in W$ (equivalently, a retraction is a
simplicial idempotent map $\varphi: G\rightarrow G$). The subgraph of $G$
induced by the image of $G$ under $\varphi$ is referred to as a {\it retract} of $G$.
\end{definition}

\begin{definition}[Mooring]
A map $f: V(G) \to V(G)$ is a \emph{mooring} of a graph $G$ onto $u$ if the following holds:
\begin{enumerate}
\item $f(u) = u$ and for every $v \neq u$, $f(v)\sim v$ and $d(f(v),u)
  = d(v,u)-1$.
\item for every edge $vw$ of $G$, $f(v)$ and $f(w)$ coincide or are
  adjacent.
\end{enumerate}
A graph $G$ is \emph{moorable} if, for every vertex $u$ of $G,$ there
exists a mooring of $G$ onto $u$.
\end{definition}

Mooring can be viewed as a combing property of graphs --- the notion coming
from  geometric group theory \cite{EpCaHoLePaTh}. Let $u$ be a
distinguished vertex (``base point") of a graph $G$.  Two shortest
paths $P(x,u),P(y,u)$ in $G$ connecting two adjacent vertices $x,y$ to
$u$ are called {\it $1$-fellow travelers} if $d(x',y')\le 1$ holds for
each pair of vertices $x'\in P(x,u), y'\in P(y,u)$ with
$d(x,x')=d(y,y').$ A {\it geodesic $1$-combing} of $G$ with respect to
the base point $u$ comprises shortest paths $P(x,u)$ between $u$ and
all vertices $x$ such that $P(x,u)$ and $P(y,u)$ are $1$-fellow
travelers for any edge $xy$ of $G$. One can select the combing paths
so that their union is a spanning tree $T_u$ of $G$ that is rooted at
$u$ and preserves the distances from $u$ to all vertices. The neighbor
$f(x)$ of $x\ne u$ in the unique path in $T_u$ connecting $x$ with the
root $u$ will be called the {\it father} of $x$ (set also $f(u)=u)$. Then
$f$ is a mooring of $G$ onto $u$ (vice-versa, any mooring of $G$ onto
$u$ can be viewed as a geodesic 1-combing with respect to $u$).  A
geodesic 1-combing of $G$ with respect to $u$ thus amounts to a tree
$T_u$ preserving the distances to the root $u$ such that if $x$ and
$y$ are adjacent in $G$ then $f(x)$ and $f(y)$ either coincide or are
adjacent in $G.$

\begin{definition}[Gated amalgam]
\label{d:gated}
An induced subgraph $H$ of a graph $G$
is {\it gated} \cite{DrSch} if for every vertex $x$ outside $H$ there
exists a vertex $x'$ in $H$ (the {\it  gate} of $x$)
such that  $x'\in I(x,y)$ for any $y$ of $H$.
A graph $G$ is a {\it gated amalgam} of two
graphs $G_1$ and $G_2$ if $G_1$ and $G_2$ are (isomorphic to) two intersecting
gated subgraphs of $G$ whose union is all of $G.$
\end{definition}

Gated sets are convex and the intersection of two gated sets is
gated. By Zorn lemma there exists a smallest gated subgraph containing
a given subgraph $S$, called the {\it gated hull} of $S$. A graph $G$
is said to be {\it elementary} \cite{Cha1} if the only proper gated
subgraphs of $G$ are singletons.

Let $G_{i}$, $i \in I$ be an arbitrary family of graphs. The
\emph{Cartesian product} $\Box_{i \in I} G_{i}$ is a graph whose vertices
are
all functions $x: i \mapsto x_{i}$, $x_{i} \in V(G_{i})$.
Two vertices $x,y$ are adjacent if there exists an index $j \in I$
such that $x_{j} y_{j} \in E(G_{j})$ and $x_{i} = y_{i}$ for all $i
\neq j$. Note that Cartesian product of infinitely many nontrivial
graphs is disconnected. Therefore, in this case the connected
components of the Cartesian product are called \emph{weak Cartesian
  products}. Since in our paper all graphs are connected, for us a
Cartesian product graph will always mean a weak Cartesian product
graph.  A graph with at least two vertices is said to be \emph{prime}
\cite{BaCh_weak,Cha1} if it is neither a Cartesian product nor a gated
amalgam of smaller graphs. A \emph{strong product} $\boxtimes_{i\in I}
G_i$ is a graph whose vertices are all functions $x: i \mapsto x_{i}$,
$x_{i} \in V(G_{i})$. Two vertices $x,y$ are adjacent if for all
indices $i\in I$ either $x_i=y_i$ or $x_{i} y_{i} \in E(G_{i}).$

\begin{definition}[Weakly modular graphs]
\label{d:weak-mod}
A graph $G$ is {\it weakly modular with respect to a vertex $u$}
 if its distance function $d$ satisfies the
following triangle and quadrangle conditions
(see Figure~\ref{fig-conditions}):
\begin{itemize}
\item
{\it Triangle condition} TC($u$):  for any two vertices $v,w$ with
$1=d(v,w)<d(u,v)=d(u,w)$ there exists a common neighbor $x$ of $v$
and $w$ such that $d(u,x)=d(u,v)-1.$
\item
{\it Quadrangle condition} QC($u$): for any three vertices $v,w,z$ with
$d(v,z)=d(w,z)=1$ and  $2=d(v,w)\le d(u,v)=d(u,w)=d(u,z)-1,$ there
exists a common neighbor $x$ of $v$ and $w$ such that
$d(u,x)=d(u,v)-1.$
\end{itemize}
A graph $G$ is {\it weakly modular} \cite{BaCh_helly} if $G$ is weakly modular with respect to any vertex $u$.
\end{definition}

Median, bridged, and weakly bridged graphs constitute three important
subclasses of weakly modular graphs.

\begin{definition}[Median graphs] A graph $G$ is {\it median} if it
is a bipartite weakly modular graph not containing $K_{2,3}$ as induced
subgraphs.
\end{definition}

Median graphs can be also defined in many other equivalent ways~\cite{ImKl,Mu,vdV}.
For example,  median graphs are exactly the graphs in which every triplet of vertices $u,v,w$ has a unique
{\it median}, i.e., a  vertex lying simultaneously in $I(u,v),I(v,w),$ and $I(w,u).$

\begin{definition}[Bridged and weakly bridged graphs]
A graph $G$ is {\it bridged} if it is weakly modular and does not
contain induced $4$- and $5$-cycles.  A graph $G$ is {\it weakly bridged}
if $G$ is a weakly modular graph and does not contain $4$-cycles.
\end{definition}

There exist other equivalent definitions of bridged
graphs~\cite{FaJa,SoCh}.
Bridged graphs are exactly the graphs that do not
contain isometric cycles of length greater than $3$.
Alternatively, a graph $G$ is bridged if and only if
the balls $B_r(A,G)$ around convex sets $A$ of $G$ are
convex. Analogously, a graph $G$ is weakly bridged if and only if $G$
has convex balls $B_r(x,G)$ and does not contain induced $C_4$~\cite{ChOs}.

Median, bridged, and weakly bridged graphs are pre-median graphs: a
graph $G$ is {\it pre-median} \cite{Cha1,Cha2} if $G$ is a weakly
modular graph without induced $K_{2,3}$ and $W^-_4$. An important
property of pre-median graphs is that within this class, a graph is
elementary if and only if it is prime~\cite{Cha1}.  Chastand
\cite{Cha1,Cha2} proved that pre-median graphs are fiber-complemented
graphs (a definition follows). Any gated subset $S$ of a graph $G$ gives rise to a partition
$F_{a}$ $(a\in S)$ of the vertex-set of $G;$ viz., the {\em fiber}
$F_{a}$ of $a$ relative to $S$ consists of all vertices $x$ (including
$a$ itself) having $a$ as their gate in $S.$ According to Chastand
\cite{Cha1,Cha2}, a graph $G$ is called {\it fiber-complemented} if
for any gated set $S$ all fibers $F_{a}$ $(a\in S)$ are gated sets of
$G$.

\begin{figure}[t]
\begin{center}
\includegraphics{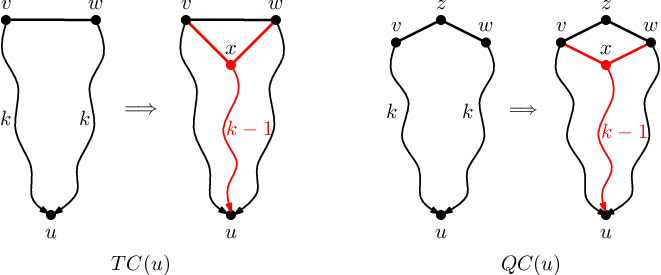}%
\end{center}
\caption{Triangle and quadrangle conditions}\label{fig-conditions}
\end{figure}

\subsection{Prism complexes}
In this paper, we consider a particular class of cell complexes
(compare e.g. \cite[p. 111-115]{BrHa}), called prism complexes, in
which all cells are prisms of finite dimension. Cubical and simplicial
cell complexes are particular instances of prism complexes. Although
most of the notions presented below can be defined for all cell
complexes and some of them for topological spaces, we will introduce
them only for prism complexes.

An {\it abstract simplicial complex} is a family $\bX$ of finite
subsets (of a given set) called {\it simplices} which is closed for
intersections and inclusions, i.e., $\sigma,\sigma'\in \bX$ and
$\sigma''\subset \sigma$ implies that $\sigma\cap\sigma', \sigma''\in
\bX$.  For an abstract simplicial complex $\bX$, denote by $V({\bX})$
and $E({\bX})$ the set of all 0-dimensional and 1-dimensional
simplices of ${\mathbf X}$ and call the pair $G(\bX)=(V({\mathbf
  X}),E({\mathbf X}))$ the {\it 1-skeleton} of $\bX$.  Conversely, for
a graph $G$ not containing infinite cliques, one can derive an
abstract simplicial complex ${\bX}_{\mbox{\footnotesize simpl}}(G)$
(the {\it clique complex} of $G$) by taking the vertex sets of all
complete subgraphs (cliques) as simplices of the complex.
By a \emph{simplicial complex} we will mean the geometric realization
of an abstract simplicial complex. It is a cell complex with cells corresponding to abstract
simplices, being (topologically) solid simplices.

A \emph{prism} is a convex polytope which is a Cartesian product of a
finite number of finite-dimensional simplices.  Faces of a prism are
prisms of smaller dimensions.  Particular instances of prisms are
simplices and cubes (products of intervals).  A {\it prism complex} is
a cell complex $\bX$ in which all cells are prisms so that the
intersection of two prisms is empty or a common face of each of them.
{\it Cubical complexes} are prism complexes in which all cells are
cubes and simplicial complexes are prism complexes in which all cells
are simplices.  The {\it 1-skeleton} $G({\bX})=\bX^{(1)}$ of a prism
complex $\bX$ has the 0-dimensional cells of $\bX$ as vertices and the
1-dimensional cells of $\bX$ as edges. The 1-skeleton of a prism of
$\bX$ is a {\it Hamming graph}, i.e., the Cartesian product of
complete subgraphs of $G(\bX).$ For vertices $v,w$ or a set of
vertices $A$ of a prism complex $\bX$ we will write $v\sim w,$ $v\sim
A$ (or $v\nsim w$, $v\nsim A$) if and only if a similar relation holds
in the graph $G(\bX)$. Note that a prism complex $\bX$ is connected if
and only $G(\bX)=\bX^{(1)}$ is a connected graph.  In this paper, all
prism complexes we consider are connected.  The $2$-skeleton
$\bX^{(2)}$ of $\bX$ is a \emph{triangle-square complex} obtained by
taking the 0-dimensional, 1-dimensional, and 2-dimensional cells of
$\bX$.  A prism complex $\bX$ is {\it simply connected} if every
continuous map $S^1\to X$ is null-homotopic. Note that $\bX$ is simply
connected if and only if $\bX^{(2)}$ is simply connected.  The \emph{star}
$\mbox{St}(v,\bX)$ of a vertex $v$ in a prism complex $\bX$ is the
subcomplex consisting of the union of all cells of $\bX$ containing
$v$.

For every graph $G$ that does not contain infinite cliques or infinite
hypercubes as induced subgraphs, $G$ gives rise to four cell
complexes: to a prism complex $\bX_{\mbox{\footnotesize prism}}(G)$,
to a simplicial complex $\bX_{\mbox{\footnotesize simpl}}(G)$, to a
cubical complex $\bX_{\mbox{\footnotesize cube}}(G),$ and to a
triangle-square complex $\bX_{\mbox{\footnotesize tr-sq}}(G).$ The
{\it prism complex} ${\bX}_{\mbox{\footnotesize prism}}(G)$ spanned by
$G$ has $P$ as a prism if and only if the 1-skeleton of $P$ is an
induced subgraph of $G$ which is a Hamming graph. Analogously, one can
define the {\it simplicial complex} $\bX_{\mbox{\footnotesize
    simpl}}(G)$ and the {\it cubical complex}
$\bX_{\mbox{\footnotesize cube}}(G)$ of $G$ as the complexes
consisting of all complete subgraphs and all induced cubes of $G$ as
cells, respectively.  In the same way, in the {\it triangle-square
  complex} $\bX_{\mbox{\footnotesize tr-sq}}(G)$ of $G$, the
triangular and square cells are spanned by the $3$-cycles and induced
$4$-cycles of $G.$ The triangle-square complex $\bX_{\mbox{\footnotesize
    tr-sq}}(G)$ of $G$ coincides with the $2$-skeleton of the prism
complex $\bX_{\mbox{\footnotesize simpl}}(G).$ Notice also that
$\bX_{\mbox{\footnotesize simpl}}(G)\bigcup \bX_{\mbox{\footnotesize
    cube}}(G)\subseteq \bX_{\mbox{\footnotesize prism}}(G).$ In
general, these four complexes can be pairwise distinct, but the graph
$G$ is the 1-skeleton of all these four complexes.

An abstract simplicial complex $\bX$ is a {\it flag complex} (or a {\it
clique complex}) if any set of vertices is included in a simplex of
$\mathbf X$ whenever each pair of its vertices is contained in a
simplex of ${\mathbf X}$ (in the theory of hypergraphs this condition
is called {\it conformality}; see for example~\cite{BaCh_survey}). A flag simplicial complex can therefore be recovered
from its underlying graph $G({\mathbf X})$ by the formula
${\mathbf X}={\bX}_{\mbox{\footnotesize simpl}}(G({\mathbf X}))$:
the complete subgraphs of
$G({\mathbf X})$ are exactly the simplices  of ${\mathbf X}.$   Analogously, a
prism (respectively, a cubical, or a triangle-square)  complex $\bX$ is a {\it flag
complex} if ${\mathbf X}={\bX}_{\mbox{\footnotesize prism}}(G({\mathbf X}))$
(respectively, ${\mathbf X}={\bX}_{\mbox{\footnotesize cube}}(G({\mathbf X}))$
or ${\mathbf X}={\bX}_{\mbox{\footnotesize tr-sq}}(G({\mathbf X}))$).
All  complexes occurring in this paper are flag complexes.

Let ${\bX}(W_k):={\bX}_{\mbox{\footnotesize tr-sq}}(W_k)$ and ${\bf X}(W^-_k):={\bX}_{\mbox{\footnotesize tr-sq}}(W^-_k)$
be the triangle-square
(or the prism) complexes whose underlying graphs are the graphs $W_k$
and $W^-_k$, respectively (the first consists of $k$ triangles and the
second consists of $k-2$ triangles and one square). Analogously, let
${\bX}(\widehat{W_5})={\bX}_{\mbox{\footnotesize simpl}}(\widehat{W_5})$ be the $2$-dimensional
simplicial complex made of 6 triangles and whose underlying graph is the extended
$5$-wheel $\widehat{W_5}$.

As morphisms between cell complexes we consider all \emph{cellular
  maps}, i.e.,  maps sending (linearly) cells to cells. An
\emph{isomorphism} is a bijective cellular map being a linear
isomorphism on each cell. A \emph{covering (map)} of a cell
complex $\bX$ is a cellular surjection $p\colon \widetilde{\bX} \to
\bX$ such that $p|_{\mbox{St}(\tv,\widetilde{\bX})}\colon
\mbox{St}(\tv,\widetilde{\bX})\to \mbox{St}(v,\bX)$ is an isomorphism
for every vertex $v$ in $\bX$, and every vertex $\tv \in
\widetilde{\bX}$ with $p(\tv)=v$; compare \cite[Section 1.3]{Ha}.  The
space $\widetilde{\bX}$ is then called a \emph{covering space}.  A
\emph{universal cover} of $\bX$ is a simply connected covering space
$\widetilde{\bX}$. It is unique up to isomorphism; cf.\ \cite[page
  67]{Ha}.  In particular, if $\bX$ is simply connected, then
its universal cover is $\bX$ itself.  A group $F$ \emph{acts by
  automorphisms} on a cell complex $\bX$ if there is a homomorphism
$F\to \mbox{Aut}(\bX)$ called an \emph{action of $F$}. The action is
\emph{geometric} (or \emph{$F$ acts geometrically}) if it is proper
(i.e. cells stabilizers are finite) and cocompact (i.e. the quotient
$\bX/F$ is compact).

\subsection{CAT(0) cubical complexes and systolic complexes}

A {\it geodesic triangle} $\Delta=\Delta (x_1,x_2,x_3)$ in a geodesic
metric space $(X,d)$ consists of three points in $X$ (the vertices
of $\Delta$) and a geodesic  between each pair of vertices (the
sides of $\Delta$). A {\it comparison triangle} for $\Delta
(x_1,x_2,x_3)$ is a triangle $\Delta (x'_1,x'_2,x'_3)$ in the
Euclidean plane  ${\mathbb E}^2$ such that $d_{{\mathbb
E}^2}(x'_i,x'_j)=d(x_i,x_j)$ for $i,j\in \{ 1,2,3\}.$

\begin{definition}[CAT(0) spaces]
 A geodesic metric space $(X,d)$ is  a {\it CAT(0)  (or a nonpositively curved)  space}
if the geodesic triangles in ${\bf X}$ are thinner than their comparison triangles in the Euclidean plane
\cite{Gr}, i.e, if $\Delta (x_1,x_2,x_3)$ is a geodesic triangle of $X$ and $y$ is a point on the side
of $\Delta(x_1,x_2,x_3)$ with
vertices $x_1$ and $x_2$ and $y'$ is the unique point on the line
segment $[x'_1,x'_2]$ of the comparison triangle
$\Delta(x'_1,x'_2,x'_3)$ such that $d_{{\mathbb E}^2}(x'_i,y')=
d(x_i,y)$ for $i=1,2,$ then $d(x_3,y)\le d_{{\mathbb
E}^2}(x'_3,y').$
\end{definition}

CAT(0) spaces satisfy many nice metric and convexity properties and can be characterized
in various ways (for a full account of this theory consult the book \cite{BrHa}).
For example, any two points of a CAT(0) space can be joined by a unique geodesic. CAT(0) property
is also equivalent to convexity of the function $f:[0,1]\rightarrow X$ given by
$f(t)=d(\alpha (t),\beta (t)),$ for any geodesics $\alpha$ and
$\beta$ (which is further equivalent to convexity of the
neighborhoods of convex sets). This implies that CAT(0) spaces are
contractible.

\begin{definition}[CAT(0) cubical complexes]  A cubical complex $\bX$ is a {\it CAT(0) cubical complex} if $\bX$ endowed with intrinsic
$\ell_2$-metric is a CAT(0) space.
\end{definition}

Gromov \cite{Gr} gave a  beautiful  characterization of CAT(0) cubical complexes as simply connected cubical complexes satisfying the following
combinatorial condition:  if three $k$-cubes
pairwise intersect in a $(k-1)$-cube and all three intersect in a
$(k-2)$-cube, then are included in a $(k+1)$-dimensional cube. This condition can be equivalently formulated as the requirement
that the links of 0-cubes are simplicial flag complexes.

\begin{figure}[t]
\begin{center}
\includegraphics{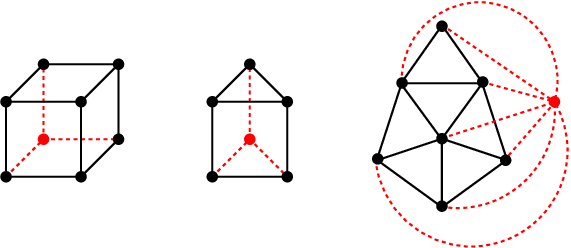}%
\end{center}
\caption{The 3-cube condition (left), the 3-prism condition (middle), and the
 $\widehat{W}_5$-wheel condition (right).}
\label{fig-conditions-X}
\end{figure}

Now we briefly recall the definitions of systolic and weakly systolic simplicial complexes, which are both considered as
simplicial complexes with combinatorial nonpositive curvature. For an integer $k\ge 4$, a flag simplicial complex $\bX$ is {\it locally $k$-large} if every cycle consisting of less than $k$
edges in any of its links of simplices has some two consecutive edges contained in
a $2$-simplex of this link, i.e., the links do not contain induced cycles of length $<k$.

\begin{definition}[Systolic and weakly systolic complexes]  A simplicial complex is
$k$-{\it systolic} if it is locally $k$-large, connected and simply connected. A flag simplicial complex is
{\it systolic} if it is 6-systolic \cite{Ch_CAT,JaSw,Hag}. A  simplicial complex $\bX$ is {\it weakly systolic} \cite{Osajda,ChOs}
if $\bX$ is flag,  connected and simply connected, locally 5-large, and
satisfies the following local condition:

\medskip\noindent {\it $\widehat{W}_5$-wheel condition}: for each
extended $5$-wheel $\bX(\widehat{W_5})$ of $\bX$, there exists a
vertex $v$ adjacent to all vertices of this extended $5$-wheel (see
Fig.~\ref{fig-conditions-X}, right).
\end{definition}

\subsection{Bucolic complexes and bucolic graphs}
\label{s:bucdef}
In this subsection we define central objects of the article: bucolic complexes and bucolic graphs\footnote{The term {\sf bucolic} is inspired by {\sf systolic}, where {\bf b} stands for {\sf bridged} and {\sf c}
for {\sf cubical}. See also Acknowledgments for another source of our ``inspiration''.}.

\begin{definition}[Bucolic complexes]
\label{d:buc}
A prism complex $\bX$ is {\it bucolic} if it is  flag, connected and simply connected, and satisfies the  following
three local conditions:

\smallskip\noindent
{\it Wheel condition:} the $1$-skeleton $\bX^{(1)}$ of $\bX$ does not contain induced $W_4$ and satisfies
the $\widehat{W}_5$-wheel condition;

\smallskip\noindent
{\it Cube condition:} if $k\ge 2$ and three $k$-cubes of ${\bX}$
pairwise intersect in a $(k-1)$-cube  and all three intersect in a $(k-2)$-cube,
then they  are included in a $(k+1)$-dimensional cube of $\bX$;

\smallskip\noindent
{\it Prism condition:} if a cube and a simplex of $\bX$ intersect in a
$1$-simplex, then they are included in a prism of $\bX$.
\medskip

A bucolic complex $\bX$ is \emph{strongly bucolic} if $G(\bX)$ does
not contain induced $W_5$, i.e., a prism complex $\bX$ is strongly
bucolic if it is flag, connected, simply connected, and satisfies the
cube and prism conditions, as well as the following local condition:

\smallskip\noindent
{\it Strong-wheel condition:} the $1$-skeleton $\bX^{(1)}$ of $\bX$ does not contain induced $W_4$ and $W_5$.

\end{definition}

As we already noticed, subject to simple connectivity, the wheel condition
characterizes the weakly systolic complexes in the class of flag simplicial complexes.
On the other hand,  the cube condition for cubical complexes is equivalent to
Gromov's condition of flagness of links. Finally, the prism condition for prism complexes
shows how simplices and cubes of  $\bX$ give rise to prisms.

Now, we consider the $2$-dimensional versions of the last two conditions. We introduce two combinatorial
conditions for a triangle-square complex ${\bX}$:

\smallskip\noindent
{\it 3-Cube condition}: any three squares of ${\bX}$, pairwise intersecting in an
edge, and all three intersecting in a vertex
of ${\bX}$, are included in the $2$-skeleton of a $3$-dimensional cube (see
Fig.~\ref{fig-conditions-X}, left);

\smallskip\noindent
{\it 3-Prism condition}: any house (i.e., a triangle and a square of
${\bX}$ sharing an edge) is included in the $2$-skeleton of a $3$-dimensional prism (see
Fig.~\ref{fig-conditions-X}, middle).

\smallskip
We conclude this section with the definition of bucolic and strongly bucolic graphs.

\begin{definition}[Bucolic graphs]
A graph $G$ is {\it bucolic} if it is weakly modular, does not contain
infinite cliques and does not contain induced subgraphs of the form
$K_{2,3},$ $W_4$, and $W_4^-$. A bucolic graph is \emph{strongly bucolic} if it does not contain
induced $W_5$.
\end{definition}

\section{Related work}\label{s:related_work}

In his seminal paper \cite{Gr}, among many other results, Gromov characterized  CAT(0) cubical  complexes as simply connected cubical complexes in which the links of 0-cubes are simplicial flag complexes. Subsequently,  Sageev~\cite{Sag} introduced and investigated the concept of (combinatorial) hyperplanes of CAT(0) cubical complexes, showing in particular that each hyperplane is itself a CAT(0) cubical complex and partitions the complex into two CAT(0) cubical complexes. These two results identify CAT(0) cubical complexes as the basic objects in a ``high-dimensional Bass-Serre theory'', and CAT(0) and nonpositively--curved cubical complexes have thus been studied extensively in geometric group theory.  For instance, many well-known classes of groups are known to act nicely on CAT(0) cubical complexes.

It was shown in \cite{Ch_CAT,Ro} that  the 1-skeleta of CAT(0)
cubical complexes are exactly the median graphs. This result establishes a bridge between two seemingly different mathematical structures and two different
areas of mathematics.   Median graphs and related structures (median algebras, event structures, copair Helly hypergraphs) occur in different areas of
discrete mathematics, universal algebra,  and theoretical computer science. Median graphs, median algebras, and
CAT(0) cubical complexes have many nice properties and admit numerous
characterizations. All median structures are intimately
related to hypercubes: median graphs are isometric subgraphs of
hypercubes; in fact, by a classical result of Bandelt \cite{Ba_retract}
they are the retracts of  hypercubes into which they embed isometrically. This isometric embedding of each median graph
into a hypercube canonically defines on the associated CAT(0) cubical complex a space with walls ``\`a la Haglund-Paulin'' and Sageev's hyperplanes.
It was also shown  by Isbell \cite{Is} and van de Vel \cite{vdV1} that
every finite median graph $G$ can be obtained by gated amalgams from hypercubes, thus showing that $K_2$ is
the only prime median graph. Median graphs also have a remarkable algebraic structure, which is
induced by the ternary operation on the vertex set that assigns to
each triplet of vertices the unique median vertex, and their algebra can be characterized
using four natural axioms \cite{BaHe,Is} among all discrete ternary algebras.
For more detailed information about median structures, the interested reader can
consult the survey
\cite{BaCh_survey} and the books \cite{ImKl,Mu,vdV}.

Bridged graphs are the graphs in which all isometric cycles have length $3$. It was shown in \cite{FaJa,SoCh} that
the bridged graphs are exactly the graphs in which the metric convexity satisfies one of the basic
properties of the CAT(0) geometry: neighborhoods of convex sets are
convex. Combinatorial and structural aspects of bridged graphs have been investigated in
\cite{AnFa,ChLaPo,Po}. In particular, it was shown in \cite{AnFa} that
finite bridged graphs are dismantlable.
Similarly to the local-to-global characterization of CAT(0) cubical complexes of \cite{Gr},
it was shown in \cite{Ch_CAT}
that the clique complexes of bridged graphs are exactly the simply connected simplicial
flag complexes  in which the links of vertices do not contain induced $4$- and $5$-cycles.
These complexes have been rediscovered and investigated in depth by Januszkiewicz and Swiatkowski \cite{JaSw}, and,
independently by Haglund \cite{Hag}, who called them ``systolic complexes"
and considered them as simplicial complexes satisfying combinatorial
nonpositive curvature property. In general, these complexes are not  CAT(0).
More recently, Osajda \cite{Osajda} proposed a generalization of systolic complexes still preserving most of the structural properties of systolic
complexes: the resulting weakly systolic complexes and their 1-skeleta -- the weakly bridged graphs -- have been investigated and
characterized in \cite{ChOs}.

The structure theory of graphs based on Cartesian multiplication and gated
amalgamation was further elaborated for more general classes of graphs.  Some of the results for median
graphs have been extended to quasi-median graphs introduced by Mulder \cite{Mu} and further
studied in \cite{BaMuWi}: quasi-median graphs are the weakly modular graphs not containing induced
$K_{2,3}$ and $K_4-e$ and they can  be characterized as
the retracts of Hamming graphs (finite quasi-median graphs can be obtained from  complete
graphs by Cartesian products and gated amalgamations). More recently, Bandelt and  Chepoi \cite{BaCh_weak} presented a similar decomposition scheme
of weakly median graphs (the weakly modular graphs in which the vertex $x$ in the triangle and quadrangle conditions is unique) and characterized
the prime graphs with respect to this decomposition: the hyperoctahedra and their subgraphs, the 5-wheel $W_5$, and the 2-connected
plane bridged graphs.   Generalizing the proof of the decomposition
theorem of \cite{BaCh_weak}, Chastand \cite{Cha1,Cha2} presented a general framework of fiber-complemented graphs allowing to
establish many  general properties, previously
proved only for particular classes of graphs.  An important subclass
of  fiber-complemented graphs is the class of pre-median graphs \cite{Cha1,Cha2}. It is an open problem
to characterize all prime (elementary) fiber-complemented or pre-median graphs (see \cite[p. 121]{Cha1}).

Since CAT(0) cubical complexes and systolic simplicial complexes can be both characterized via their
1-skeleta and via simple connectivity and local conditions, a natural question is to find a common generalization
of such complexes which still obey the combinatorial nonpositive curvature properties.
The prism complexes derived from
fiber-complemented graphs is a potential source of such cell complexes. In \cite{BrChChKoLaVa},
answering a question from \cite{brth-09}, the first step in this direction was taken, and
the 1-skeleta of prism complexes resulting from clique complexes of chordal graphs by applying Cartesian
products and gated amalgams have been characterized. It was also shown that, endowed with the $l_2$-metric,
such prism complexes are CAT(0) spaces.

In this paper, we continue this line of research and characterize the
graphs $G$ which are retracts of Cartesian products of weakly bridged
and bridged graphs. We show (cf. Theorem \ref{theorem1}) that
these graphs are exactly the bucolic and strongly bucolic graphs.
We also establish that the finite bucolic (respectively, strongly bucolic)
graphs are exactly the graphs obtained by gated amalgamations of
Cartesian products of weakly bridged (respectively, bridged) graphs, thus
answering Question 1 from \cite{BrChChKoLaVa}. This also provides a
partial answer to Chastand's problem mentioned above (by showing that
the weakly bridged graphs are exactly the prime graphs of pre-median
graphs without $W_4$ and that the bridged graphs are the prime graphs
of pre-median graphs without $W_4$ and $W_5$) and extends the
analogous results on finite median, quasi-median, and weakly median
graphs.  Our previous result can be viewed as a characterization of
1-skeleta of bucolic complexes.  We also characterize (Theorem
\ref{theorem0}) bucolic complexes via their 2-skeleta by showing
that they are exactly the simply connected triangle-square complexes
satisfying the 3-cube and 3-prism conditions and not containing $W_4$,
$W_5,$ and $W^-_5$ (this answers Question 2 from \cite{BrChChKoLaVa})
and, together with the first result, generalizes the characterizations
of CAT(0) cubical complexes, systolic and weakly systolic complexes
via theirs 1- and 2-skeleta.

Then we prove that the locally-finite bucolic complexes are contractible (Theorem
\ref{contractible}).
Thus the three results constitute a version of the Cartan-Hadamard theorem,
saying that under some local conditions the complex is aspherical, i.e.
its universal covering space is contractible. Only limited number of such local
characterizations of asphericity is known, and most of them refer to the notion
of nonpositive curvature; cf. e.g. \cite{BrHa,EpCaHoLePaTh,Gr,JaSw,Osajda}.
In fact bucolic complexes exhibit many nonpositive-curvature-like properties.
Besides the Cartan-Hadamard theorem we prove the fixed point theorem for finite
groups acting on locally-finite bucolic complexes (Theorem \ref{fixed_prism}), and we conclude
that groups acting geometrically on
such complexes have finitely
many conjugacy classes of finite subgroups (Corollary \ref{conj}).
Counterparts of such results are known for other nonpositively curved spaces;
cf. e.g. \cite{BrHa,ChOs,JaSw,Osajda}.
Thus our classes of complexes and groups acting on them geometrically form
new classes of combinatorially nonpositively curved complexes and groups (see e.g.
\cite{Gr,JaSw,Osajda,ChOs} for more background) containing the CAT(0) cubical
and systolic classes of objects. A question of studying such unification
theories was raised often by various researchers, e.g. by Januszkiewicz
and {\'S}wi{\c{a}}tkowski (personal communication). Due to our knowledge,
bucolism is the first
generalization of the CAT(0) cubical and systolic worlds studied up to now.

The class of bucolic complexes is closed under taking finite Cartesian products and
gated amalgamations. Thus
the class of groups
acting geometrically on them is also closed under similar operations.
It should be noticed that both systolic and CAT(0) cubical groups satisfy some
strong (various for different classes) restrictions; cf. e.g. \cite{Osajda} and
references therein. It implies that there are groups that are neither systolic
nor CAT(0) cubical but which act geometrically on our complexes.
In particular, in view of Theorem \ref{fixed_prism} and the fixed point theorems
for systolic and CAT(0) complexes (compare \cites{ChOs,BrHa}), the free product of a
systolic group with a CAT(0) cubical group amalgamated over a finite subgroup always acts geometrically on
a complex from our class. Note however that such a product is often not
systolic neither CAT(0) cubical. Another example with these properties is the
Cartesian product of two systolic but not CAT(0) cubical groups.

\section{Bucolic graphs}
\label{STh12}

In this section, we prove the following characterization of bucolic graphs:

\thgraphs*

The most difficult part of the proof of the theorem are the implications
(ii)$\Rightarrow$(iii) and (iii)$\Rightarrow$(i). The main step in the proof of  (ii)$\Rightarrow$(iii)
is  showing  that all primes of weakly modular  graphs
not containing induced $W_4$ and $W_4^-$ are 2-connected weakly bridged graphs or $K_2$. To prove (iii)$\Rightarrow$(i), we
need to show that weakly bridged graphs are moorable. For locally-finite graphs this was proven in \cite{ChOs}.
The mooring of non-locally-finite weakly bridged graphs is established in Section \ref{moorability}.
Then, we deduce the theorem from the results of \cite{BaCh_wma1,Cha1,Cha2}.

\subsection{Gated closures of triangles}
\label{Striangle}
In this section, we prove that if $G$ is a weakly modular graph not
containing induced 4-wheels $W_4$ and almost 4-wheels $W_4^-$,
then the gated hull of any triangle of $G$ is a weakly bridged graph.
Additionally,
if $G$ does not contain induced 5-wheels $W_5,$ then the gated hull
of a triangle is a bridged graph.

\begin{lemma}\label{p:wheel}
Let $G$ be a graph without induced $W_4,W_4^-$ and satisfying the triangle condition.
Then $G$ does not contain an induced $W^-_n$ for $n > 4.$
\end{lemma}

\begin{proof} Suppose by way of contradiction that $W^-_n$
is an induced subgraph of $G$ and suppose that $G$ does not contain
induced $W^-_k$ for any $3<k<n.$ Let $(x_1,x_2,\ldots,x_n,x_1)$ be the
outer cycle $C$ of $W^-_n$ and consider a vertex $c$ adjacent to all
vertices of $C$ except $x_1.$ We apply the triangle condition to the
triple $x_1,x_2,x_{n-1}$ and find a vertex $a \in N(x_1) \cap N(x_2)
\cap N(x_{n-1})$. Note that if $a \sim c$, then $x_1,x_2,c,x_n,a$
induce $W_4$ if $a$ is adjacent to $x_n$ or $W_4^-$ otherwise. Thus $a\nsim c$.
If $n=5$, then the vertices $x_4, a, x_2, c,
x_3$ induce either a $W_4$ if $x_3$ is adjacent to $a$, or a $W_4^-$
otherwise.  Now, if $n\ge 6$ and if $a$ is not adjacent to $x_3,x_4,
\ldots,x_{n-3}$ or $x_{n-2}$, the subgraph induced by the vertices
$a,x_2,x_3,\ldots, x_{n-1},c$ has an induced subgraph isomorphic to
one of the forbidden induced subgraphs $W^-_k,$ where $k < n$. Thus
$a$ is adjacent to all vertices of $C$ except maybe $x_n$. The
vertices $a,x_3,c,x_{n-1},x_4$ induce $W_4$, if $n=6$, or $W^-_4$
otherwise, a contradiction.
\end{proof}

Let $H$ be an induced
subgraph of a graph $G$. A 2-path $P=(a,v,b)$ of $G$ is
\emph{$H$-fanned} if $a,v,b \in V(H)$ and if there exists an
$(a,b)$-path $P'$ in $H$ not passing via $v$ and such that $v$ is
adjacent to all vertices of $P'$, i.e., $v\sim P'$.  Notice that
$P'$ can be chosen to be an induced path of $G$. A path $P=(x_0,x_1,
\ldots,x_{k-1},x_k)$ of $G$ with $k>2$ is \emph{$H$-fanned} if every
three consecutive vertices $(x_i,x_{i+1},x_{i+2})$ of $P$ form an
$H$-fanned 2-path. When $H$ is clear from the context (typically
when $H = G$), we say that $P$ is \emph{fanned}.  If the endvertices
of a 2-path $P=(a,v,b)$ coincide or are adjacent, then $P$ is fanned.
Here is a simple generalization of this remark (whose immediate proof is left
to the reader).

\begin{lemma}
\label{l:chord}
If $P=(x_0,x_1,\ldots,x_k)$ is a fanned path and the vertices
$x_{i-1}$ and $x_{i+1}$ coincide or are adjacent, then the paths
$P'=(x_0,\ldots,x_{i-2},x_{i+1},x_{i+2}, \ldots,x_k)$ in the first
case and $P''=(x_0,\ldots,x_{i-1},x_{i+1}, \ldots,x_k)$ in the second
case are also fanned.
\end{lemma}

In the remaining auxiliary results of this section, $G$ is a
weakly modular graph without induced
$W_4$ and $W_4^-$. By Lemma \ref{p:wheel}, $G$ does
not contain $W^-_k$ with $k > 3.$

\begin{lemma}\label{l:cycle}
If $C=(x,u,y,v,x)$ is an induced $4$-cycle of $G$, then no simple $2$-path of $C$
is fanned.
\end{lemma}

\begin{proof} Suppose that the simple 2-path  $P=(u,y,v)$ is
fanned.  Let $R=(u,t_1,\ldots,t_m,t_{m+1}=v)$ be a shortest
$(u,v)$-path such that $y\sim R$ (such a path exists because $P$ is
fanned).  Necessarily, $R$ is an induced path of $G$.  Since $C$ is
induced, $m \geq 1$ and $t_i \neq x$ for all $i \in \{1, \ldots,m\}.$
If $t_1$ is adjacent to $x$, then the vertices $x,u,y,v,t_1$
induce $W_4$ if $t_1$ is adjacent to $v$, or $W^-_4$
otherwise. Suppose now that $t_1$ is not adjacent to $x$ and let $i
\geq 2$ be the smallest index such that $t_i$ is adjacent to $x$.
Since $R$ is a shortest path, the cycle $(x,u,t_1,\ldots,t_i,x)$ is
induced. Thus the vertices $x,u,t_1,\ldots,t_i,y$ induce a forbidden
$W^-_{i+2}$.
\end{proof}

Let $v$ be a common neighbor of vertices $a$ and $b$ of $G.$ For an
$(a,b)$-path $P$, we denote by $D_v(P)$ the distance sum $D_v(P):=\sum_{x\in P}
d(x,v).$

\begin{lemma}
\label{l:covered}
Let $P=(a=x_0, x_1,\ldots,x_m=b)$ be a fanned $(a,b)$-path not
containing $v,$ let $k = \max \{ d(x_i,v): x_i \in P\}$ and $j$
be the smallest index so that $d(x_{j},v)=k$. If $k\ge 2$, and $j
\notin \{0,m\}$ then
\begin{itemize}
\item[(1)] either $x_{j-1}=x_{j+1}$ and the path $P'=(x_0, \ldots,x_{j-2},x_{j+1},
x_{j+2},\ldots,x_m)$ is fanned,
\item[(2)] either $x_{j-1}\sim x_{j+1}$ and the path
$P''=(x_0,\ldots,x_{j-1},x_{j+1},\ldots,x_m)$ is fanned,
\item[(3)] or there exists a vertex $y$ such that $d(y,v)=k-1$ and the path $P''' =
(x_0,\ldots,x_{j-1}, y, x_{j+1},\ldots,x_m)$ is fanned.
\end{itemize}
\end{lemma}

\begin{proof}
If $x_{j-1}=x_{j+1}$ or $x_{j-1}\sim x_{j+1},$ then Lemma
\ref{l:chord} implies that the paths $P'$ and $P''$ are fanned. So,
suppose that $x_{j-1}$ and $x_{j+1}$ are different and
non-adjacent. Note that $d(x_{j-1},v)=k-1$ and $k-1\le d(x_{j+1},v)\le
k$. If $d(x_{j+1},v)=k-1$, then we can use the quadrangle condition
for vertices $v,x_{j-1},x_j$ and $x_{j+1}$ and find a vertex $z\in
N(x_{j-1})\cap N(x_{j+1})$ such that $d(v,z)=k-2$ ($z=v$ if $k=2$).
Since $z$ and $x_j$ are not adjacent, the $4$-cycle
$(z,x_{j-1},x_j,x_{j+1},z)$ is induced. Since $P$ is fanned, the
2-path $(x_{j-1},x_j,x_{j+1})$ is fanned as well, contradicting Lemma
\ref{l:cycle}.

Thus $d(x_{j+1},v)=k$. Applying  the triangle condition to the
triple $v,x_j,x_{j+1}$,
we can find  a common neighbor $y$ of $x_j$ and $x_{j+1}$  with
$d(v,y)=k-1$.  Note that $y\neq x_{j-1}$ since $x_{j-1}\nsim x_{j+1}.$
Assume first $x_{j-1}\nsim y$. Then
we can apply the quadrangle condition to the vertices
$x_{j-1},x_j,y,v$, and find a vertex $z \in N(x_{j-1}) \cap N(y)$ with
$d(z,v)=k-2$ ($z=v$ if $k=2$).  Clearly, $z$ is not adjacent to $x_j$ and
$x_{j+1}.$ Hence, the cycle
$(x_{j-1},x_j,y,z,x_{j-1})$ is
induced. Since the $2$-path $(x_{j-1},x_{j},x_{j+1})$ is fanned, there exists a
$(x_{j-1},x_{j+1})$-path
$Q_0$ not containing $x_j$ such that $x_j\sim Q_0.$ As a consequence, $(Q_0,y)$
is a
$(x_{j-1},y)$-path of $G$  not passing via $x_j$ whose all vertices are adjacent
to $x_j$. Therefore the $2$-path
$(x_{j-1},x_j,y)$ of the induced $4$-cycle $(x_{j-1},x_j,y,z,x_{j-1})$ is fanned,
contradicting Lemma \ref{l:cycle}.
This implies that $x_{j-1}$ must be adjacent to $y$.
Then $P'''=(x_0,\ldots,x_{j-1},y,x_{j+1},\ldots,x_m)$ is a path
of $G$. We claim that $P'''$ is fanned. Indeed, all  2-paths of $P'''$, except
the three
consecutive 2-paths $(x_{j-2},x_{j-1},y),(y,x_{j+1},x_{j+2}),
(x_{j-1},y,x_{j+1}),$ are also 2-paths of $P$, hence they are
fanned. The 2-path
$(x_{j-1},y, x_{j+1})$ is fanned because $y$ is adjacent to all
vertices of the path $(x_{j-1},x_j,x_{j+1}).$ Since the 2-path
$(x_{j},x_{j+1},x_{j+2})$
is fanned, there is an
$(x_{j},x_{j+2})$-path $R$ such that $x_{j+1}\sim R$.  Then all vertices of the
$(y,x_{j+2})$-path $(y,R)$ are
adjacent to $x_{j+1}$, whence the 2-path $(y,x_{j+1},x_{j+2})$ is fanned.
Analogously, one can show that the
2-path $(x_{j-2},x_{j-1},y)$ is fanned, showing that $P'''$ is fanned.
\end{proof}

From the proof of Lemma \ref{l:covered}, since each of
$D_v(P'),D_v(P''),D_v(P''')$ is smaller than $D_v(P),$ we conclude
that if $v \sim a,b$ and $P$ is a fanned $(a,b)$-path not containing
$v$ with minimal distance sum $D_v(P),$ then $k=1.$ Therefore, we
obtain the following result:

\begin{corollary}\label{fanned-avoid}
If $v \sim a,b$ and if $P$ is a fanned $(a,b)$-path avoiding $v$ with
minimal distance sum $D_v(P),$ then $v\sim P$.
\end{corollary}

Let $\triangleleft$ be a well-order on $V(G)$.  Let $T=\{a_0,b_0,c_0\}$
be a triangle in $G$. We define a subgraph $K$ of
$G$ by (possibly transfinite) induction as follows. Let $H_0, H_1,
H_2$ be the subgraphs respectively induced by $\{a_0\}, \{a_0, b_0\}$
and $ \{a_0, b_0,c_0\}$. Given an ordinal $\alpha$, assume that for
every $\beta < \alpha$, we have defined $H_\beta$, and let
$H_{<\alpha}$ be the subgraph induced by $\bigcup_{\beta < \alpha}
V(H_\beta)$.  Let
$$X=\{v \in V(G)\setminus V(H_{< \alpha}):
\mbox{ there exist } x,y \in V(H_{<\alpha}) \mbox{ such that } v \sim x,y\}.$$
  If $X$ is nonempty, then let $v$ be the least element of $(X,\triangleleft)$
and define $H_{\alpha}$ to be the subgraph of $G$ induced by $V(H_{< \alpha}\cup\{v\})$.
Otherwise, if $X$ is empty, then set $K:=H_{< \alpha}$.

\begin{lemma}
\label{p:gate1}
For any ordinal $\alpha$, $H_{\alpha}$ is 2-connected and any 2-path
of $H_\alpha$ is $K$-fanned. Moreover, $K$ is 2-connected and any 2-path
of $K$ is $K$-fanned.

\end{lemma}

\begin{proof}
We proceed by induction on $\alpha$.  Clearly, $H_0, H_1, H_2=T$
fulfil these properties. Assume by induction hypothesis that for
every $\beta < \alpha$, $H_\beta$ is $2$-connected and that any 2-path
of $H_{\beta}$ is $K$-fanned.

We first show that $H_{< \alpha}$ is 2-connected and that any 2-path
of $H_{< \alpha}$ is $K$-fanned. Consider any three vertices $a, b, u
\in V(H_{<\alpha})$. There exists $\beta<\alpha$ such that $a,b,u
\in V(H_\beta)$. By the induction hypothesis, there exists a path from
$a$ to $b$ in $H_\beta\setminus \{u\}$. Since $H_\beta\setminus \{u\}$
is a subgraph of $H_{< \alpha}\setminus \{u\}$, $a$ is not
disconnected from $b$ in $H_{< \alpha}\setminus \{u\}$, and thus $H_{<
  \alpha}$ is 2-connected.  For every 2-path $(a,b,c)$ in $H_{<
  \alpha}$, there exists $\beta < \alpha$ such that $a,b,c \in
V(H_\beta)$. By the induction hypothesis, the 2-path $(a,b,c)$ is
$K$-fanned.

If $K = H_{< \alpha}$, we are done. Otherwise, let $v$ be the unique
vertex of $V(G)$ such that $V(H_\alpha) = V(H_{<\alpha}) \cup \{v\}$.
By the definition of $H_{\alpha}$, $v$ has at least two neighbors
$x,y$ in $H_{<\alpha}$.

Suppose that $H_\alpha$ is not 2-connected. Consider three distinct
vertices $a,b,u \in V(H_\alpha)$.  If $a, b \in V(H_{<\alpha}) $, we
know there exists a path from $a$ to $b$ in $H_{<\alpha} \setminus
\{u\}$. Without loss of generality, assume now that $b = v$ and $u
\neq x$. We know that there exists a path from $a$ to $x$ in
$H_{<\alpha} \setminus \{u\}$ and consequently, there exists a path
from $a$ to $b = v$ in $H_{\alpha} \setminus \{u\}$ since $x \sim
v$. Consequently, for every $a, b, u \in V(H_\alpha)$, $u$ does not
disconnect $a$ from $b$, i.e., $H_\alpha$ is 2-connected.

We will prove that any 2-path of $H_{\alpha}$ is $K$-fanned.  It
suffices to consider the 2-paths $Q$ of $H_{\alpha}$ that contain $v$,
since all other 2-paths lie in $H_{<\alpha}$ and are $K$-fanned.

\smallskip
\noindent{\em Case 1.} $Q=(a,v,c)$.

Since $H_{< \alpha}$ is connected and $a,c \in V(H_{< \alpha})$, there
exists an $(a,c)$-path $R$ in $H_{<\alpha}$. Since any 2-path of $H_{<
  \alpha}$ is $K$-fanned by induction hypothesis, $R$ itself is
$K$-fanned.  As $H_{< \alpha}$ is a subgraph of $K$, $R$ belongs to
$K$.  Among all $K$-fanned $(a,c)$-paths belonging to $K$ and avoiding
$v$, let $P=(a =x_0, x_1,\ldots,x_m=c)$ be chosen in such a way that
the distance sum $D_v(P)=\sum_{x_i\in P}d(v,x_i)$ is minimized (note
that $P$ does not necessarily belong to $H_\alpha$).  By Corollary
\ref{fanned-avoid}, $v \sim P$ and thus the 2-path $Q$ is $K$-fanned.

\smallskip
\noindent{\em Case 2.} $Q=(c,b,v)$.

If $c$ and $v$ coincide or are adjacent, then $Q$ is trivially
fanned. Thus we may assume that $c\ne v$, and $c\nsim v$. Since $v$
has at least two neighbors in $H_{< \alpha}$, there exists a vertex $a
\in H_{< \alpha}$ adjacent to $v$ and different from $b$.  Since
$H_{<\alpha}$ is 2-connected and $a,c \in H_{<\alpha}$, there exists
an $(a,c)$-path $P_0$ in $H_{< \alpha}$ that avoids $b$. The paths
$P_0$ and $(P_0,b)$ are $K$-fanned because all their 2-paths are
fanned by the induction hypothesis. Hence, there exists at least one
$K$-fanned $(a,b)$-path $(P_0,b)$ that passes via $c$, avoids $v$, and
all vertices of $P_0$ are different from $b$.  Among all such
$(a,b)$-paths $(P_0,b)$ of $K$ (i.e., that pass $c$, avoid $v$, the
vertices of $P_0$ are different from $b,$ and are $K$-fanned), let
$P=(a=x_0, x_1, \ldots, x_m,x_{m+1}=c,b)$ be chosen in such a way that
$D_v(P)$ is minimized.  Since $v$ and $x_{m+1}=c$ are different and
not adjacent, $k = \max \{d_G(x_i,v): x_i \in P\}\ge 2$.  Let $j$ be
the smallest index such that $d(x_{j},v)=k$.

First suppose that $j\ne m+1$. By Lemma \ref{l:covered}, the vertices
$a$ and $b$ can be connected by one of the paths $P',P'',P'''$ derived
from $P$. These paths are $K$-fanned, contain the vertex $c$, avoid
the vertex $v$, and all three have smaller distance sums than $P$. In
case of $P'$ and $P''$ we obtain a contradiction with the minimality
choice of $P$. Analogously, in case of $P'''$ we obtain the same
contradiction except if the vertex $y$ coincides with $b$, i.e., $b$
is adjacent to the vertices $x_{j-1},x_j,$ and $x_{j+1}$. In this case,
$d(x_j,v) = 2$ and $x_{j-1} \sim v$. Consider the
$2$-path $(c,b,x_{j+1})$. By construction, the path
$R=(x_{m+1} =c,x_m, \ldots,x_{j+2}, x_{j+1})$ is $K$-fanned and
avoids $b$. Applying Lemma~\ref{l:covered} and Corollary \ref{fanned-avoid}
with $b$ and $R$, there
exists a $K$-fanned $(c,x_{j+1})$-path $R'$ avoiding $b$ such that $b
\sim R'$. Consequently, there is a path $(R',x_j,x_{j-1},v)$ in $K$
from $c$ to $v$ in the neighborhood of $b$ and thus $(c,b,v)$ is
$K$-fanned.

 Now suppose that $j=m+1,$ i.e., $v$ is adjacent to all vertices of
 $P$ except $x_{m+1}=c.$ From the choice of $P$ we conclude that $b\ne
 x_m.$ If $b\nsim x_m,$ then $C=(v,x_m,c,b,v)$ is an induced
 $4$-cycle. Since the $2$-path $(b,c,x_m)$ is $K$-fanned and simple, we obtain a
 contradiction with Lemma \ref{l:cycle}. Finally, if $b$ is adjacent
 to $x_m,$ then the (simple) 2-path $(c,b,v)$ is $K$-fanned because $c$ and $v$
 are connected in $K$ by the (simple) 2-path $(c,x_m,v)$ and $x_m$ is adjacent
 to $b$.
\end{proof}

\begin{lemma}
\label{p:gate2}
For any ordinal $\alpha$, $H_{<\alpha}$ and $H_{\alpha}$ do not contain induced
$4$-cycles.
\end{lemma}

\begin{proof} Again we proceed by induction on $\alpha$.

Suppose by induction hypothesis that for every $\beta < \alpha$,
$H_{\beta}$ does not contain induced $4$-cycles. If there exists a
$4$-cycle $(a,b,c,d,a)$ in $H_{<\alpha}$, there exists $\beta < \alpha$
such that $a, b, c, d \in V(H_\beta)$. Since $H_\beta$ is an induced
subgraph of $H_{<\alpha}$, $(a,b,c,d,a)$ is an induced $4$-cycle of
$H_\beta$, contradicting the induction hypothesis.

If $K = H_{<\alpha}$, we are done. Otherwise, let $v$ be the unique
vertex of $V(G)$ such that $V(H_\alpha) = V(H_{<\alpha}) \cup \{v\}$.
Suppose by way of contradiction that $H_{\alpha}$ contains an induced
$4$-cycle $C$. Then necessarily $v$ belongs to $C$. Let $C=(v,a,b,c,v).$
Since by Lemma \ref{p:gate1} the $2$-paths of $H_{\alpha}$ are $K$-fanned,
the simple $2$-path $(a,b,c)$ of $C$ is fanned and we obtain a
contradiction with Lemma \ref{l:cycle}. Consequently, $H_\alpha$ does
not contain induced $4$-cycles.
\end{proof}

\begin{lemma} \label{p:gate3}
$K$ is the gated hull of $T$ in $G$.
\end{lemma}

\begin{proof}
Let $A$ be the gated hull of $T$. First we prove that all vertices
of $K$ belong to $A$. Suppose by way of contradiction that
$K\setminus A \neq \emptyset.$ From all vertices in $K \setminus A$
we choose $v$ with smallest $\alpha$, such that $v \notin H_{<\alpha}$,
$v \in H_\alpha$, i.e., all vertices from $H_{<\alpha}$ are contained
in $A$. Since $v \in H_\alpha,$ it has at least two neighbors in $H_{< \alpha}$
and thus in $A$. Therefore there is no gate of $v$ in $A$, a contradiction.

On the other hand, since $G$ is weakly modular, $K$ is gated if
and only if for every $x,y\in K$ at distance at most $2$, any common
neighbor $v$ of $x$ and $y$ also belongs to $K$ \cite{BaCh_weak,Ch_triangle}.
This is obviously true by the definition of $K$.
\end{proof}

Summarizing, we obtain the main result of this subsection.

\begin{proposition} \label{gated_hull} Let $G$ be a locally-finite weakly
modular graph not containing induced
$W_4$ and $W_4^-$. Then the gated hull of any triangle $T$ of $G$ is a
2-connected weakly bridged graph. Additionally, if
$G$ does not contain induced $W_5,$ then the gated hull of  $T$ is a 2-connected
bridged
graph.
\end{proposition}

\begin{proof} By Lemma \ref{p:gate3}, the gated hull of $T$ is the 2-connected
subgraph $K$ of $G$ constructed by our procedure.
Since $K$ is a convex subgraph of a weakly modular graph $G,$ $K$ itself is a
weakly modular graph. By Lemma \ref{p:gate2}, the
graph $K$ does not contain induced $4$-cycles, thus $K$ is weakly bridged
by \cite[Theorem 3.1(iv)]{ChOs}. If, additionally,
$G$ does not contain $5$-wheels, then $G$ does not contain induced $5$-cycles
because in a weakly bridged graph any induced $5$-cycle is included in
a $5$-wheel. Then $K$ is a weakly modular graph without induced $4$- and $5$-cycles,
thus $K$ is bridged.
\end{proof}

\subsection{Proof of Theorem \ref{theorem1}}
\label{STh1}

We first prove the implication (i)$\Rightarrow$(ii). First, bridged
and weakly bridged graphs are weakly modular.  Weakly bridged graphs
do not contain induced $K_{2,3},$ $W_4,$ and $W^-_4$ because they do
not contain induced 4-cycles. Bridged graphs additionally do not
contain induced $W_5.$ Weakly modular graphs are closed by taking
(weak) Cartesian products (this holds also when there are infinite
number of factors in weak Cartesian products, since the distances
between vertices in a weak Cartesian product are finite).  If a (weak)
Cartesian product $\Box_{i\in I} H_i$ contains an induced
$K_{2,3},W_4,W_5$ or $W_4^-,$ then necessarily this graph occurs in
one of the factors $H_i$. This follows from the fact that in a
product, each triangle comes from one factor, and two opposite edges
of a square must come from the same factor.  As a consequence,
Cartesian products $H=\Box_{i\in I} H_i$ of weakly bridged graphs do
not contain induced $K_{2,3},W_4,$ and $W_4^-.$ Analogously, Cartesian
products $H=\Box_{i\in I} H_i$ of bridged graphs do not contain
induced $K_{2,3},W_4,W^-_4,$ and $W_5.$ If $G$ is a retract of $H,$
then $G$ is an isometric subgraph of $H$, and therefore $G$ does not
contain induced $K_{2,3},W_4,W_4^-$ in the first case and induced
$K_{2,3},W_4,W_4^-$ and $W_5$ in the second case.  It remains to
notice that the triangle and quadrangle conditions are preserved by
retractions, thus $G$ is a weakly modular graph, establishing that
(i)$\Rightarrow$(ii).

Now suppose that $G$ is a weakly modular graph satisfying the
condition (ii) of Theorem \ref{theorem1}. Then $G$
is a pre-median graph. By \cite[Theorem 4.13]{Cha1}, any pre-median
graph is fiber-complemented. By \cite[Lemma 4.8]{Cha1}, this implies
that any gated subgraph $H$ of $G$ is elementary if and only if it is
prime.
Note that the gated hull of any edge in $G$ is either the edge itself,
or it is included in a triangle by weak modularity, and by
Proposition~\ref{gated_hull} we find that the gated hull of this edge is
a 2-connected (weakly) bridged graph.  Hence every elementary (= prime) gated subgraph is a 2-connected (weakly) bridged graph or an edge. This
establishes the implication (ii)$\Rightarrow$(iii).

To prove the implication (iii)$\Rightarrow$(i), we will use
\cite[Theorem 3.2.1]{Cha2} and \cite[Theorem 5.1]{ChOs}. By Chastand
\cite[Theorem 3.2.1]{Cha2}, any fiber-complemented graph $G$ whose
primes are moorable graphs is a retract of the Cartesian product of
its primes. Note that elementary gated subgraphs of $G$, enjoying (iii), are
edges and 2-connected weakly bridged graphs. In~\cite[Theorem~5.1]{ChOs},
it is shown that locally-finite weakly bridged graphs are
moorable. Proposition~\ref{prop-moorable} in Section \ref{moorability} extends this result to
non-locally-finite graphs. Thus, by \cite[Theorem 3.2.1]{Cha2} $G$
is a retract of the
Cartesian product of its primes, establishing the implication
(iii)$\Rightarrow$(i) of Theorem \ref{theorem1}.

Now, for finite graphs we show that (iv)$\iff$(ii). As noticed above,
bridged and weakly bridged graphs are weakly modular and do not
contain induced $K_{2,3}$, $W_4,$ and $W^-_4$. Bridged graphs
additionally do not contain induced $W_5$.  Weakly modular graphs are
closed by Cartesian products and gated amalgams. Moreover, if $G$ is
the Cartesian product or the gated amalgam of two graphs $G_1$ and
$G_2$, then $G$ contains an induced $K_{2,3}$ (respectively,  $W_4, W_4^-,
W_5$) if and only if $G_1$ or $G_2$ does. Therefore
(iv)$\Rightarrow$(ii).  Conversely, suppose that $G$ is a finite
bucolic (respectively, strongly bucolic) graph.  Then $G$ is a pre-median
graph. By \cite[Theorem 4.13]{Cha1}, any pre-median graph is
fiber-complemented. Then according to \cite[Theorem 5.4]{Cha1}, $G$
can be obtained from Cartesian products of elementary (=prime) graphs
by a sequence of gated amalgamations.  By Proposition
\ref{gated_hull}, any elementary graph of $G$ is either an edge or a
2-connected weakly bridged graph (respectively, a 2-connected bridged
graph). Thus the implication (ii)$\Rightarrow$(iv) in Theorem
\ref{theorem1} holds.  This concludes the proof of Theorem
\ref{theorem1}.

\section{Bucolic complexes and their skeleta}
\label{proof_triangle-square-complex_bridged}

In this section, we prove the following local-to-global characterization of bucolic complexes  via properties of their
1- and 2-skeleta:

\thcomplexes*

As an immediate corollary we obtain the following analogous characterization of strongly bucolic complexes:

\begin{corollary}\label{cor-strong-buc-complexes} For a prism complex $\bX$, the following conditions are equivalent:

\begin{itemize}
\item[(i)] $\bX$ is a strongly bucolic complex;
\item[(ii)] the $2$-skeleton $\bX^{(2)}$ of $\bX$ is a connected and simply connected triangle-square flag
complex satisfying the strong-wheel, the 3-cube, and the 3-prism conditions;
\item[(iii)] the $1$-skeleton $G(\bX)$ of $\bX$ is a connected weakly modular graph not containing induced subgraphs of the form
  $K_{2,3},$ $W_4,$ $W_4^-,$ and $W_5$, i.e., $G(\bX)$ is a
  strongly bucolic graph not containing infinite hypercubes as induced
  subgraphs;
\end{itemize}
Moreover, if $\bf X$ is a connected flag prism complex satisfying the strong-wheel, the cube, and the prism conditions,
 then the universal cover $\widetilde{\bf X}$ of $\bf X$ is strongly bucolic.
\end{corollary}

\subsection{Auxiliary results.} We start this section with several auxiliary properties of triangle-square
flag complexes occurring in condition (ii) of Theorem \ref{theorem0}. Throughout this and the next subsections,
we will denote such triangle-square complexes by $\bX$, assume that they are connected, and
use the shorthand $G:=G({\bX})$ for the 1-skeleton of $\bX$.
We denote by ${\bX}(C_3)$ and ${\bX}(C_4)$
the triangle-square complex consisting of a single triangle and a single square,
respectively.
Let ${\bX}(H)={\bX}(C_3+C_4)$ be the complex consisting of a
triangle and a square sharing one edge; its graph is the house $H$ and
with some abuse of notation, we call the complex itself a {\it
  house}. The {\it twin-house} ${\bX}(2H)$ is the complex consisting
of two triangles and two squares, which can be viewed as two houses
glued along two incident edges or as a domino and a kite glued along
two incident edges (for an illustration, see Fig.~\ref{fig-houses}, left).
Let also ${\bf  X}(W_k)$ and ${\bX}(W^-_k)$ be the triangle-square complexes
whose
underlying graphs are $W_k$ and $W^-_k$: the first consists of $k$
triangles and the second consists of $k-2$ triangles and one square.
The complex ${\bX}(CW_3)$ consists of three squares sharing a vertex
and pairwise sharing edges (its graph is the cogwheel $CW_3$). The
{\it triangular prism} ${\bX}(Pr)={\bX}(C_3\times K_2)$ consists
of the surface complex of the 3-dimensional triangular prism (two
disjoint triangles and three squares pairwise sharing an edge). The
{\it double prism} ${\bX}(2Pr)$ consists of two prisms ${\bX}(Pr)$
sharing a square (See Fig.~\ref{fig-houses}, left).
Finally, the {\it double-house} ${\bX}(H+C_4)={\bX}(2C_4+C_3)$ is
the complex consisting of two squares and a triangle, which can be
viewed as a house plus a square sharing with the house two incident
edges, one from the square and another from the triangle (see
Fig.~\ref{fig-houses}, right). In the following results, we
use the notation $G=G(\bX)$.

\begin{figure}[t]
\begin{center}
\includegraphics{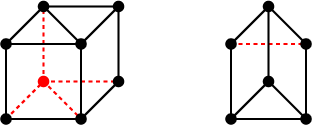}%
\end{center}
\caption{On the left, a twin-house (in black)  included
 in a double prism (Lemma~\ref{house+K2,3}). On the right, a
 double house (in black) included in a prism
 (Lemma~\ref{no-double-house}).}
\label{fig-houses}
\end{figure}

\begin{lemma} \label{K_2,3+W^-_4} If ${\bX}$ is a triangle-square flag complex,
then its 1-skeleton $G$ does not contain induced $K_{2,3}$ and $W^-_4$.
\end{lemma}

\begin{proof} If $G$ contains $K_{2,3}$ or $W^-_4,$ then, since $\bX$ is a
flag complex, we will obtain two squares intersecting in two edges, which is
impossible.
\end{proof}

\begin{lemma} \label{house+K2,3} If ${\bX}$ satisfies the 3-prism
  condition, then any twin-house ${\bX}(2H)$ of ${\bX}$ is
  included in ${\bX}$ in a double prism ${\bX}(2Pr)$.
\end{lemma}

\begin{proof} Let $u,v,w,x_1,x_2$ be the vertices of one house and
  $u,v,w,y_1,y_2$ be the vertices of another house, where the edge
  $uv$ is common to the two squares $uvx_2x_1$ and $uvy_2y_1$, and
  where the edge $vw$ is common to the two triangles $vwx_2$ and
  $vwy_2$. By the 3-prism condition, there exists a vertex $a$ adjacent in
  $G$ to $x_1, u, w$ that is not adjacent to
  $x_2,v$. Analogously, there exists a vertex $b$ adjacent to $u, y_1,
  w$ that is not adjacent to $y_2,v$. If $a \neq b$, the graph induced
  by $a,b,u,v,w$ is either $K_{2,3}$ if $a \nsim b$, or $W_4^-$
  otherwise; in both cases, we get a contradiction with
  Lemma~\ref{K_2,3+W^-_4}.  Thus $a=b$, and since $a \nsim v,x_2,y_2$,
  the vertices $a,u,v,w,x_1,x_2,y_1,y_2$ induce a double prism.
\end{proof}

\begin{lemma} \label{no-double-house} If ${\bX}$ satisfies the 3-prism
  condition, then any double-house $\bX(H+C_4)$ in $\bX$ is included
  in a prism $\bX(Pr)$, i.e., $G$ does not contain an induced
  double-house $H+C_4$.
\end{lemma}

\begin{proof} Suppose by contradiction that $G$ contains
  an induced double-house having $x,y,u,v,w,z$ as the set of vertices,
  where $uvw$ is a triangle and $xyvu$ and $xuwz$ are two squares of
  this house. By 3-prism condition, there exists a vertex $a$ different
  from $z$ (since $y \nsim z$) that is adjacent to $x,y,w$ and that
  is not adjacent to $u,v$.  Thus, the vertices $z,a,w,u,x$ induce
  either $K_{2,3}$ if $a \nsim z$ or $W_4^-$ otherwise. In both
  cases, we get a contradiction with Lemma~\ref{K_2,3+W^-_4}.
\end{proof}

\begin{lemma} \label{W__k} If ${\bX}$ satisfies the 3-prism condition and
does not contain ${\bX}(W_4)$, then ${\bX}$ does not contain
${\bX}(W^-_k)$ for any $k\ge 5.$
\end{lemma}

\begin{proof} Suppose by way of contradiction that ${\bX}$ contains
  ${\bX}(W^-_k)$, where $k$ is the smallest value for which this
  subcomplex exists. Since, by Lemma \ref{K_2,3+W^-_4}, $G$ does not contain
$W^-_4,$
  necessarily $k\ge 5$. Denote the vertices of $\bX(W_k^-)$ by
  $q,x_1,x_2,\ldots,x_k$ where $x_1,x_2, \ldots, x_k$ induce a cycle
  and where $q$ is adjacent to $x_1, \ldots, x _{k-1}$ but not to
  $x_k$. By the 3-prism condition applied to the house induced by
  $q,x_{k-1},x_k,x_1,x_2$, there exists $p$ in $G$ such that $p
  \sim x_{k-1},x_k,x_2$ and $p \nsim q,x_1$. If $p \sim x_3$, then the
  vertices $x_3,p,x_2,q,x_{k-1}$ induce $W_4$ if $x_3 \sim x_{k-1}$
  (i.e., if $k=5$), or $W_4^-$ otherwise; in both cases, we get a
  contradiction. Thus $p \nsim x_3$. Let $j$ be the smallest index
  greater than $3$ such that $p\sim x_j$. Since $p \sim x_{k-1}$, $j$
  is well-defined. But then, the vertices $q,p,x_2, \ldots, x_j$
  induce $W_j^-$ with $j<k$, contradicting the choice of $k$.
\end{proof}

\begin{lemma} \label{simplyconnected} Let ${\bf X}$ be a
  triangle-square flag complex such that $G({\bf X})$ satisfies the
  triangle and the quadrangle conditions TC$(v)$ and QC$(v)$, for some
  basepoint $v$. Then ${\bf X}$ is simply connected.
\end{lemma}

\begin{proof}
A \emph{loop} in ${\bf X}$ is a sequence $(w_1,w_2,...,w_k,w_1)$ of
vertices of ${\bf X}$ consecutively joined by edges. To prove the
lemma it is enough to show that every loop in ${\bf X}$ can be freely
homotoped to a constant loop $v$.  By contradiction, let $A$ be the
set of loops in $G({\bf X})$, which are not freely homotopic to $v$,
and assume that $A$ is non-empty.  For a loop $C\in A$ let $r(C)$
denote the maximal distance $d(w,v)$ of a vertex $w\in C$ to the
basepoint $v$.  Clearly $r(C)\geq 2$ for any loop $C\in A$ (otherwise
$C$ would be null-homotopic). Let $B\subseteq A$ be the set of loops
$C$ with minimal $r(C)$ among loops in $A$. Let $r:= r(C)$ for some
$C\in B$. Let $D\subseteq B$ be the set of loops having minimal number
$e$ of edges in the $r$-sphere around $v$, i.e. with both endpoints at
distance $r$ from $v$. Further, let $E\subseteq D$ be the set of loops
with the minimal number $m$ of vertices at distance $r$ from $v$.

Consider a loop $C=(w_1,w_2,...,w_k,w_1)\in E$.
We  can assume without loss of generality that $d(w_2,v)=r$.
We distinguish two cases corresponding to the
triangle or quadrangle condition that we apply to them.
\medskip

\noindent \emph{Case 1: $d(w_1,v)=r$ or $d(w_3,v)=r$.}  Assume without
loss of generality that $d(w_1,v)=r$. Then, by the triangle condition
$TC(v)$, there exists a vertex $w\sim w_1,w_2$ with $d(w,v)=r-1$.
Observe that the loop $C'=(w_1,w,w_2,...,w_k,w_1)$ belongs to $B$ --
in ${\bf X}$ it is freely homotopic to $C$ by a homotopy going through
the triangle $ww_1w_2$. The number of edges of $C'$ lying on the
$r$-sphere around $v$ is less than $e$ (we removed the edge
$w_1w_2$). This contradicts the choice of the number $e$.
\medskip

\noindent \emph{Case 2: $d(w_1,v)=d(w_3,v)=r-1$.}
By the quadrangle condition $QC(v)$, there exists a vertex $w\sim w_1,w_3$ with $d(w,v)=r-2$.
Again, the loop $C'=(w_1,w,w_3,...,w_k,w_1)$ is freely homotopic to $C$ (via the square $w_1w_2w_3w$). Thus $C'$
belongs to $D$ and the number of its vertices at distance $r$ from $v$ is equal to $m-1$. This contradicts
the choice of the number $m$.
\medskip

In both cases above we get contradiction. It follows that the set $A$ is empty and hence the lemma is proved.
\end{proof}

\subsection{Proof of (i)$\Rightarrow$(ii).}
Since the cube condition implies the 3-cube condition and the prism condition
implies the 3-prism condition, if $\bX$ is a bucolic complex, then its 2-skeleton
$\bX^{(2)}$ satisfies the condition (ii), thus (i)$\Rightarrow$(ii).

\subsection{Proof of (ii)$\Rightarrow$(iii).}
To prove the implication (ii)$\Rightarrow$(iii) of Theorem \ref{theorem0},
from now on, we suppose that  ${\bX}$ is a connected (but not necessarily simply connected)
triangle-square flag complex
satisfying the wheel, the 3-prism, and the 3-cube conditions.
The most difficult part of the proof is to show that the 1-skeleton of  ${\bX}$
is weakly modular. To show this, we closely follow the proof method of a
local-to-global characterization of weakly systolic complexes provided by
Osajda \cite{Osajda} using the level-by-level construction of the universal
cover of ${\bX}$.

\subsubsection{Structure of the construction.}
\label{s:struct}

We construct the universal cover $\widetilde{{\bX}}$ of ${\bX}$ as an increasing
union $\bigcup_{i\ge 1} \widetilde{{\bX}}_i$ of triangle-square
complexes. The complexes $\widetilde{{\bX}}_i$ are in fact spanned by
concentric combinatorial balls $\widetilde{B}_i$ in $\widetilde{{\bX}}$.
The covering map $f$ is then the union $\bigcup_{i\ge 1} f_i,$ where
$f_i: \widetilde{{\bX}}_i\rightarrow {\bX}$ is a locally injective
cellular map such that $f_i|_{\widetilde{{\bX}}_j}=f_j$, for every $j\le i$. We
denote by $\widetilde{G}_i=G(\widetilde{{\bX}}_i)$ the underlying graph of
$\widetilde{{\bX}}_i$. We denote by $\tS_i$ the set of vertices
$\tB_i\setminus \tB_{i-1}$.

Pick any vertex $v$ of ${\bX}$ as the basepoint. Define $\widetilde{B}_0=\{
\widetilde{v}\}:=\{ v\}, \widetilde{B}_1:=B_1(v,G),$ and $f_1:=$Id$_{B_1(v,G)}.$
Let $\widetilde{{\bX}}_1$  be the triangle-square complex spanned by $B_1(v,G).$
Assume that, for $i\geq 1$,  we have constructed the vertex sets
$\widetilde{B}_1,\ldots,\widetilde{B}_i,$ and we have defined
the triangle-square complexes $\widetilde{{\bX}}_1\subseteq \cdots\subseteq \widetilde{{\bX}}_i$
(for any $1\le j<k\le i$ we have an identification map $\widetilde{{\bX}}_j\rightarrow \widetilde{{\bX}}_{k}$)
and the corresponding cellular maps $f_1,\ldots,f_i$ from
$\widetilde{{\bX}}_1,\ldots,\widetilde{{\bX}}_i,$ respectively,  to ${\bX}$ so
that the graph $\tG_i=G(\widetilde{\bX}_i)$ and the complex $\widetilde{\bX}_i$
satisfy the following conditions:

\begin{enumerate}[(A{$_i$})]
\item[(P$_i$)]
$B_j(\tv,\tG_i)=\widetilde{B}_j$ for any $j\le i$;
\item[(Q$_i$)]
$\widetilde{G}_i$ is weakly modular with respect to $\widetilde{v}$ (i.e.,
$\widetilde{G}_i$ satisfies the conditions TC($\widetilde{v}$) and
QC($\widetilde{v}$));
\item[(R$_i$)]
for any $\widetilde{u}\in \widetilde{B}_{i-1},$ $f_i$
defines an isomorphism between the subgraph of $\tG_i$ induced by
$B_1(\widetilde{u},\tG_i)$ and the subgraph of $G$ induced by
$B_1(f_i(\widetilde{u}),G)$;
\item[(S$_i$)]
for any $\widetilde{w},\widetilde{w}'\in \widetilde{B}_{i-1}$ such that the
vertices $w=f_i(\widetilde{w}),w'=f_i(\widetilde{w}')$ belong to a square
$ww'uu'$ of ${\bX}$, there exist $\widetilde{u},\widetilde{u}'\in
\widetilde{B}_i$ such that $f_i(\widetilde{u})=u, f_i(\widetilde{u}')=u'$ and
$\widetilde{w}\widetilde{w}'\widetilde{u}\widetilde{u}'$ is a square of
$\widetilde{{\bf
    X}}_i$.
\item[(T$_i$)]
 for any $\widetilde{w}\in \widetilde{S}_i:=\widetilde{B}_i\setminus
  \widetilde{B}_{i-1},$ $f_i$
  defines an isomorphism between the subgraphs of $\tG_i$ and of $G$ induced by
$B_1(\widetilde{w},\tG_i)$ and
  $f_i(B_1(\widetilde{w},\tG_i))$.
\end{enumerate}

 It can be easily checked that
$\widetilde{B}_1,\widetilde{G}_1,\widetilde{\bX}_1$ and $f_1$ satisfy the
conditions (P$_1$),(Q$_1$),(R$_1$),(S$_1$), and
(T$_1$). Now we construct the set $\widetilde{B}_{i+1},$ the graph
$\widetilde{G}_{i+1}$ having $\widetilde{B}_{i+1}$ as the vertex-set, the
triangle-square complex $\widetilde{{\bX}}_{i+1}$ having $\tG_{i+1}$ as its
1-skeleton,  and the map $f_{i+1}: \widetilde{{\bX}}_{i+1}\rightarrow {\bX}.$
Let
 $$Z=\{ (\widetilde{w},z): \widetilde{w} \in \widetilde{S}_i \mbox{ and } z\in
B_1(f_i(\widetilde{w}),G)\setminus f_i(B_1(\widetilde{w},\tG_i))\}.$$
On $Z$ we define a binary relation $\equiv$ by setting $(\widetilde{w},z)\equiv
(\widetilde{w}',z')$ if and only if $z=z'$  and one of the following two
conditions is satisfied:

\begin{itemize}
\item[(Z1)] $\widetilde{w}$ and $\widetilde{w}'$ are the same or adjacent in
$\widetilde{G}_i$ and $z\in B_1(f_i(\widetilde{w}),G)\cap
B_1(f_i(\widetilde{w}'), G);$
\item[(Z2)] there exists $\widetilde{u}\in \widetilde{B}_{i-1}$ adjacent in
$\widetilde{G}_i$ to $\widetilde{w}$ and $\widetilde{w}'$ and such that
$f_i(\widetilde{u})f_i(\widetilde{w})zf_i(\widetilde{w}')$ is a square-cell
of ${\bX}$.
\end{itemize}

In what follows, the above relation will be used in the inductive step to construct $\tX_{i+1}$, $f_{i+1}$ and all related objects.

\subsubsection{Definition of $\tG_{i+1}$.}
\label{s:def-i+1}
In this subsection, performing the inductive step, we define $\tG_{i+1}$ and $f_{i+1}$.
First however we show that the relation $\equiv$ defined in the previous subsection is an equivalence relation.
The set of vertices of the graph $\tG_{i+1}$ will be then defined as the union of the set of vertices of the previously
constructed graph $\tG_{i}$ and the set of equivalence classes of $\equiv$.

\begin{lemma} \label{equiv} The relation $\equiv$ is an equivalence relation on
$Z$.
\end{lemma}

\begin{proof} For any vertex $\widetilde{w}\in \widetilde{B}_i,$ we will
  denote by $w=f_i(\widetilde{w})$ its image in $X$ under $f_i$. Since the
  binary relation $\equiv$ is reflexive and symmetric, it suffices to
  show that $\equiv$ is transitive. Let $(\widetilde{w},z)\equiv
  (\widetilde{w}',z')$ and $(\widetilde{w}',z')\equiv (\widetilde{w}'',z'')$. We
  will prove that $(\widetilde{w},z)\equiv (\widetilde{w}'',z'').$ By
  definition of $\equiv,$ we conclude that $z=z'=z''$. By definition of
$\equiv,$ $z\in B_1(w,G)\cap B_1(w',G)\cap
B_1(w'',G).$

If $\tw \sim \tw''$ (in $\widetilde{G}_i$), then by definition of $\equiv$,
$(\tw,z)\equiv
(\tw'',z)$ and we are done. If $\tw \nsim \tw''$ and if there exists
$\tu \in \tB_{i-1}$ such that $\tu \sim \tw, \tw''$, then by (R$_i$)
applied to $\tu$, we obtain that $u \sim w, w''$ and $w \nsim
w''$. Since $(\tw,z), (\tw'',z) \in Z$, we have $z\sim w,w''$. Moreover,
if $z\sim u$, then by (R$_i$) applied to $u$, there exists $\tz \in
\tB_i$, such that $\tz \sim \tu, \tw, \tw''$ and $f_i(\tz)=z$.  Thus
$(\tw,z), (\tw',z) \notin Z$, which is a contradiction. Consequently,
if $\tw \nsim \tw''$ and if there exists $\tu \in \tB_{i-1}$ such that
$\tu \sim \tw, \tw''$ and $f_i(\widetilde u)=u$, then $uwzw''$ is an induced
square in $G$,
and by condition (Z2), we are done. Therefore, in the rest of the
proof, we will make the following assumptions and show that they lead to a
contradiction.
\begin{enumerate}
\item[($A_1$)] $\tw \nsim \tw''$;
\item[($A_2$)] there is no $\tu \in \tS_{i-1}$ such that $\tu \sim
\tw, \tw''$.
\end{enumerate}

\begin{claim}
For any couple $(\tw,z) \in Z$ the following properties hold:
\begin{enumerate}
\item[($A_3$)] there is no neighbor $\tz \in
  \tB_{i-1}$ of $\tw$ such that $f_i(\tz)=z$;
\item [($A_4$)] there is no neighbor  $\tu \in
  \tB_{i-1}$ of $\tw$ such that $u \sim z$;
\item [($A_5$)] there are no $\tx, \ty \in \tB_{i-1}$ such that
  $\tx \sim \tw, \ty$ and   $y \sim z$.
\end{enumerate}
\end{claim}

\begin{proof}
If $\tw$ has a neighbor $\tz \in \tB_{i-1}$ such that $f_i(\tz)=z$,
then $(\tw,z) \notin Z$,  a contradiction. This establishes $(A_3)$.

If $\tw$ has a neighbor $\tu \in \tB_{i-1}$ such that $u \sim z$, then
by (R$_i$) applied to $\tu$, there exists $\tz \in \tB_i$ such that
$\tz \sim \tu,\tw$. Thus $(\tw,z) \notin Z$, a contradiction, establishing $(A_4)$.

If there exist $\tx, \ty \in \tB_{i-1}$ such that
$\tx \sim \tw, \ty$ and $y \sim z$, then $yxwz$ is an
induced square in $G$. From (S$_i$) applied to $\ty,\tx$, there
exists $\tz \in \tB_i$ such that $\tz \sim \ty,\tw$ and $f_i(\tz)=z$. Thus
$(\tw,z)
\notin Z$, a contradiction, and therefore $(A_5)$ holds as well.
\end{proof}

We distinguish three cases depending on which of the conditions (Z1)
or (Z2) are satisfied by the pairs $(\widetilde{w},z)\equiv
(\widetilde{w}',z')$ and $(\widetilde{w}',z')\equiv (\widetilde{w}'',z'')$.

\medskip\noindent
{\bf Case 1:}  $\widetilde{w}'$ is adjacent in  $\widetilde{G}_i$ to both
$\widetilde{w}$ and $\widetilde{w}''.$

\medskip
By (Q$_i$), the graph $\widetilde{G}_i$ satisfies the triangle condition
TC($\widetilde{v}$),  thus there exist two vertices
$\widetilde{u},\widetilde{u}'\in \widetilde{S}_{i-1}$ such that $\widetilde{u}$
is
adjacent to $\widetilde{w}, \widetilde{w}'$ and $\widetilde{u}'$ is adjacent to
$\widetilde{w}',\widetilde{w}''$. By ($A_2$), $\tu \nsim \tw''$, $\tu' \nsim
\tw$, $\tu \neq \tu'$.

If $\tu \sim \tu'$, then by (T$_i$) applied to $\tw'$ and by
$(A_3)\&(A_4$), the vertices $u,u',w,w',w'',z$ induce $W_5$ in $G$.
By TC($\tv$), there exists $\tx \in \tS_{i-2}$ such that $\tx \sim \tu, \tu'$.
By
(R$_i$) applied to $\tu$ and $\tu'$, we get $x \notin
\{u,u',w,w',w''\}$ and $x \sim u, u'$. From ($A_4)$\&$(A_5$), we get $x
\neq z$ and $x \nsim z$. Since $G$ satisfies the
$\widehat{W}_5$-wheel condition, there exists a vertex $y$ of $G$ adjacent to
$x,u,u',w,w',w'',z$. By (R$_{i}$) applied to $\tu$, there
exists $\ty \sim \tw,\tu,\tx$ and thus $\ty \in \tB_{i-1}$, contradicting
the property $(A_4)$.

Suppose now that $\tu \nsim \tu'$. Then $i \geq 2$ and by QC($\tv$),
there exists $\tx \in \tS_{i-2}$ such that $\tx \sim \tu, \tu'$. From
($A_4)\&(A_5$), $x \neq z$ and $x \nsim z$. Consequently,
$z,w,w',w'',u,u',x$ induce a $W_6^-$, contradicting Lemma~\ref{W__k}.

\medskip\noindent {\bf Case 2:} $\widetilde{w}$ and $\widetilde{w}'$ are
adjacent in $\widetilde{G}_i$, and there exists $\widetilde{u}'\in
\widetilde{B}_{i-1}$ adjacent to $\widetilde{w}'$ and $\widetilde{w}''$ such
that
$u'w'w''z$ is a square-cell of ${\bX}$.

\medskip By $(A_1)\&(A_2)$, $\tw \nsim \tw''$ and $\tu' \nsim
\tw$.  By the triangle condition TC($\widetilde{v}$) for $\widetilde{G}_i$, there
exists a vertex $\widetilde{u}\in \widetilde{B}_{i-1}$ different from
$\widetilde{u}'$ and adjacent to $\widetilde{w}$ and $\widetilde{w}'$. By
$(A_3)\&(A_4)$,
$u \neq z$ and $u \nsim z$. By $(A_2)$, $\tu \nsim \tw''$.

If $\tu \sim \tu'$, by (T$_i$) applied to $w'$, $z,w,w',u,u',w''$ induce a
$W_5^-$, contradicting Lemma~\ref{W__k}. Thus $\tu \nsim \tu'$. By the
quadrangle condition QC($\widetilde{v}$) for $\widetilde{G}_i$, there exists a
vertex $\widetilde{x}\in \widetilde{S}_{i-2}$ adjacent to $\widetilde{u}$ and
$\widetilde{u}'$. From $(A_4)\&(A_5)$, $x \neq z$ and $x \nsim z$. By
(T$_i$) applied to $\tw'$ and by (R$_i$) applied to $\tu'$, we get that
$z,w,w',w'',u,u',x$ induce a twin-house. By Lemma~\ref{house+K2,3}
there exists $y$ in $G$ such that $y\sim w,w'',u',x$ and $y \nsim
u, z$. By (R$_i$) applied to $u'$, there exists $\ty \in \tB_i$ such
that $\ty \sim \tu',\tw'',\tx$. By (S$_i$) applied to $\tu,\tx$ and to
the square $uxyw$, we get $\ty \sim \tw$. Consequently, $\ty \in \tS_{i-1}$,
$\ty \sim \tw, \tw''$, contradicting ($A_2$).

\medskip\noindent
{\bf Case 3:}  There exist
$\widetilde{u},\widetilde{u}'\in \widetilde{B}_{i-1}$ such that the vertex
$\widetilde{u}$ is adjacent in $\widetilde{G}_i$ to
$\widetilde{w},\widetilde{w}',$ the vertex
$\widetilde{u}'$ is adjacent to $\widetilde{w}',\widetilde{w}'',$ and $uwzw'$
and $u'w'zw''$ are square-cells of ${\bX}$.

\medskip From $(A_1)\&(A_2)$, $\tw \nsim \tw''$, $\tu \neq \tu'$, $\tu
\nsim \tw''$, and $\tu' \nsim \tw$. From $(A_3)$, $u \neq z \neq u'$ and $z
\nsim u,u'$.  If $\tu \sim \tu'$, by (T$_i$) applied to $w'$
and by (R$_i$) applied to $u, u'$, the vertices $z,w,w',w'',u,u'$
induce a double-house, which is impossible from
Lemma~\ref{no-double-house}. Thus $\tu \nsim \tu'$.

By QC($\tv$), there exists $\tx \in \tS_{i-2}$ such that $\tx \sim
\tu, \tu'$. By $(A_4)\&(A_5)$, $x \neq z$ and $x \nsim z$. By (T$_i$)
applied to $w'$ and by (R$_i$) applied to $u,u'$, the vertices
$z,w,w',w'',u,u',x$ induce $CW_3$. Thus, by the 3-cube condition,
there exists a vertex $y$ of $G$ such that $y \sim x, w,w''$ and $y \nsim
z,w',u,u'$. By (R$_i$) applied to $\tx$, there is $\ty \in \tB_i$ such
that $\ty \sim \tx$.  By (S$_i$) applied to $\tu,\tx$ and to the
square $uxyw$, we have $\ty \sim \tw$. By (S$_i$) applied to $\tu',\tx$ and to
the square $u'xyw''$, we get $\ty \sim \tw''$. Consequently, $\ty \in
\tS_{i-1}$, $\ty \sim \tw, \tw''$, contradicting ($A_2$).
\end{proof}

Let $\widetilde{S}_{i+1}$ denote the equivalence classes of $\equiv$, i.e.,
$\widetilde{S}_{i+1}=Z/_{\equiv}$. For a couple $(\widetilde{w},z)\in Z$,
we will denote by $[\widetilde{w},z]$ the equivalence class of $\equiv$
containing  $(\widetilde{w},z).$ Set $\widetilde{B}_{i+1}:=\widetilde{B}_i\cup
\widetilde{S}_{i+1}.$
Let $\widetilde{G}_{i+1}$ be the graph having  $\widetilde{B}_{i+1}$ as the
vertex set in which two vertices $\widetilde{a},\widetilde{b}$ are adjacent if
and
only if one of the following conditions holds:
\begin{itemize}
\item[(1)] $\widetilde{a},\widetilde{b}\in \widetilde{B}_i$ and
$\widetilde{a}\widetilde{b}$ is an edge of $\widetilde{G}_i$,
\item[(2)] $\widetilde{a}\in \widetilde{B}_i$,  $\widetilde{b}\in
\widetilde{S}_{i+1}$ and $\widetilde{b}=[\widetilde{a},z]$,
\item[(3)] $\widetilde{a},\widetilde{b}\in \widetilde{S}_{i+1},$
$\widetilde{a}=[\widetilde{w},z]$,
$\widetilde{b}=[\widetilde{w},z']$ for a vertex $\widetilde{w}\in \tB_i,$ and
$z\sim z'$ in the graph $G$.
\end{itemize}

Finally, we define the map $f_{i+1}: \widetilde{B}_{i+1}\rightarrow V({\bX})$ in
the
following way: if $\widetilde{a}\in \widetilde{B}_i$, then
$f_{i+1}(\widetilde{a})=f_i(\widetilde{a}),$  otherwise, if $\widetilde{a}\in
\widetilde{S}_{i+1}$ and
$\widetilde{a}=[\widetilde{w},z],$ then $f_{i+1}(\widetilde{a})=z.$ Notice that
$f_{i+1}$ is well-defined because all couples representing $\widetilde{a}$ have one and the same vertex $z$ in the
second argument. In the sequel, all vertices of $\widetilde{B}_{i+1}$ will
be denoted with a tilde and their images in $G$ under $f_{i+1}$ will be denoted
without tilde, e.g. if $\widetilde{w}\in \widetilde{B}_{i+1},$
then $w=f_{i+1}(\widetilde{w})$.

\subsubsection{Properties of $\tG_{i+1}$ and $f_{i+1}$.}
\label{s:propG}
In this subsection we check our inductive assumptions, verifying the
properties (P$_{i+1}$) through (T$_{i+1}$) for $\tG_{i+1}$ and $f_{i+1}$
defined above. In particular it allows us to define the corresponding
complex $\tX_{i+1}$.
\medskip

\begin{lemma} \label{Pi+1}  $\tG_{i+1}$ satisfies the property $(P_{i+1})$,
i.e.,
$B_j(v,\tG_{i+1})=\widetilde{B}_j$ for any $j\le i+1.$
\end{lemma}

\begin{proof} By definition of edges of $\widetilde{G}_{i+1},$ any vertex
$\widetilde{b}$ of $\widetilde{S}_{i+1}$ is adjacent
to at least one vertex of $\tB_i$ and all such neighbors of $\widetilde{b}$ are
vertices of the form $\widetilde{w}\in \widetilde{B}_i$
such that $\widetilde{b}=[\widetilde{w},z]$ for a couple $(\widetilde{w},z)$ of
$Z.$ By definition of $Z,$ $\widetilde{w}\in \tS_i,$
whence any vertex of $\tS_{i+1}$ is adjacent only to vertices of $\tS_i$ and
$\tS_{i+1}.$ Therefore, the distance between
the basepoint $\widetilde{v}$ and any vertex $\widetilde{a}\in \tB_i$ is the
same in the graphs $\tG_i$ and $\tG_{i+1}.$ On the
other hand, the distance in  $\tG_{i+1}$ between $\widetilde{v}$ and any vertex
$\widetilde{b}$ of $\tS_{i+1}$ is $i+1.$ This
shows that indeed $B_j(v,\tG_{i+1})=\widetilde{B}_j$ for any $j\le i+1.$
\end{proof}

\begin{lemma} \label{Qi+1}  $\widetilde{G}_{i+1}$ satisfies the property
$(Q_{i+1}),$ i.e., the graph $\widetilde{G}_{i+1}$
is weakly modular with respect to the basepoint $\widetilde v$.
\end{lemma}

\begin{proof} First we show that  $\widetilde{G}_{i+1}$ satisfies the triangle
condition TC($\widetilde{v}$).
Pick two adjacent vertices $\widetilde{x},\widetilde{y}$ having in
$\widetilde{G}_{i+1}$ the same distance to $\widetilde{v}.$
Since by Lemma \ref{Pi+1}, $\widetilde{G}_{i+1}$ satisfies the property
(P$_{i+1}$) and the graph $\tG_i$ is weakly modular
with respect to $\widetilde v$, we can suppose that
$\widetilde{x},\widetilde{y}\in \tS_{i+1}.$ From the definition of the edges
of $\tG_{i+1}$, there exist two couples $(\widetilde{w},z),(\widetilde{w},z')\in
Z$ such that $\widetilde{w}\in \tB_i,$ $z$ is
adjacent to $z'$ in $G,$ and
$\widetilde{x}=[\widetilde{w},z],\widetilde{y}=[\widetilde{w},z'].$ Since
$\widetilde{w}$ is
adjacent in $\tG_{i+1}$ to both $\widetilde{x}$ and $\widetilde{y},$ the
triangle condition TC($\widetilde{v}$) is established.

Now we show that $\widetilde{G}_{i+1}$ satisfies the quadrangle condition
QC($\widetilde{v}$). Since  the properties
(P$_{i+1}$) and (Q$_i$) hold, it suffices to consider a vertex $\widetilde{x}\in
\tS_{i+1}$ having two nonadjacent
neighbors $\widetilde{w},\widetilde{w}'$ in $\widetilde{S}_i.$ By definition of
$\tG_{i+1},$ there exists
a vertex $z$ of ${\bX}$ and couples $(\widetilde{w},z),(\widetilde{w}',z)\in Z$
such that $\widetilde{x}=[\widetilde{w},z]$
and $\widetilde{x}=[\widetilde{w}',z].$ Hence $(\widetilde{w},z)\equiv
(\widetilde{w}',z).$ Since $\widetilde{w}$ and $\widetilde{w}'$
are not adjacent, by condition (Z2) in the definition of $\equiv$ there exists
$\widetilde{u}\in \tB_{i-1}$ adjacent
to $\widetilde{w}$ and $\widetilde{w}'$, whence
$\widetilde{x},\widetilde{w},\widetilde{w}'$ satisfy QC($\widetilde{v}$).
\end{proof}

We first prove that the mapping $f_{i+1}$ is a graph homomorphism
(preserving edges) from
$\tG_{i+1}$ to $G$. In particular, this implies that two adjacent vertices of
$\tG_{i+1}$ are mapped in $G$ to
different vertices.

\begin{lemma}\label{lem-homomorphism}
$f_{i+1}$ is a graph homomorphism from $\tG_{i+1}$ to $G$, i.e., for any edge
$\ta\tb$ of $\tG_{i+1}$, $ab$ is an edge of $G$.
\end{lemma}

\begin{proof}
Consider an edge $\ta\tb$ of $\tG_{i+1}$. If $\ta,\tb \in \tB_i$, the
lemma holds by (R$_i$) or (T$_i$)  applied to $\ta$. Suppose that $\ta \in
\tS_{i+1}$. If $\tb \in \tB_i$, then $\ta = [\tb,a]$, and $ab$ is an
edge of $G$. If $\tb \in \tB_{i+1}$, then the fact that $\ta$ and $\tb$ are
adjacent  implies that there
exists a vertex $\tw \in \tB_i$ such that $\ta = [\tw,a], \tb= [\tw,b]$ and
such that $a \sim b$ in $G$.
\end{proof}

We now prove that $f_{i+1}$ is locally surjective at any vertex in
$\tB_i$.

\begin{lemma}\label{lem-locally-surjective}
If $\ta \in \tB_i$ and if $b \sim a$ in $G,$ then there exists
a vertex $\tb$ of $\tG_{i+1}$  adjacent to $\ta$ such that $f_{i+1}(\tb) = b$.
\end{lemma}

\begin{proof}
If $\ta \in \tB_{i-1}$, the lemma holds by (R$_i$). Suppose that $\ta
\in \tS_i$ and consider $b \sim a$ in $G$. If $\ta$ has a
neighbor $\tb \in \tB_i$ mapped to $b$ by $f_i$, we are
done. Otherwise $(\ta,b) \in Z$, $[\ta,b]\sim \ta $ in
$\tG_{i+1}$ and $[\ta,b]$ is mapped to $b$ by $f_{i+1}$.
\end{proof}

Before proving the local injectivity of $f_{i+1}$, we formulate a technical
lemma.

\begin{lemma}\label{lem-prop-Z}
Let $(\tw,a),(\tw',a) \in Z$ be such that $(\tw,a)\equiv (\tw',a)$.  If
$(\tw,b) \in Z$ and
$b \sim w'$ in $G$, then $\tw \sim \tw'$, $(\tw',b) \in Z$ and
$(\tw,b) \equiv (\tw',b)$.
\end{lemma}

\begin{proof}
First suppose that $\tw \nsim \tw'$. Since $(\tw,a)\equiv (\tw',a)$,
there exists $\tu \in \tS_{i-1}$ such that $\tu \sim \tw, \tw'$ and
$wuw'a$ is an induced square in $G$. In $G$, $b \sim
w, w'$, thus $b,w,u,a,w'$ induce $K_{2,3}$ if $b \nsim
a,u$, $W_4$ if $b\sim a,u$, or $W_4^-$ otherwise. In any case,
we get a contradiction.

Thus $\tw \sim \tw'$. If $(\tw',b) \notin Z$, then there
exists $\tb \in \tB_i$ such that $\tb \sim \tw'$ and $f_i(\tb)=b$. In $G$,
$wbw'$
is a triangle, thus $\tb \sim \tw$ by condition (R$_i$) applied to $\tb$. This
implies that $(\tw,b) \notin Z$. Consequently, $(\tw,b), (\tw',b) \in
Z$ and $(\tw,b) \equiv (\tw',b)$ since $\tw \sim \tw'$.
\end{proof}

We can now prove that $f_{i+1}$ is locally injective.

\begin{lemma}\label{lem-locally-injective}
If $\ta \in \tB_{i+1}$ and  $\tb, \tc$ are distinct neighbors of
$\ta$ in $\tG_{i+1}$, then $b \neq c$.
\end{lemma}

\begin{proof}
First note that if $\tb \sim \tc$, the assertion holds by
Lemma~\ref{lem-homomorphism}; in the following we assume that $\tb
\nsim \tc$. If $\ta, \tb, \tc \in \tB_i$, the lemma holds by (R$_i$)
or (T$_i$) applied to $\ta$.
Suppose first that $\ta \in \tB_i$. If $\tb, \tc \in \tS_{i+1}$, then
$\tb = [\ta,b]$ and $\tc = [\ta,c],$ and thus $b\neq c$. If $\tb \in
\tB_i$ and $\tc = [\ta,c] \in \tS_{i+1}$, then $(\ta,b) \notin Z$, and
thus $c \neq b$. Thus, let $\ta\in \tS_{i+1}.$

If $\tb, \tc \in \tB_i$ and $\ta \in \tS_{i+1}$, then $\ta =
[\tb,a] = [\tc,a]$. Since $(\tb,a) \equiv (\tc,a)$ and since $\tb
\nsim \tc$, there exists $\tu \in \tB_{i-1}$ such that $\tu \sim \tb,
\tc$ and $abuc$ is an induced square of $G$. This
implies that $b \neq c$.

If $\ta, \tb, \tc \in \tS_{i+1}$, then there exist $\tw, \tw' \in \tB_i$
such that $\tb = [\tw,b]$, $\tc =[\tw',c],$ and $\ta= [\tw,a] =
[\tw',a]$. If $b = c$, then  $[\tw,b] =
[\tw',b] = [\tw',c]$ by Lemma~\ref{lem-prop-Z}, and thus $\tb = \tc$,
which is impossible.

If $\ta, \tb \in \tS_{i+1}$ and $\tc \in \tS_i$, then there exists
$\tw \in \tS_i$ such that $\tb = [\tw,b]$ and $\ta = [\tw,a] =
[\tc,a]$. If $\tw \sim \tc$, then $(\tw,c) \notin Z$ and thus,
$(\tw,c) \neq (\tw,b)$, i.e., $b \neq c$. If $\tw \nsim \tc$, since
$[\tw,a] = [\tc,a]$, there exists $\tu \in \tS_{i-1}$ such that
$\tu\sim \tw,\tc$ and such that $acxu$ is an induced square of
$G$. Since $\tw$ and $\tc$ are not adjacent, by (R$_i$) applied
to $\tu$, $w$ and
$c$ are not adjacent as well. Since $w \sim b$, this implies that $b \neq c$.
\end{proof}

\begin{lemma}\label{lem-triangles}
If $\ta \sim \tb, \tc$ in $\tG_{i+1}$, then $\tb \sim \tc$ if and only
if $b \sim c$.
\end{lemma}

\begin{proof}
If $\tb \sim \tc$, then  $b \sim c$ by Lemma~\ref{lem-homomorphism}.
Conversely, suppose that $b \sim c$ in $G$.
If $\ta, \tb, \tc \in
\tB_i$, then  $\tb \sim \tc$ by condition (R$_i$) applied to
$\ta$. Therefore, further we will assume that at least one of the
vertices $\ta, \tb, \tc$ does not belong to $\tB_i$.

First suppose that $\ta \in \tB_i$. If $\tb, \tc \in \tS_{i+1}$, $\tb =
[\ta,b]$ and $\tc = [\ta,c]$. Since $b\sim c$, by construction, we
have $\tb \sim \tc$ in $\tG_{i+1}$. Suppose now that $\tb = [\ta,b]
\in S_{i+1}$ and $\tc \in \tB_i$. If there exists $\tb' \sim \tc$ in
$\tG_i$ such that $f_i(\tb') = b$, then by ($R_i$) applied to $\tc$,
$\ta \sim \tb'$ and $(\ta,b) \notin Z$, which is a contradiction. Thus
$(\tc,b) \in Z$ and since $\tc \sim \ta$, $[\tc,b]=[\ta,b]=\tb$, and
consequently, $\tc \sim \tb$. Therefore, let $\ta \in \tS_{i+1}$.

If $\tb, \tc \in \tB_i$ and $\ta \in \tS_{i+1}$, then
$\ta =[\tb,a] = [\tc,a]$ and either $\tb \sim \tc$, or there exists
$\tu \in \tS_{i-1}$ such that $\tu \sim \tb, \tc$ and $ubac$ is an
induced square in $G$, which is impossible because $b \sim c$.

If $\ta, \tb \in \tS_{i+1}$ and $\tc \in \tB_i$, then there exists
$\tw \in \tB_i$ such that $\tb = [\tw,b]$ and $\ta = [\tw,a] =
[\tc,a]$. By Lemma~\ref{lem-prop-Z}, $(\tc,b) \in Z$ and $\tb =
[\tw,b] = [\tc,b]$. Consequently, $\tc \sim \tb$.

If $\ta, \tb, \tc \in \tS_{i+1}$, there exist $\tw, \tw' \in \tB_i$
such that $\tb = [\tw,b]$, $\tc = [\tw',c]$ and $\ta = [\tw,a] =
[\tw',a]$. If $\tw \sim \tc$ or $\tw' \sim \tb$, then $\tb \sim \tc$ because
$b\sim
c$. Suppose now that $\tw \nsim \tc$, $\tw' \nsim
\tb$. From previous case applied to $\ta,\tb \in \tS_{i+1}$
(respectively, $\ta,\tc \in \tS_{i+1}$) and $\tw' \in \tB_i$ (respectively, $\tw \in
\tB_i$), it follows that $w \nsim c$ and $w' \nsim b$.  If $\tw \sim
\tw'$, then $a,b,w,w',c$ induce $W_4$ in $G$, which is
impossible. Since $[\tw,a] = [\tw',a]$, there exists $\tu \in
\tS_{i-1}$, such that $\tu \sim \tw, \tw'$ and such that $awuw'$ is an
induced square in $G$. If $u\sim b$, then by (R$_i$) applied to
$u$, there exists $\tb' \in \tB_i$ mapped to $b$ by $f_i$ such that
$\tb' \sim \tu, \tw$ and thus $(\tw,b) \notin Z$, which is a
contradiction. Using the same arguments, we have that $u \nsim c$ and
thus, $a,b,c,w',u,w$ induce $W_5^-$ in $G$, which is impossible.
\end{proof}

We can now prove that the image under $f_{i+1}$ of an induced triangle
or square is an induced triangle or square.

\begin{lemma} \label{cellular}
If $\widetilde{a}\widetilde{b}\widetilde{c}$ is a triangle in $\tG_{i+1}$, then
$abc$ is a triangle in $G$. If
$\widetilde{a}\widetilde{b}\widetilde{c}\widetilde{d}$ is an induced square of
$\tG_{i+1}$, then $abcd$ is an induced square in $G$. In particular,
the graph $\tG_{i+1}$ does not contain induced $K_{2,3}$ and $W^-_4$.
\end{lemma}

\begin{proof}
For triangles, the assertion follows directly from Lemma~\ref{lem-homomorphism}.
Consider now a square $\ta\tb\tc\td$. From
Lemmas~\ref{lem-homomorphism} and \ref{lem-locally-injective}, the vertices $a,
b,
c,$ and $d$ are pairwise distinct and $a\sim b$, $b \sim c$, $c \sim d$, $d
\sim a$. From Lemma~\ref{lem-triangles}, $a\nsim c$ and $b \nsim
d$. Consequently, $abcd$ is an induced square in $G$.

Now, if $\tG_{i+1}$ contains an induced $K_{2,3}$ or $W^-_4,$ from the first
assertion
and Lemma \ref{lem-triangles} we conclude that the image under $f_{i+1}$ of this
subgraph
will be an induced $K_{2,3}$ or  $W^-_4$
in the graph $G$, a contradiction.
\end{proof}

The second assertion of Lemma~\ref{cellular} implies that replacing all 3-cycles
and all induced 4-cycles of $\widetilde{G}_{i+1}$ by triangle- and
square-cells, we will obtain a triangle-square flag complex, which we denote by
$\widetilde{{\bX}}_{i+1}.$
Then obviously $\widetilde{G}_{i+1}=G(\widetilde{{\bX}}_{i+1}).$
The first assertion  of Lemma~\ref{cellular} and the flagness of $\bX$ imply that $f_{i+1}$ can be
extended
to a cellular map from $\widetilde{\bX}_{i+1}$ to $\bX$: $f_{i+1}$
maps a triangle $\widetilde{a}\widetilde{b}\widetilde{c}$ to the triangle $abc$
of
${\bX}$ and a square $\widetilde{a}\widetilde{b}\widetilde{c}\widetilde{d}$ to
the
square $abcd$ of ${\bX}$.

\begin{lemma}\label{Ri+1}\label{Ti+1}
$f_{i+1}$ satisfies the conditions $(R_{i+1})$ and $(T_{i+1}).$
\end{lemma}

\begin{proof}
From Lemmas~\ref{lem-locally-injective} and~\ref{lem-triangles}, we
know that for any $\tw \in \tB_{i+1}$, $f_{i+1}$ induces an
isomorphism between the subgraph of $\tG_{i+1}$ induced by $B_1(\tw,\tG_{i+1})$
and the subgraph of $G$ induced by
$f_{i+1}(B_1(\tw,\tG_{i+1}))$. Consequently, the condition  (T$_{i+1}$) holds.
From Lemma~\ref{lem-locally-surjective}, we know that
$f_{i+1}(B_1(\tw,\tG_{i+1})) = B_1(w,G)$ and consequently
(R$_{i+1}$) holds as well.
\end{proof}

\begin{lemma} \label{Si+1} For any $\widetilde{w},\widetilde{w}'\in
  \widetilde{B}_{i}$ such that the vertices
  $w=f_{i+1}(\widetilde{w}),w'=f_{i+1}(\widetilde{w}')$ belong to a square
  $ww'u'u$ of ${\bX}$, there exist $\widetilde{u},\widetilde{u}'\in
  \widetilde{B}_{i+1}$ such that $f_{i+1}(\widetilde{u})=u,
  f_{i+1}(\widetilde{u}')=u'$ and
$\widetilde{w}\widetilde{w}'\widetilde{u}'\widetilde{u}$
  is a square of $\widetilde{{\bX}}_{i+1},$ i.e., $\widetilde{{\bf
      X}}_{i+1}$ satisfies the property $(S_{i+1}).$
\end{lemma}

\begin{proof}

By Lemma~\ref{Ri+1} applied to $\tw$ and $\tw'$, we know that in
$\tG_{i+1}$ there exist a unique $\tu$ (respectively, a unique $\tu'$) such
that $\tu\sim\tw$ (respectively, $\tu'\sim\tw'$) and $f_{i+1}(\tu)=u$
(respectively, $f_{i+1}(\tu)=u'$). By Lemma~\ref{Ri+1}, $\tw$ (respectively, $\tw'$)
is the unique neighbor of $\tu$ (resp. $\tu$) mapped to $w$
(respectively, $w'$) by $f_{i+1}$.

Note that if $\tw, \tw' \in \tB_{i-1}$, the lemma holds by
condition (S$_i$). Let us assume further that $\tw \in \tS_i$.

\medskip
\noindent{\textbf{Case 1.} $\tw' \in \tS_{i-1}$.}

\medskip
If $\tu' \in \tB_{i-1}$, by (S$_i$) applied to $\tw'$ and $\tu'$, we
conclude that $\tw\tw'\tu'\tu$ is a square in $\tG_{i+1}$.

If $\tu' \in \tS_i$ and $\tu \in \tS_{i-1}$, then Lemma~\ref{Ri+1}
applied to $\tw$, implies that $\tu$ is not adjacent to $\tw'$. Thus,
by the quadrangle condition QC($\tv$), there exists $\tx \in \tS_{i-2}$
such that $\tx\sim\tu, \tw'$. Hence, $w,w',u,u',x$ induce in $G$
a forbidden $K_{2,3}, W_4^-$, or $W_4$, which is impossible.

Suppose now that $\tu', \tu \in \tS_i$. By TC($\tv$), there exists
$\tx \in \tS_{i-1}$ different from $\tw'$ such that $\tx\sim\tu,\tw$.
Since $G$ does not contain $W_4^-$ or $W_4$, $x \nsim u', w'$ and
the vertices $u,w,w',u',x$ induce a house.  By the 3-prism condition
there exists $y$ in $G$ such that $y\sim x,u',w'$ and $y \nsim
u,w$.  Since $x\nsim w'$, by R$_{i}$ applied to $\tw$, $\tx \nsim
\tw'$.  Applying QC($\tv$), there exists $\tz \in \tS_{i-2}$ such that
$\tz\sim\tx,\tw'$ and $\tz \nsim \tw$. Since $\tz \in \tS_{i-2}$, $\tz
\nsim \tu'$ and thus by R$_{i+1}$ applied to $\tw'$, $z \nsim
u'$. Consequently, $z \neq y$.  Thus, from Lemma~\ref{cellular},
$xzw'w$ is an induced square of $G$ and $y,x,z,w',w$ induce a
$K_{2,3}$ if $z\nsim y$ and $W_4^-$ otherwise, which is impossible.
Note that if $\tu'$ has a neighbor $\tu_2$ in $\tB_i$ mapped to $u$ ,
then, exchanging the roles of $\tu'$ and $\tw$, we also get a
contradiction.  Suppose now that neither $\tw$ nor $\tu'$ has a
neighbor in $\tB_i$ mapped to $u$. Thus, $(\tw,u), (\tu',u) \in Z$ and
since $\tw'\in \tS_{i-1}$ is adjacent to $\tw$ and $\tu'$, $(\tw,u)
\equiv (\tu',u)$. Consequently, $\tw\tw'\tu'[\tw,u]$ is a square of
$\tG_{i+1}$ which is mapped by $f_{i+1}$ to the square $ww'u'u$.

\medskip
\noindent{\textbf{Case 2.} $\tw' \in \tS_{i}$.}

\medskip

If $\tu \in \tS_{i-1}$, exchanging the role of $\tw'$ and $\tu$, we
are in the previous case and thus there exists $\tu'' \sim \tw',\tu$
such that $f_{i+1}(\tu'') = u'$. By Lemma~\ref{Ri+1}, we get that
$\tu'=\tu''$ and we are done. For the same reasons, if $\tu' \in
\tS_{i-1}$, applying Case 1 with $\tw'$ in the role of $\tw$ and
$\tu'$ in the role of $\tw'$, we are done.

If $\tu \in \tS_{i}$, by TC($\tv$) there exists $\tx
\in \tB_{i-1}$ such that $\tx\sim \tw,\tu$. Thus, in $G$ there
exists $x \sim u,w$ and, since $G$ does not contain $W_4$ or
$W_4^-$, $x\nsim u',w'$. Applying the 3-prism condition, we get $y$ in $G$
such that $y \sim u',w',x$ and $y \nsim u,w$. Applying the
previous case to $\tw,\tx$ and the square $wxyw'$ of $G$, we know
that there exists $\ty \in \tB_{i}$ such that $\tw\tx\ty\tw'$ is an
induced square in $\tG_{i+1}$. From Lemma~\ref{Ri+1} applied to
$\tw'$, we deduce that $\ty\sim\tu'$. Applying (S$_i$) to $\tx,\ty$
and to the square $xyu'u$, we get that $\tu\sim\tu'$, thus $\tw\tw'\tu'\tu$ is a
square in $\tG_{i+1}$.  If $\tu' \in \tS_i$, then exchanging the roles of
$\tw, \tw', \tu, \tu'$ we also get that $\tw\tw'\tu'\tu$ is a square
in $\tG_{i+1}$.

Suppose now that $\tw$ has no neighbor in $\tB_i$ mapped to $u$  and
that $\tw'$ has no neighbor in $\tB_i$ mapped to $u'$. Thus, there
exist $[\tw,u]$ and $[\tw',u']$ in $\tS_{i+1}$. By TC($\tv$), there
exists $\tx \in \tS_{i-1}$ such that $\tx \sim \tw, \tw'$. In
$G$, $x \sim w, w'$ and $x \nsim u, u'$ since $G$ does not
contain $W_4^-$ or $W_4$. Applying the 3-prism condition, there is a
vertex $y$ in $G$ such that $y \sim u, u'$ and $y \nsim w, w'$.
By (R$_i$) applied to $\tx$, there exists $\ty$ in $\tB_i$ such that
$\ty \sim \tx$ and $\ty\nsim \tw,\tw'$. If $\ty$ has a neighbor in $\tB_i$
mapped to $u$, then applying the previous case to $\tw,\tx$ and
the square $wxyu$, we conclude that $\tw$ has a neighbor in $\tB_i$
mapped to $u$, which is impossible.
Consequently, $(\ty,u) \in Z$, and since there is
$\tx \in S_{i-1}$ such that $\tx \sim \tw, \ty$ and $wxyu$ is an
induced square in $G$, $(\ty,u) \equiv (\tw,u)$. Using the same
arguments, one can show that there exists $(\ty,u') \in [\tw',u']$. Since
$yuu'$ is a triangle in $G$, and since $[\tw,u] = [\ty, u]$ and
$[\tw',u'] = [\ty, u']$, there is an edge in $\tG_{i+1}$ between
$[\tw,u]$ and $[\tw',u']$. Consequently, $\tw\tw'[\tw',u'][\tw,u]$ is a
square of $\tG_{i+1}$ satisfying the lemma.
\end{proof}

\subsubsection{The universal cover $\tX$.}
\label{s:univ}
Let $\widetilde{\bX}_v$ denote the triangle-square complex obtained as the
directed union
$\bigcup_{i\ge 0} \widetilde{\bX}_i$ with a vertex $v$ of $\bX$ as the
basepoint. Denote by $\tG_v$ the
1-skeleton of $\widetilde{\bX}_v.$ Since each $\tG_i$
is weakly modular with respect to $\tv,$ the graph $\tG_v$ is also weakly
modular with respect to $\tv$. Thus the complex $\widetilde{\bX}_v$ is
simply connected by virtue of Lemma \ref{simplyconnected}. Let also
$f=\bigcup_{i\ge 0}f_i$ be the map from $\widetilde{\bX}_v$ to $\bX$.

\begin{lemma} \label{covering_map}
For any $\tw \in \widetilde{\bX}_v$,
$\mbox{St}(\tw,\widetilde{\bX}_v)$ is isomorphic to $\mbox{St}(w,\bX)$
where $w=f(\tw)$. Consequently,
$f:~\!\!\widetilde{{\bX}}_v\rightarrow~\!\!\bX$ is a covering map.
\end{lemma}

\begin{proof}
Note that, since $\widetilde{\bX}_v$ is a flag complex, a vertex $\tx$
of $\widetilde{\bX}_v$ belongs to $\mbox{St}(\tw,\widetilde{\bX}_v)$ if and only if either $\tx \in
B_1(\tw,\tG_v)$ or $\tx$ has two non-adjacent neighbors in
$B_1(\tw,\tG_v)$.

Consider a vertex $\tw$ of $\btX_v$. Let $i$ be the distance between
$\tv$ and $\tw$ in $\tG_v$ and consider the set $\tB_{i+2}$.  Then the
vertex-set of $\mbox{St}(\tw,\widetilde{\bX}_v)$ is included in
$\tB_{i+2}$. From ($R_{i+2}$) we know that $f$ is an isomorphism
between the graphs induced by $B_1(\tw,\tG_v)$ and $B_1(w,G).$

For any vertex $x$ in  $\mbox{St}(w,\bX)\setminus B_1(w,G)$ there exists an
induced square $wuxu'$ in $G$. From ($R_{i+2}$), there exist $\tu,
\tu' \sim \tw$ in $\tG_v$ such that $\tu \nsim \tu'$. From ($S_{i+2}$) applied
to $\tw,\tu$ and since $\tw$ has a unique neighbor $\tu'$ mapped to
$u'$, there exists a vertex $\tx$ in  $\tG_v$ such that $f(\tx) = x$, $\tx \sim
\tu, \tu'$ and $\tx \nsim \tw$. Consequently, $f$ is a surjection
from $V(\mbox{St}(\tw,\widetilde{\bX}_v))$ to $V(\mbox{St}(w,\bX))$.

Suppose by way of contradiction that there exist two distinct vertices
$\tu, \tu'$ of $\mbox{St}(\tw,\widetilde{\bX}_v)$ such that $f(\tu) =
f(\tu') = u$. If $\tu,\tu' \sim \tw$, by condition ($R_{i+1}$) applied
to $\tw$ we get a contradiction.  Suppose now that $\tu \sim \tw$ and
$\tu' \nsim \tw$ and let $\tz \sim \tw, \tu'$. This implies that $w,
u, z$ are pairwise adjacent in $G$. Since $f$ is an isomorphism
between the graphs induced by $B_1(\tw,\tG_v)$ and $B_1(w,G)$, we
conclude that $\tz \sim \tu$. But then $f$ is not locally injective
around $\tz$, contradicting the condition (R$_{i+2}$).  Suppose now
that $\tu, \tu' \nsim \tw$. Let $\ta, \tb \sim \tu, \tw$ and $\ta',
\tb' \sim \tu',\tw'$. If $\ta' = \ta$ or $\ta' = \tb$, then applying
(R$_{i+2}$) to $\ta'$, we get that $f(\tu)\neq f(\tu')$. Suppose now
that $\ta' \notin \{\ta,\tb\}$. Then the subgraph of $G$ induced by
$a',w,a,b,u$ is either $K_{2,3}$ if $a' \nsim a,b$, or $W_4$ if
$a'\sim a,b$, or $W_4^-$ otherwise. In all cases, we get a
contradiction.

Hence $f$ is a bijection between the vertex-sets of
$\mbox{St}(\tw,\widetilde{\bX}_v)$ and $\mbox{St}(w,\bX)$. Since
$\btX_v$ is a flag complex, by (R$_{i+2}$), $\ta\sim \tb$ in
$\mbox{St}(\tw,\widetilde{\bX}_v)$ if and only if $a \sim b$ in
$\mbox{St}(w,\bX)$.  By (R$_{i+2}$) applied to $w$ and since $\bX$ and
$\widetilde{\bX}_v$ are flag complexes, $\ta\tb\tw$ is a triangle in
$\mbox{St}(\tw,\widetilde{\bX}_v)$ if and only if $abw$ is a triangle
in $\mbox{St}(w,\bX)$. By ($R_{i+2}$) and since $\bX$ is a flag
complex, if $\ta\tb\tc\tw$ is a square in
$\mbox{St}(\tw,\widetilde{\bX})$, then $abcw$ is a square in
$\mbox{St}(w,\bX)$. Conversely, by the conditions ($R_{i+2}$) and
($S_{i+2}$) and flagness of $\widetilde{\bX}_v$, we conclude that if
$abcw$ is a square in $\mbox{St}(w,\bX)$, then $\ta\tb\tc\tw$ is a
square in $\mbox{St}(\tw,\widetilde{\bX}_v)$. Consequently, for any
$\tw \in \widetilde{\bX}_v$, $f$ defines an isomorphism between
$\mbox{St}(\tw,\widetilde{\bX}_v)$ and $\mbox{St}(w,\bX)$, and thus
$f$ is a covering map.
\end{proof}

\begin{lemma} \label{house-condition-tX}
$\widetilde{\bX}_v$ satisfies the $3$--prism, the $3$--cube, and the
  $\widehat{W}_5$--wheel conditions, and the graph $\tG_v$ does not contain
induced $K_{2,3},
  W_4^-,$ and $W_4$. Moreover, if $G$ is $W_5$--free, then $\tG_v$ is also
$W_5$--free.
\end{lemma}

\begin{proof}
If $\tG_v$ contains an induced $K_{2,3}$ or $W_4^-$, there exists $i$
such that $\tG_i$ contains an induced $K_{2,3}$ or $W_4^-$,
contradicting Lemma~\ref{cellular}. Let $C\in \{W_4,W_5\}$ be an
induced subgraph of $\tG_v$. By Lemma~\ref{covering_map} applied to
the center of the wheel, the subgraph induced by $f(V(C))$ is
isomorphic to $C$.  Since $G$ does not contain any induced $W_4$, the
graph $\tG_v$ also does not contain any induced $W_4$ and if $G$ is
$W_5$--free, $\tG_v$ is also $W_5$--free.

\medskip

\noindent
\emph{$3$--prism condition:}
Let $\tu\tu'\tw'\tw\tx$ be a house in $\widetilde{\bX}_v$ where,
$\tu\tu'\tw'\tw$ is a square and $\tu\tw\tx$ is a triangle.
Consider the image of this house by $f$, i.e.
$uu'w'wx$ in $\bX$.
If the image is not an induced subgraph of $\bX$ then, by Lemma~\ref{covering_map}
(applied to $\tu$), we have $x\sim w'$ and the vertices $x, u, u', w,
w'$ induce a $W_4^-$, a contradiction. Thus $uu'w'wx$ is an induced house in $\bX$.
By the $3$--prism
condition in $\bX$, there exists a vertex $y \in G$ such that $y\sim
u',w',x$ and $y \nsim u,w$. Since $f$ is locally bijective, there
exists $\ty \sim \tx$ such that $f(\ty) = y$. Since $f$ is an
isomorphism from $\mbox{St}(\tx,\widetilde{\bX}_v)$ to $\mbox{St}(x,\bX)$,
considering the
squares $xyu'u$ and $xyw'w$, we get that $\ty \sim \tu',\tw'$ and $\ty
\nsim \tu, \tw$. Thus, $\widetilde{\bX}_v$ satisfies the $3$--prism condition.
\medskip

\noindent
\emph{$3$--cube condition:} Consider three squares
$\tx\ta_1\tb_3\ta_2$, $\tx\ta_2\tb_1\ta_3$, $\tx\ta_3\tb_2\ta_1$ in
$\widetilde{\bX}_v$.
By Lemma~\ref{covering_map} applied to $\tx$, the images $xa_1b_3a_2$
of $\tx\ta_1\tb_3\ta_2$, $xa_2b_1a_3$ of $\tx\ta_2\tb_1\ta_3$ and
$xa_3b_2a_1$ of $\tx\ta_3\tb_2\ta_1$ are squares of $\bX$.
By the $3$--cube condition, in $G$ there exists a vertex $y$ such that
$y \sim b_i$ and $y \nsim x,a_i$, for all $i$. Moreover, for all
distinct $i,j$, $b_i \nsim a_i$ and $b_i \nsim b_j$. By
Lemma~\ref{covering_map} applied to $\ta_1, \ta_2, \ta_3$, for all
distinct $i,j$, $\tb_i \nsim \ta_i$  and $\tb_i \nsim \tb_j$.
Since $f$ is locally
bijective, there exists $\ty \sim \tb_1$ such that $f(\ty) = y$. Since
$f$ is an isomorphism from $\mbox{St}(\tb_1,\widetilde{\bX}_v)$ to
$\mbox{St}(b_1,\bX)$, we get that $\ty \sim \tb_2, \tb_3$ and $\ty
\nsim \ta_2,\ta_3$. When considering
$\mbox{St}(\tb_2,\widetilde{\bX}_v)$, we get that $\ty \nsim
\ta_1$. If $\ty \sim \tx$, $\tG_v$ contains an induced $K_{2,3}$, a
contradiction.  Thus, $\widetilde{\bX}$ also satisfies the $3$--cube
condition.
\medskip

\noindent
\emph{$\widehat{W}_5$--wheel condition:}
Consider $W_5$ in $\tG_v$ made of a $5$--cycle
$(\tx_1,\tx_2,\tx_3,\tx_4,\tx_5,\tx_1)$ and a vertex $\tc$ adjacent to all vertices
of this cycle. Suppose that
there exists a vertex $\tz$ such that $\tz \sim \tx_1,\tx_2$ and $\tz
\nsim \tx_3,\tx_4,\tx_5, \tc$. By Lemma~\ref{covering_map}, the vertices $c, x_1, x_2, x_3, x_4,
x_5$ are all distinct and they induce $W_5$ in $G$.
Moreover, $z \notin \{c, x_1, x_2, x_3,
x_4, x_5\}$, and $z\nsim c,x_3,x_5$.
Similarly, if $z\sim x_4$ then $zx_2x_3x_4$ is a square and, by
Lemma~\ref{covering_map} (applied to $\tx_2$), we have $\tz \sim
\tx_4$, a contradiction.
 By the
$\widehat{W}_5$--wheel condition for $\bX$, there exists $y \sim
c,z,x_1,x_2,x_3,x_4,x_5$ in $G$. Consider the
neighbor $\ty$ of $\tc$ such that $f(\ty)=y$. Since $\mbox{St}(\tc,\widetilde{\bX}_v)$
is isomorphic to $\mbox{St}(c,\bX)$, $\ty \sim \tx_1, \tx_2, \tx_3, \tx_4,
\tx_5$. Considering the star $\mbox{St}(\tx_1,\widetilde{\bX}_v)$, we conclude that $\ty
\sim \tz$.
Consequently,
$\widetilde{\bX}_v$ satisfies the $\widehat{W}_5$--wheel condition.
\end{proof}

Now, we are ready to complete the proof  of the implication  (ii)$\Leftrightarrow$(iii).
Let $\bX$ be a connected triangle-square
flag complex satisfying the local conditions of (ii).
By Lemma \ref{covering_map}, $f: \widetilde{\bX}_v\rightarrow {\bX}$ is
a covering map. By Lemma \ref{simplyconnected}, $\widetilde{\bX}_v$ is simply
connected, thus $\widetilde{\bX}_v$ is the universal cover $\widetilde{\bf
  X}$ of ${\bX}$. Therefore the triangle-square complexes
$\widetilde{\bX}_v, v\in V({\bX}),$ are all universal covers of
$\bX$, whence they are all isomorphic. Since for each vertex $v$ of
$\bX$, the graph $\tG_v=G(\widetilde{\bX}_v)$ is weakly modular with respect
to the basepoint $v,$ we conclude that the 1-skeleton $G(\widetilde{\bf
  X})$ of $\widetilde{\bX}$ is weakly modular with respect to each
vertex, thus $G(\widetilde{\bX})$ is a weakly modular graph. Since
$\widetilde{\bX}$ is isomorphic to any $\widetilde{\bX}_v,$ by Lemma
\ref{house-condition-tX}, $\widetilde{\bX}$ satisfies the same local conditions as
$\bX$. Thus $\widetilde{\bX}$ satisfies the wheel, the 3-prism, and
the 3-cube conditions. If, additionally, $\bX$ is simply connected, then the universal cover
$\widetilde{\bX}$ is ${\bX}$ itself. Therefore, ${\bX}$ coincides
with $\widetilde{\bX}_v$ for any choice of the basepoint $v\in V({\bX}).$
Therefore, by what has been proven above, $G({\bX})$ is a bucolic graph.
This establishes the implication (ii)$\Rightarrow$(iii) of Theorem \ref{theorem0}.

\subsection{Proof of (iii)$\Rightarrow$(ii).}  Let $\bX$ be a
prism flag complex such that $G:=G({\bX})$ is a weakly modular graph
not containing induced $W_4$.  Then $G$ does not contain induced
$K_{2,3}$ and $W^-_4$ because $G$ is the 1-skeleton of a flag
triangle-square cell complex $\bX^{(2)}$ and both $K_{2,3}$ and
$W_4^-$ contain squares intersecting on two edges.  From Lemma
\ref{simplyconnected} we conclude that $\bX^{(2)}$ (and therefore
$\bX$) is simply connected.  Thus, it remains to show that $\bX$
satisfies the 3-prism, the 3-cube, and the $\widehat{W}_5$-wheel
conditions. First suppose that the triangle $uvw$ and the square
$uvxy$ define in $\bX$ a house. Then $w$ is at distance 2 to the
adjacent vertices $x$ and $y$. By the triangle condition, there exists
a vertex $w'$ adjacent to $w,x,$ and $y$ and different from $u$ and
$v.$ If $w'$ is adjacent to one or both of the vertices $u,v,$ then we
will get a forbidden $W^-_4$ or $W_4$ induced by $u,v,x,y,w'.$ This
establishes the 3-prism condition.

To prove the 3-cube condition, let $xyuv,$ $uvwz,$ and $uytz$ be three
squares of $\bX$ pairwise intersecting in edges and all three
intersecting in $u$. If $x$ and $w$ are adjacent, then the vertices
$v,x,w,u,y,z$ induce in $\bX$ a double house, which is impossible by
Lemma \ref{no-double-house} because $\bX$ satisfies the 3-prism
condition. Hence $x\nsim w$ and analogously $x\nsim t$ and $t\nsim w$.
If $x$ is adjacent to $z,$ then $x,y,u,t,z$ induce in $G$ a forbidden
$K_{2,3}$. Thus $x\nsim z$ and analogously $y\nsim w$ and $v\nsim
t$. First suppose that $d(x,z)=2$ in $G$. Since $d(y,z)=2,$ by the
triangle condition there exists a vertex $s$ adjacent to $x,y,$ and
$z$. From what has been shown before, $s\ne u,t,$ hence $y,u,z,t,s$
induce $K_{2,3},$ $W^-_4,$ or $W_4$ depending of whether $s$ is
adjacent to none, one or two of the vertices $u,t$. Thus, $d(x,z)=3$
and for the same reasons, $d(y,w) = d(v,t) = 3$.  By  the quadrangle
condition there exists a vertex $s$ adjacent to $x,w,t$ and distinct
from previous vertices. Since $d(x,z) = d(w,y) = d(t,v) = 3$, $s \nsim
z,y,v$.  If $s$ is adjacent to $u$, then $s,u,v,w,z$ induce a
forbidden $K_{2,3}$. This shows that in this case the vertices
$s,t,u,v,w,x,y,z$ define a 3-cube, establishing the 3-cube condition.

Finally, we establish the $\widehat{W}_5$-wheel condition.  Note
that $\bX$ satisfies the 3-cube and the 3-prism
conditions and  does not contain a
${\bX}(W^-_5)$ by Lemma \ref{W__k}.  Pick a 5-wheel defined by a 5-cycle
$(x_1,x_2,x_3,x_4,x_5,x_1)$ and a vertex $c$ adjacent to all vertices
of this cycle, and suppose that $x_0$ is a vertex adjacent to $x_1$
and $x_5$ and not adjacent to remaining vertices of this 5-wheel. If
$d(x_0,x_3)=3,$ then by the quadrangle condition QC$(x_0)$, there exists a
vertex $y$ adjacent to $x_0,x_2,x_4$ and not adjacent to $x_3.$ Then
the vertices $c,y,x_2,x_3,x_4$ induce a $W_4$ if $y$ is adjacent to
$c$, and a $W^-_4$ otherwise. So, suppose that $d(x_0,x_3)=2$. By the
triangle condition TC$(x_0)$, there exists a vertex $z$ adjacent to
$x_0,x_2,x_3$. Suppose that $z\nsim c$. If $z \sim x_1$, then $x_2,
x_1,z,x_3,c$ induce a forbidden $W_4$. If $z \sim x_5$, the vertices
$x_1,x_2,c,x_5,z$ induce a forbidden $W_4$ if $z \sim x_1$ or a
$W_4^-$ otherwise. If $z \nsim x_1,x_5$, the vertices
$z,x_2,c,x_5,x_0,x_1$ induce a forbidden $W_5^-$. Thus, $z \sim c$.
To avoid a forbidden $W^-_4$ or $W_4$ induced by $z,c,x_1,x_0,x_5,$
the vertex $z$ must be adjacent to $x_1$ and $x_5.$ Finally, to avoid
$W_4$ induced by $z,c,x_3,x_4,x_5$, the vertex $z$ must be adjacent to
$x_4$ as well. As a result, we conclude that $z$ is adjacent to $x_0$
and to all vertices of the 5-wheel, establishing the
$\widehat{W}_5$-wheel condition. This concludes the proof of the
implication (iii)$\Rightarrow$(ii).
\medskip

Before proving the implication (ii)$\&$(iii)$\Rightarrow$(i), we will establish the last assertion of Theorem \ref{theorem0}.
Let $\bX$ be a flag prism complex satisfying the wheel, the cube, and the prism conditions.
Then its 2-skeleton  $\bY:=\bX^{(2)}$ is a triangle-square flag complex satisfying the wheel,  the 3-cube, and the 3-prism  conditions.
Let $\widetilde{\bX}$  be the universal cover of $\bX$. Then  the 2-skeleton $\widetilde{\bX}^{(2)}$ of $\widetilde{\bX}$ is a covering
space of  $\bY$. But at the same time  $\widetilde{\bX}^{(2)}$ is simply connected (because the 2-skeleton carries all the
information about the fundamental group), so $\widetilde{\bX}^{(2)}$ is the universal cover of $\bY$. Since $\widetilde{\bX}$ is the
prism complex of $\widetilde{\bX}^{(2)}$ and  $\widetilde{\bX}^{(2)}=\widetilde{\bY}$ satisfies the condition (ii) of Theorem \ref{theorem0},
we conclude that $\widetilde{\bX}$ is a bucolic complex.

\subsection{Proof of (ii)$\&$(iii)$\Rightarrow$(i).} Now, we will show that a flag prism complex $\bX$ satisfying
the conditions (ii) and (iii) of Theorem \ref{theorem0} also satisfies the cube and the prism conditions. We start
with an auxiliary result and some conventions.

\begin{lemma}\label{prism-convex} Any prism $H$ (and in particular, any cube) of $\bX$ is convex, i.e., $G(H)$
induces a convex subgraph of $G(\bX).$
\end{lemma}

\begin{proof} If the 1-skeleton $G(H)$ of a prism $H$ is not convex in $G(\bX),$ then  $G(H)$ is not
locally convex since $G(\bX)$ is weakly modular; cf.\ \cite[Theorem 7]{Ch_triangle}. Thus we can find  two vertices $x,y$ of $H$ at distance 2 in $G(H)$
having a common neighbor outside $H$. Since $x$ and $y$ already have two common (non-adjacent) neighbors in $H,$ we will obtain in
$G(\bX)$ a forbidden $K_{2,3},W^-_4,$ or $W_4$.
\end{proof}

Further we will use recurrently this result without referring to Lemma \ref{prism-convex}, simply saying that the prisms and the cubes of $\bX$ are convex.
Furthermore, when we say that  some vertex set or some subgraph $\pi$ of $G(\bX)$ is a prism or a cube, we mean that $\pi$ is the $0$-skeleton or the
$1$-skeleton of a prism or a cube of $\bX$. In that case, if the meaning is clear from the context, we will denote the resulting prism or cube also by $\pi$.
Finally, the notation $Q_n$ stands for the $n$-cube and $K_n$ for the $n$-clique or $n$-simplex.

\medskip\noindent {\bf Cube condition.} Let $q_1,q_2,q_3$ be
three $k$-cubes of $\bX$ that share a common $(k-2)$-cube $q$ and
pairwise share common $(k-1)$-cubes $q_{ij}$. Note that the vertices of $q_{ij}
\setminus q$ span a $(k-2)$-cube and those of $q_i \setminus q_{ij}$ span a $(k-1)$-cube.
For a vertex $x$ of $q$, let $x_{ij}$ be the unique
neighbor of $x$ in $q_{ij} \setminus q$. Let $x_i$ be the second
common neighbor in $q_i$ of the vertices $x_{ij}$ and $x_{ik}$; $x_i$
is in $q_i \setminus (q_{ij} \cup q_{ik})$. By the 3-cube condition,
there exists a vertex $\mx$ such that $\mx \sim x_1, x_2, x_3$ and
${\mx} \nsim x_{12}, x_{13}, x_{23}$, and the vertices
${\mx},x,x_{12},x_{13},x_{23},x_1,x_2,x_3$ constitute a 3-cube $q_x$ of
$\bX$. Since $x_2 \in I({\mx},x_{12})$, $x_2 \notin q_1$, and
the cube $q_1$ is convex,
${\mx} \notin q_1$. For similar reasons, ${\mx}
\notin q_2, q_3$.  Now, for another vertex $y$ of $q$ we denote by
$y_{12},y_{13},y_{23},y_1,y_2,y_3,\my$ the vertices defined in the
same way as for $x$ and  we denote by $q_y$ the 3-cube spanned by
these vertices and $y$.
From the definition of these vertices we conclude that all $x_i,x_{ij},y_i,y_{ij}$ are distinct and
for all distinct $i,j\in \{ 1,2,3\}$, $x_{ij}\sim y_{ij}$ as well as $x_i\sim
y_i$ hold if and only if $x\sim y$.

\begin{lemma}\label{lem-cube-ijk}
For any $x, y \in q$, for any distinct $i,j,k$, $x \nsim y_i, y_{ij}$,
$x_{ik} \nsim y_i, y_{ij}, y_j$ and $x_i \nsim y_j$.
\end{lemma}

\begin{proof}
If $x$ (respectively, $x_{ik}$) is adjacent to $y_i$ or $y_{ij}$, then since
$x \sim x_{ij}$ (respectively,  $x_{ik} \sim x_i$), either the cube $q_i$ will contain a
triangle  or the cube $q_i \setminus q_{ik}$ is not
convex, which are impossible.  Since $x_{ik}\sim
x$ and $x \nsim y_j$, the convexity of $q_j$ ensures that $x_{ik}
\nsim y_j$. Finally, the convexity of $q_i$ ensures that $x_i \nsim
y_j$, since $y_j \sim y_{ij}$ and $x_i \nsim y_{ij}$.
\end{proof}

\begin{lemma} \label{m_x}
For any $x,y \in q$, for any distinct $i,j$, ${\mx} \nsim y,y_i,y_{ij}$.
\end{lemma}

\begin{proof}
First suppose by way of contradiction that ${\mx}$ is adjacent to $y$ or
$y_{ij}$. Since ${\mx} \notin q_i$, ${\mx} \sim x_i$, and since $x_i
\nsim y, y_{ij}$ by Lemma~\ref{lem-cube-ijk}, we get a contradiction
with the convexity of $q_i$.
Suppose now by way of contradiction that ${\mx}\sim y_i$. If $x \nsim
y$, then $x_i \nsim y_i$ and since both $x_i,y_i \in q_i$ are adjacent
to ${\mx}\notin q_i$, we obtain a contradiction with the convexity of $q_i$.
Now, suppose that $x\sim y$. Then $x_i \sim y_i,
x_{ij} \sim y_{ij}$ and the vertices $x_j,x_{ij},y_{ij},y_i,x_i,{\mx}$
define a double-house; by Lemma~\ref{no-double-house}, it implies that
$x_j \sim y_{ij}$, contradicting Lemma~\ref{lem-cube-ijk}. Thus, ${\mx}
\nsim y_i$.
\end{proof}

\begin{lemma} \label{k-2-cube}
The set $\{ {\mx}: x\in q\}$ spans a $(k-2)$-cube $q'$ of $\bX$ and the
vertices of $q_1 \cup q_2 \cup q_3 \cup q'$ span a $(k+1)$-cube of
$\bX$.
\end{lemma}

\begin{proof}
First note that since $y_1 \sim \my$ and $y_1 \nsim {\mx}$ by
Lemma~\ref{m_x}, we have that ${\mx} \neq \my$. To prove the first
assertion of the lemma, since $q$ is a $(k-2)$-cube of $\bX$, it
suffices to show that ${\mx}\sim \my$ if and only if $x\sim y$.

First suppose that $x$ is adjacent to $y$. Consider the three 2-cubes
induced by the 4-cycles $(x_1,{\mx},x_2,x_{12},x_1),$
$(x_1,y_1,y_{12},x_{12},x_1),$ and $(x_2,y_2,y_{12},x_{12},x_2)$ of
$G(\bX).$ By the 3-cube condition, they are included in a 3-cube of
$\bX$, i.e., there exists a vertex $s$ adjacent to ${\mx},y_1,$ and
$y_2$. Since  $(y_1,y_{12},y_2,\my,y_1)$ is an induced
$4$-cycle in the 1-skeleton of the 3-cube $q_y$, it is also an induced $4$-cycle of
$G(\bX)$. Since $G(\bX)$ does not contain induced
$K_{2,3}$, $W_4^-$ or $W_4$, we conclude that $s = \my$ or $s =
y_{12}$. Since ${\mx} \sim s$ and ${\mx} \nsim y_{12}$ from
Lemma~\ref{m_x}, $s=\my$ and ${\mx}\sim \my$.  Conversely, suppose that
${\mx}\sim \my$ and assume that $x \nsim y$. Then $x_i\nsim y_i$ and
$x_{ij}\nsim y_{ij}$. Since $x_i,y_i\in q_i$ and since $q_i$
is convex, we conclude that $d(x_i,y_i)=2$, (otherwise,
$(x_i,{\mx},\my,y_i)$ would be a shortest path from $x_i$ to
$y_i$). Since $q_i$ is a cube, it implies that $d(x,y)=2$. Let $z$ be
a common neighbor of $x$ and $y$ in the cube $q$ and let $q_z$ be the
3-cube spanned  by the vertices
$z,z_{12},z_{13},z_{23},z_1,z_2,z_3,\mz$. Since $z \sim x,y$, $z_1
\sim x_1, y_1$ and $\mz \sim {\mx}, \my$. Consequently, the vertices
$x_1,z_1,y_1,\my,\mz,{\mx}$ define a double-house, and from
Lemma~\ref{no-double-house}, it implies that $x_1 \sim y_1$, a
contradiction. Therefore, ${\mx}\sim \my$ if and only if $x\sim y$,
whence $q'$ is a $(k-2)$-cube.

From Lemmas~\ref{lem-cube-ijk} and \ref{m_x}, and since $q'$ is a
$(k-2)$-cube, the vertices of $q_1 \cup q_2 \cup q_3 \cup q'$ span a
$(k+1)$-cube of $\bX$.
\end{proof}

\noindent
{\bf Prism condition:} Let $q$ be a $k$-cube intersecting a
simplex $\sigma$ in an edge $xy$. We will prove
that $q\cup \sigma$ is included in a prism of $\bX$. Let $uv$ be
the unique edge of the $k$-cube $q$ that is the farthest from $xy$
such that $d(x,u)=d(y,v)=k-1$ and $d(x,v)=d(y,u)=k.$ Let
$\sigma=\{ x,y, w_1,\ldots, w_m\}.$ Since $q$ is convex,
$d(w_i,u)=d(w_i,v)=k$ for any $i=1,\ldots,m$.

\begin{lemma}\label{k=1} If  $q$ is a 2-cube (i.e., $k=1$), then $q$ and
$\sigma$ satisfy the prism condition.
\end{lemma}

\begin{proof} By the 3-prism condition, the square
$q=xyvu$ together with each triangle $xyw_i$ of $\sigma$ is included in a 3-prism
$H_i$. Let
$a_i$ be the common neighbor in $H_i$ of $u,v,$ and $w_i$. Then $a_i\ne a_j$ for $i\ne j$, otherwise
$w_j\in I(a_i,y),$ contrary to the convexity of $H_i.$ On the other hand, if $a_i\nsim a_j,$ by the
triangle condition there exists a vertex $b$ adjacent to $a_i,a_j,w_j$. To avoid a forbidden
$W_4^-$ or $W_4$ induced by $a_i,u,v,a_j,b$ necessarily $b\sim u,v$. But then $v,b,a_j,y,w_j$
induce a forbidden $W_4^-$ or $W_4$ because $a_j\nsim y$. This shows that $a_i\sim a_j$,
i.e., the vertices of $\sigma$ together with
$u,v, a_1,\ldots a_m$ span a prism $K_2\Box K_{m+2}$.
\end{proof}

Now let $k\ge 2$  and proceed by induction on $k$. Denote by
$q'=I(x,u)$ and $q''=I(y,v)$ the two disjoint $(k-1)$-cubes obtained
from $q$ by removing all edges parallel to $xy$ (and $uv$). Let $x'$ be an arbitrary neighbor of
$x$ in $q'$ and let $y'$ be the neighbor of $y$ in $q''$ such that $xyy'x'$ is a square of $q$. Finally, let $u'$  be the neighbor of $u$ in $q'$ and
$v'$ be the neighbor of $v$ in $q''$ such that $uu'$ and $vv'$ are parallel in $q$ to $xx'$ and $yy'$. Then $u'\sim v'$ and $uvv'u'$ is a square of $q$.
Consider the decomposition of the $k$-cube $q$ with respect to the squares $xyy'x'$ and $uvv'u'$ into four $(k-2)$-cubes $q_x=I(x,u'), q_y=I(y,v'), q_{y'}=I(y',v),$ and
$q_{x'}=I(x',u).$   Note that $q_{x}\cup q_{y}$ and $q_{x'}\cup q_{y'}$ are two  $(k-1)$-cubes constituting $q$.

By the triangle condition, for each vertex $w_i$ of $\sigma$ there exists a vertex $w'_i$ adjacent to $w_i,x',y'$. Since $w'_i\in I(w_i,x')\subset I(w_i,u)$
and $w'_i\in I(w_i,y')\subset I(w_i,v),$ we conclude that $d(w'_i,u)=d(w'_i,v)=k-1.$ By the triangle condition, there exists a vertex $a_i$ adjacent to $u$
and $v$ at distance $k-2$ from $w'_i$ (and at distance $k-1$ from $w_i$).

\begin{lemma} \label{q_i} The interval $I(w_i,a_i)$ spans a $(k-1)$-cube $q_i$ such that $q_i\cup q'\cup q''$ is a prism.
\end{lemma}

\begin{proof} Since $a_i,u'\in I(u,w_i),$ by the quadrangle condition there exists a vertex $b_i\sim u',a_i$ at distance $k-2$ from $w_i$.
By the 3-prism condition, the square $uu'b_ia_i$ and the triangle $ua_iv$  are included into a 3-prism $H$. Since $H$ is convex and
$v'\in I(u',v),$ necessarily $v'$ belongs to $H$, whence $b_i\sim v'$. Applying the induction hypothesis to $\sigma$ and the $(k-1)$-cube $q_{x}\cup q_{y}$,
we conclude that $I(w_i,b_i)$ spans a $(k-2)$-cube $q'_i$ such that $q'_i\cup q_{x}\cup q_{y}$ is a prism $Q_{k-2}\times K_3$. Analogously, applying the induction hypothesis to the 3-simplex $w'_ix'y'$ and the $(k-1)$-cube $q_{x'}\cup q_{y'}$,
we conclude that $I(w'_i,a_i)$ spans a $(k-2)$-cube $q''_i$ such that $q''_i\cup q_{x'}\cup q_{y'}$ is a prism  $Q_{k-2}\times K_3$.

We show now that $q'_i\cup q''_i$ is a $(k-1)$-cube. If $q'_i\cap q''_i\ne\emptyset,$ then a simple distance comparison shows that $b_i\in I(w'_i,a_i).$ Since $a_i\in I(w'_i,u)$ and
$u'\in I(b_i,u),$ we obtain that $u'\in I(w'_i,u),$ contrary to the assumption that  $q''_i\cup q_{x'}$ is a $(k-1)$-cube. Thus $q'_i$ and $q''_i$ are disjoint.  Now, pick any
vertex $z$ of $q'_i.$ Let $x_z$ be the unique neighbor of $z$ in $q_{x}$,  $x'_z$ be the unique neighbor of $x_z$ in $q_{x'},$ and  $w_z$ be the unique
neighbor of $x'_z$ in $q''_i$. We will prove by induction on $r=d(z,b_i)$ that $z$ and $w_z$ are adjacent. If $r=0,$ then we are done because $z=b_i,x_z=u',x'_z=u,$ and
$w_z=a_i.$  Now, let $r>0$. Let $s$ be a neighbor of $z$ in the interval $I(z,b_i).$ Since $d(s,b_i)=r-1,$ by the induction assumption the vertex  $s$ together with the vertices $x_s\in q_{x},$ $x'_s\in
q_{x'},$ and $w_s\in q''_i$ span a square. Applying the 3-cube condition to this square and the squares $zx_zx_ss,x_zx'_zx'_sx_s,$ we conclude that the vertices
$z,x_z,x'_z,s,x_s,x'_s,w_s$ are included in a 3-cube. Since this cube is convex and $w_z\in I(x'_z,w_s),$ necessarily $w_z$ belongs to this cube, whence $z\sim w_z$. Finally, we show that $w_z$ is the unique neighbor of $z$ in $q''_i$. Suppose by way of contradiction that $z$ is adjacent to yet another vertex $t\in q''_i$. Since the cube
$q''_i$ is convex, $w_z\sim t$. Let $t'$ be the neighbor of $t$ in  $q_{x'}.$ Since $t\sim w_z,$ necessarily $t'\sim x'_z$. By the 3-prism condition, the square $w_ztt'x'_z$ and the triangle $ztw_z$ are included in a convex 3-prism.  Since $x_z\in I(z,x'_z),$ necessarily $x_z$ belongs to this prism, i.e., $x_z\sim t'.$ But then the vertex $x_z$ of $q_{x}$ has two neighbors in the cube $q_{x'},$ contrary to the fact that $q_{x}\cup q_{x'}$ is a $(k-1)$-cube. This establishes that $q'_i\cup q''_i$ is a $(k-1)$-cube and that $q_i\cup q'\cup q''$ is a prism $Q_{k-1}\Box K_3.$
\end{proof}

\begin{lemma} \label{q_iq_j} If $i\ne j,$ then $q_i\cup q_j$ is a $k$-cube and $q'\cup q''\cup q_i\cup q_j$ is a prism $Q_{k-1}\Box K_4$.

\end{lemma}

\begin{proof} First we show that the cubes $q_i$ and $q_j$ are disjoint. If this is not the case and $z\in q_i\cap q_j,$ then $d(z,w_i)=d(z,w_j).$ By the triangle condition, there exists a vertex $z_0\in I(z,w_i)\cap I(z,w_j)\subseteq q_i\cap q_j$ adjacent to $w_i$ and $w_j$. Since $z_0$ belongs to  $q_i$ and $q_j$, $z_0$ has a neighbor
$x_0\in q'$ and a neighbor $y_0\in q''$ such that  $x_0\sim y_0$ and $x_0\sim x, y_0\sim y.$ But then the vertices $x,y,z_0,x_0,y_0,w_i$ define a 3-prism, which is not convex
because $w_j\in I(x,z_0),$ a contradiction. Hence, the $(k-1)$-cubes $q_i$ and $q_j$ are disjoint.

Now, we show that $a_i\sim a_j$. Suppose by way of contradiction that $a_i\nsim a_j$. Consider the vertices $u'\in q'$ and
$v'\in q''$ defined above. Recall that $u'\sim v'$ and $u'\sim u, v'\sim v.$ From Lemma \ref{q_i} we know that $u'$ has a
unique neighbor $b_i$ in $q_i$ and a unique neighbor $b_j$ in $q_j$; moreover, $b_i\sim a_i,v'$ and $b_j\sim a_j,v'$.
By induction assumption applied to the simplex $\sigma$ and to the $(k-1)$-cube $q_x\cup q_y$ spanned by the parallel edges $xy$ and $u'v'$,
we conclude that $b_i\sim b_j$. Now, applying the case $k=1$ (Lemma \ref{k=1}) to the 4-simplex spanned by  $b_i,b_j,u',v'$ and
to the 2-cube spanned by  $u',u,v,v',$ we will obtain a contradiction. Thus $a_i\sim a_j$.

Finally, we  establish that $q_i\cup q_j$ is a $k$-cube. Pick two adjacent vertices $z'\in q'$ and $z''\in q''$, and let  $x'\in q_i$ and $y'\in q_j$ be their
common neighbors. If $z'=u$ and $z''=v$, then $x'=a_i,y'=a_j$ and $x'\sim y'$ because $a_i\sim a_j$. Otherwise, if $z'\ne u,z''\ne v,$ then  $x'\sim y'$
follows from the induction hypothesis applied to $\sigma$ and the cube spanned by the parallel edges $xy$ and $z'z''$. This shows that indeed
$q_i\cup q_j$ is a cube.

Since by Lemma \ref{q_i} $q\cup q_i$ and $q\cup q_j$ are prisms of the form $Q_{k-1}\Box K_3$ and $q_i\cup q_j$ is a $k$-cube, we obtain that $q\cup q_i\cup q_j$
is a prism $Q_{k-1}\Box K_4$.
\end{proof}

From Lemma \ref{q_iq_j} we immediately conclude that the vertex set of the union of $q$ with $\cup_{i=1}^mq_i$ spans a prism $Q_k\Box K_{m+2},$ thus establishing the
prism condition. This also  concludes the proof of the implication (ii)$\&$(iii)$\Rightarrow$(i) of Theorem \ref{theorem0} and finishes the proof of Theorem \ref{theorem0}.

\section{Contractibility and the fixed point property}
\label{pf56}

In this section, we prove contractibility and the fixed point theorem for finite group actions for locally-finite bucolic complexes.
The proofs of both results are based on the fact that in a locally-finite graph the convex hull of any finite set is finite (this property
is no longer true for non-locally-finite bucolic graphs).

\subsection{Convex hulls of finite sets.}

\begin{restatable}{proposition}{propconvexhull}\label{convex_hull} If $G=(V,E)$ is a
  locally-finite bucolic graph, then the convex hull $\mbox{conv}(S)$
  in $G$ of any finite set $S\subset V$ is finite.
\end{restatable}

\begin{proof}
By Theorem~\ref{theorem1},
$G$ is a retract of the (weak) Cartesian product $H=\Box_{i\in I} H_i$
of weakly bridged graphs $H_i$. Each $H_i$ is
locally-finite since it is isomorphic to a gated subgraph of $G$.
Note that $G$ is an isometric subgraph of
$H$.  For each index $i\in I,$ let $S_i$ denote the projection of $S$
in $H_i.$ Since the set $S$ is finite and the distance between any two
vertices of $S$ is finite, for all but a finite set $I'$ of indices
$i$ the set $S_i$ is a single vertex.
Since each set $S_i$ is finite, it is included in a ball of $H_i$, which is
necessarily finite. Since the balls in weakly bridged graphs are
convex, we conclude that for each $S_i$, the convex hull
$\mbox{conv}_{H_i}(S_i)$ of $S_i$ in $H_i$ is finite.  The convex hull
$\mbox{conv}_H(S)$ of $S$ in $H$ is the Cartesian product of the
convex hulls of the sets $\mbox{conv}_{H_i}(S_i)$:
$\mbox{conv}_H(S)=\Box_{i\in I} \mbox{conv}_{H_i}(S_i)$.  All
$\mbox{conv}_{H_i}(S_i)$ for $i\in I\setminus I'$ are singletons, thus
the size of $\mbox{conv}_H(S)$ equals the size of $\Box_{i\in I'}
\mbox{conv}_{H_i}(S_i),$ and thus is finite because $I'$ is finite and
each factor $\mbox{conv}_{H_i}(S_i)$ in this product is finite by what
has been shown above.

Since $A:=V\cap \mbox{conv}_H(S)$ is convex in $G$ and it contains the set $S$,
the convex hull of $S$ in $G$ is necessarily included in $A$.
Thus this convex hull is finite, concluding the proof of the proposition.
\end{proof}

Now, we show that Proposition \ref{convex_hull} is false for non-locally-finite bucolic graphs.
Namely, we present an infinite bridged graph $G$ in which all
maximal cliques have size $3$ (i.e., the systolic complex whose $1$-skeleton is $G$ has dimension $3$)
and the convex hull of five of its vertices is infinite.

\begin{example}
The graph $G$ consists of a graph $H$ of girth $6$ (recall that the
{\it girth} of a graph is the length of its smallest cycle) and a
vertex $c$ not belonging to $H$ and adjacent to all vertices of
$H$. Obviously $G$ is bridged,  has diameter $2$, and clique-number
$3$.  The graph $H$ is defined in the following way: it has a set of
four pairwise nonadjacent vertices $A=\{a_0,a_1,a_2,a_3\}$ and a
one-way infinite path $P = \{b_0, b_1, b_2, \ldots, b_j, \ldots \}$
disjoint from $A$. In $H$, $b_j$ is adjacent to $a_{i}$ if and only if
$j = i \mod 4$. For any distinct vertices $a_i,a_{j}$, $d_H(a_i,a_j)
\geq 3$ and thus any cycle containing $a_i$ and $a_j$ has length at
least $6$. Any shortest cycle containing only one vertex $a_i$, has
the form $(b_j,b_{j+1},b_{j+2},b_{j+3},b_{j+4},a_{j\!\!\mod 4},b_j)$ and has
also length at least $6$. Thus the girth of $H$ is $6$. Now, take the
convex hull in $G$ of the $5$-point set $A\cup\{b_0\}$. For each $j$,
note that $b_j$ is in the interval $I(b_{j-1},a_{j\!\!\mod
  4})$. Consequently, one can easily show by induction on $j$ that
conv$(A)$ is the whole graph $G$.
\end{example}

\subsection{Contractibility.}

\thcontractible*

\begin{proof}
Let $\bX$ be a bucolic complex and let $G=(V,E)$ be its
1-skeleton.  Pick any vertex $v_0$ of $G$ and let $B_k(v_0,G)$ be the ball of radius $k$ centered at $v_0$.
Since $G$ is locally-finite, each ball $B_k(v_0,G)$ is finite.  By Proposition \ref{convex_hull} the convex
hulls conv$(B_k(v_0,G)),
k\ge 1,$ are finite. Hence $V$ is an increasing union of the finite
convex sets $\textrm{conv}(B_k(v_0,G)), k\ge 1$.
A subgraph $G'$ of $G$ induced by a convex set of $G$ satisfies the
condition (ii) of Theorem \ref{theorem1}, thus $G'$ satisfies all other conditions
of this theorem, whence $G'$ is bucolic. Hence each
subgraph $G_k$ induced by $\mbox{conv}(B_k(v_0,G))$ is bucolic.

The prism complex $\bX$ is an increasing union of the finite bucolic
complexes $\bX(G_k)$ of the graphs $G_k, k\ge 1.$ Thus, to show that
$\bX$ is contractible, by Whitehead theorem, it suffices to show that
each complex $\bX(G_k)$ is contractible.  By condition (iv) of Theorem
\ref{theorem1}, the graph $G_k$ can be obtained via Cartesian products
of finite weakly bridged graphs using successive gated amalgams.  The
clique complexes of weakly bridged graphs are exactly the weakly
systolic complexes, therefore they are contractible by the results
\cite{Osajda}. Cartesian products of contractible topological spaces
are contractible, thus the prism complexes resulting from the
Cartesian products of prime graphs are contractible. Now, if a graph
$G'$ is a gated amalgam of two finite bucolic graphs $G_1,G_2$ with
contractible prism complexes $\bX(G_1),\bX(G_2)$ along a gated
subgraph $G_0=G_1\cap G_2$ which also has a contractible prism complex
$\bX(G_0),$ then by the gluing lemma \cite[Lemma 10.3]{Bj}, the prism
complex $\bX(G')$ of the bucolic graph $G'$ is also
contractible. Therefore, for each $k$, the prism complex $\bX(G_k)$ is
contractible. This concludes the proof of the contractibility theorem.
\end{proof}

\subsection{Fixed prism property.}

\thfixedprism*

\begin{proof}
Let $\bX$ be a bucolic complex and let $G$ denote the 1-skeleton
of $\bX$.  Let $F$ be a finite group acting by cell automorphisms on
$\bX$ (i.e., any $f\in F$ is a bijection and maps isometrically prisms onto prisms). Then
for an arbitrary vertex $v$ of $\bX,$ its orbit $Fv=\{ fv:\; \; f\in
F\}$ is finite.  Let $G_v$ be the subgraph of $G$ induced by the
convex hull in $G$ of the orbit $Fv$.  Since $Fv$ is finite, the graph
$G_v$ is finite by Proposition \ref{convex_hull}. Moreover, as a
convex subgraph of $G$, $G_v$ satisfies the conditions of Theorem
\ref{theorem1}(ii), hence $G_v$ is bucolic.  Clearly, the prism
complex $\bX(G_v)$ of $G_v$ is $F$-invariant.  Thus there exists a
minimal by inclusion finite non-empty bucolic subgraph $\ovG$ of $G$ whose
prism complex is $F$--invariant. We assert that $\bX(\ovG)$ is a
single prism, i.e., $\ovG$ is the Cartesian product of complete
graphs. We prove this assertion in two steps: first we show that $\ovG$
is a \emph{box}, (i.e., a Cartesian product of prime graphs), and then
we show that each prime graph must be a complete graph. By minimality
choice of $\ovG$ as an $F$-invariant bucolic subgraph, we conclude that
each proper bucolic subgraph of $\ovG$ is not $F$-invariant.  Therefore,
the first step of our proof is a direct consequence of the following
result.

\begin{proposition}
\label{Pfixedbox}
If $\ovG$ is a finite bucolic graph,  then there exists a box that is invariant under every automorphism of $G$.
\end{proposition}

\begin{proof} If $\ovG$ is a box, then the assertion
is trivially true. Suppose now that $\ovG$ is not a box and assume without loss of generality that each proper bucolic
subgraph of $\ovG$ is not $Aut(\ovG)$-invariant.  By Theorem~\ref{theorem1}(iv), $\ovG$ is a gated amalgam of two proper nonempty
gated subgraphs $G'$ and $G''$ along a common gated subgraph $H_0$. Then we say that $H_0$ is a {\it gated separator} of $\ovG$.
Following \cite{br-03}, we will call $U':=G'\setminus H_0$ a \emph{peripheral subgraph} of $\ovG$ if $U'$ does not contain any gated separator of $\ovG$.

Since $\ovG$ is not a box,  it contains at least one gated
separator, and therefore $\ovG$ contains at least one peripheral subgraph (indeed, among all gated separators of $\ovG$ it suffices to consider a gated
separator $H_0$ so that $\ovG$ is the gated amalgam of $G'$ and $G''$ along $H_0$ and $G'$ has minimum size; then $G'\setminus H_0$ is a peripheral subgraph).
Let ${\mathcal U}=\{U_i: i \in I\}$ be the family of all peripheral subgraphs of $\ovG$,
such that $G$ is the gated amalgam of $G_i'$ and $G_i''$ along the gated separator $H_i$, where
$U_i=G_i'-H_i$ and $G_i'' \neq H_i.$ Note that any automorphism $f$ of $\ovG$ maps peripheral subgraphs to peripheral subgraphs,
thus the subgraph $\bigcup_{i\in I}U_i$ and the subgraph $H=\bigcap_{i \in I}G_i''$ induced by the complement of this union are both
$AutovG)$-invariant subgraphs of $\ovG.$ As an intersection of gated subgraphs of $\ovG,$ the graph $H$ is either empty or a proper gated subgraph of $\ovG$. In the second case,
since gated subgraphs of $\ovG$ are bucolic, we conclude that $H$ is a proper bucolic $Aut(\ovG)$-invariant subgraph of $\ovG$, contrary to minimality of $\ovG$. So, $H$ is empty.
By the Helly property for gated sets of a metric space \cite{DrSch}, we can find two indices $i,j\in I$ such that the gated subgraphs $G''_i$ and $G''_j$ are disjoint.
Since $H_i \cap H_j \subseteq G_i'' \cap G_j''$, the gated separators $H_i$ and $H_j$ are disjoint. But in this case, since $U_i=G'_i\setminus H_i$ is peripheral, we conclude that $H_j$ is contained in $G''_i$ (analogously, $H_i$ is contained in $G''_j$). Thus $H_i \cup H_j \subseteq G_i'' \cap G_j''$, contrary to the choice of $G''_i$ and $G''_j$. This finishes the proof of the proposition.
\end{proof}

Thus $\ovG$ is a box, and to finish the proof of Theorem \ref{fixed_prism} it is enough to show the following.

\begin{proposition}
\label{Pfixedprism} The graph $\ovG$ is the Cartesian product of complete graphs, i.e., $\bX(\ovG)$ is a prism.
\end{proposition}

\begin{proof} Let $\ovG=G_1\Box\cdots \Box G_k$, where each factor $G_i, i=1,\ldots,k,$ is a $2$-connected finite weakly bridged graph. By \cite[Theorem B]{ChOs} every factor $G_i$ is
dismantlable. Since dismantlable graphs form a variety (cf.\ e.g.\ \cite[Theorem 1]{NoWi}), it follows that the strong product $G'=G_1\boxtimes \cdots  \boxtimes G_k$ is dismantlable.  Observe that the finite group $F$ acts by automorphisms on $G'$. By the definition of the strong product, any clique of $G'$ is included in a prism of $\bX(G).$ By \cite[Theorem A]{Polat-inv}, there exists a clique $\sigma$ in $G'$ invariant under the action of $F$. Since $F$ acts by cellular automorphisms on ${\bf X}(G)$, it follows that $F$ fixes  the minimal prism containing all vertices of $\sigma$ (treated as vertices of $G$, and hence of ${\bf X}(G)$).
  By the minimality choice of $\ovG$ it follows that ${\bf X}(\ovG)$ is itself a prism.
\end{proof}

This concludes the proof of the fixed prism theorem.
\end{proof}

\section{Moorability of weakly bridged graphs}\label{moorability}

In this section, we extend Theorem 5.1 of \cite{ChOs} and  prove that non-locally-finite weakly bridged
graphs without infinite cliques are moorable. This result is established in \cite{ChOs} via a LexBFS ordering of vertices, which heavily
uses local-finiteness of $G$. Simple examples  show that not every non-locally-finite graph admits a LexBFS ordering. On the other hand,
Polat \cite{Po-infbrid} showed that all graphs admit a  BFS (breadth-first-search) ordering and, extending the result
of~\cite{Ch_bridged}, he showed that this BFS order provides a mooring of non-locally-finite bridged graphs.  In order to circumvent the
bottleneck  of LexBFS, we refine Polat's definition of BFS and define  a well-ordering of the vertices of a graph, which is intermediate
between BFS and LexBFS, that we call SimpLexBFS. We show that any (non-locally-finite) graph without infinite
cliques admits  a SimpLexBFS and that for weakly bridged graphs SimpLexBFS provides a mooring. This will complete the proof of the
implication (iii)$\Rightarrow$(i) of Theorem \ref{theorem1}.

\begin{definition}
A well-order $\leq$ on the vertex-set $V(G)$ of a graph $G$ is a {\it SimpLexBFS order} if for
every vertex $x\in V(G)$, there exists a mapping $L_x: \{y: y > x\} \to 2^{\{t: t\leq x\}}$ satisfying the
following conditions (in what follows we set $L_{(x)}(y)=\bigcup_{t < x}L_{t}(y)$, for $y
  \geq x$):
\begin{enumerate}[(S1)]
\item If $x < y < z$, then $L_{x}(z) \subseteq L_{y}(z)$.
\item If $x<y$ and $L_{(x)}(x) \neq L_{(x)}(y)$, then
  $\min_{\leq} \{ L_{(x)}(x) \Delta L_{(x)}(y)\} \in L_{(x)}(x)$.
\item For $x<y$, we have $L_x(y)=L_{(x)}(y)\cup \{x\}$ if $x\sim y$ and $x\sim t$ for all $t\in L_{(x)}(y)$, and $L_x(y)=L_{(x)}(y)$ otherwise.
\end{enumerate}
\end{definition}

Consider a graph $G$ and a SimpLexBFS order $\leq$ on $V(G)$. We now
explain how to build a spanning tree using $\leq$. Let $u_0$ be the
least element of $(V(G),\leq)$ and for every vertex $v \neq u_0$, let
$f(v) = \min_\leq \{u: u \in L_{(v)}(v)\}$; we say that $f(v)$ is
the \emph{father} of $v$ and $f:V(G)\rightarrow V(G)$ is the {\it father map} of $\leq$.
Note that for every $v \neq u_0$, $f(v)=\min_\leq \{u: u \sim v\}$, and thus,
$f(v) \leq v$. Since $\leq$ is a well-order, the set of edges $\{vf(v): v \neq u_0\}$ constitutes
a spanning tree of $G$.

The following lemma provides some basic properties of SimpLexBFS orders and can be easily proved by (transfinite) induction.

\begin{lemma}
 \label{l7.1}
 Let $\leq$ be a SimpLexBFS order on $V(G)$, let $u_0$ be the least
 element of $(V(G),\leq)$ and let $(L_{x})_{x\in V(G)}$ be the
 corresponding family of mappings. Then the following properties hold:
\begin{enumerate}
\item $\leq$ is a BFS order, i.e., if $v \leq w$, then  $d(v,u_0) \leq
  d(w,u_0)$;
\item if $v \neq u_0$, then  $d(f(v),u_0) = d(v,u_0)-1$;
\item if $f(v) \neq f(w)$, then $v < w$ if and only if $f(v) < f(w)$;
\item if $w \sim v$, then $f(v) \leq w$.
\item if $v \leq w$, then $L_{(v)}(w)\cup \{w\}$ is a
  clique of $V(G)$;
\end{enumerate}
\end{lemma}

Properties (2)-(4) also hold for all BFS orderings. On the other hand, (5) is the property which
distinguishes SimpLexBFS from BFS.

\begin{proposition}\label{SimpLexBFS} If a graph $G$ does not contain  infinite
cliques and $u_0$ is an arbitrary vertex of $G,$  then
there exists a SimpLexBFS order $\leq$ on $V(G)$ such that $u_0$ is the
least element of $(V(G), \leq)$.
\end{proposition}

\begin{proof}
We proceed as in the proof of Lemma 3.6 of \cite{Po-infbrid}. Consider an arbitrary well-order $\triangleleft$ on $V(G)$. We
inductively construct a well-order $\leq$ on $V(G)$ and a family of
mappings $(L_x)_{x \in V(G)}$ satisfying the conditions (S1), (S2), (S3). For  every $y \in V(G)$,
let $L_{u_0}(y)=\{u_0\}$ if $u_0\sim y$ and
$L_{u_0}(y)=\emptyset$ if $u_0\nsim y$. Assume that for a set $I \subseteq V(G)$ (including $u_0$), we
have constructed a well-order $\leq$ and a family of mappings $(L_x)_{x \in I}$ such that:
\begin{enumerate}[(P1)]
\item If $x \in I$ and  $y\in (V(G)\setminus I) \cup I_{>x}$, then $L_x(y) \subseteq I_{\le x}$, where $I_{>x}:=\{ t\in I: t>x\}$ and $I_{\le x}:=\{ t\in I: t\le x\}$.
\item If $x,y \in I$ such that $x < y$ and $z \in
  (V(G)\setminus I) \cup I_{>}(y)$, then $L_x(z) \subseteq L_y(z)$.
\item If $x\in I$, $y\in (V(G)\setminus I)\cup I_{>x}$, and $L_{(x)}(x)\neq L_{(x)}(y)$, then $\min_\leq
  \{ L_{(x)}(x) \Delta L_{(x)}(y)\} \in L_{(x)}(x)$.
\item If $x\in I$ and $y\in (V(G)\setminus I)\cup I_{>x}$, then
  $L_x(y)=L_{(x)}(y)\cup \{x\}$ if $x \sim y$ and $x\sim t$ for all $t
  \in L_{(x)}(y)$; and $L_x(y)=L_{(x)}(y)$ otherwise.
\item If $y\notin I$, then $L_I(y)\cup\{y\}$ induces a clique of $G$, where $L_I(y):=\bigcup_{t \in I}L_{t}(y)$.
\end{enumerate}

\medskip
If $I = V(G)$, then $\leq$ is a SimpLexBFS order on $V(G)$ and we are done. Otherwise, if $I \neq V(G)$,
we iteratively define a set $L'$ as follows. Initially, let $L' = \emptyset$ and while there exists
$y \in V(G)\setminus I$ such that $L' \subsetneq L_I(y)$, we add $$\min_\leq \{x \in I\setminus
L': \exists y \in V(G)\setminus I  \mbox{ such that } L' \subsetneq L_I(y)\}$$ to $L'$. Since by (P5), for
each $y \in V(G)\setminus I$, $L_I(y)$ induces a clique of $G$, and since $G$
does not contain infinite cliques, after a finite
number of steps the iteration stops and that there exists $y \in
V(G)\setminus I$ such that $L_I(y) = L'$.

Let $w$ be the least element of $(\{y \in V(G)\setminus I: L_I(y) =
L'\},\triangleleft)$. We extend $\leq$ by setting $x<w$  for any $x \in I$.
We define $L_w$ as follows: for every $y \notin I\cup\{w\}$, we set
$L_w(y):= L_I(y) \cup\{w\}$ if $w \sim y$ and $w \sim t$ for all $t \in L_I(y)$;
otherwise, we set $L_w(y):=L_I(y)$. Let $I':=I \cup \{w\}$. To complete the
proof of the proposition, it remains to
show that $I'$ satisfies the induction properties (P1)-(P5).

For (P1), if $x<w$, then the property holds by the induction
hypothesis. If $x=w$, then for every $y \notin I\cup\{w\}$ we have $L_w(y)
\subseteq L_I(y)\cup\{w\} \subseteq I'$.

For (P2), if $x<y< w$, then the property holds by the induction
hypothesis. If $x<y = w$, then for every $z \notin I'$
we have $L_x(z)\subseteq L_I(z) \subseteq L_w(z)$.

For (P3), if $x < w$, then the property holds by the induction
hypothesis. If $x=w$, then for every $y \notin I'$, we have $L_{(w)}(y)
= L_I(y)$. By the definition of $L' = L_{I}(w) = L_{(w)}(w)$, either $L_{(w)}(y) = L'$
or $\min_\leq \{ L' \Delta L_{(w)}(y)\} \in L'$.

For (P4), if $x <w$, then the property holds by the induction
hypothesis. If $x=w$, the property holds by the definition of $L_w$.

For (P5), if $x <w$, then the property holds by the induction
hypothesis. If $x=w$, then, by induction hypothesis, $L_I(y)\cup\{y\}$ is a
clique of $G$ for every $y$. If $w \notin L_w(y)$, then $L_w(y)
= L_I(y)$ and we are done. If $w \in L_w(y)$, from the definition of
$L_w(y)$ it follows that for every $t \in L_I(y)\cup\{y\}$ we have $t \sim w$;
consequently, $L_w(y) \cup \{y\}= L_I(y) \cup \{y\} \cup \{w\}$ is
a clique of $G$, and we are done.
\end{proof}

We can now prove the main result of this section.

\begin{proposition}\label{prop-moorable} Any weakly bridged graph $G$ without infinite cliques is moorable.
\end{proposition}

\begin{proof} We proceed as in the proof of Theorem 5.1 of \cite{ChOs}. Let $u_0$ be any vertex of $G$. By Proposition \ref{SimpLexBFS},
$V(G)$ admits a SimpLexBFS order $\leq$, where $u_0$ is the least element of $(V(G),\leq).$  Let
$(L_{x})_{x\in V(G)}$ be the corresponding family of mappings.  For
every vertex $v \neq u_0$, let $f(v) = \min_\leq \{u: u \in L_{(v)}(v)\}$ be the father map of $\leq$.

The following property of weakly bridged graphs immediately follows from the convexity of balls.

\begin{lemma}\label{rem-quad-wbrid}
If $u,v,v',w$ are four  vertices of a  weakly bridged graph $G$ such that  $u \sim
v,v'$ and $v,v' \in I(u,w)$,  then $v \sim v'$.
\end{lemma}

We now prove that $G$ satisfies the fellow-traveler property and
that  $f$ is a mooring of $G$.

\begin{lemma}\label{lem-fpp}
If  $v\sim w$, then either $f(v) = f(w)$ or $f(v)\sim f(w)$; additionally, if $v \leq w$,
then either $f(w)=v$ or $f(w) \sim v$. In particular, the father map $f$ is a mooring of $G$ onto $u_0$.
\end{lemma}

\begin{proof}
Let $w' = f(w)$ and $v' = f(v)$.
To prove the first assertion of the lemma, we proceed by induction on $i+1 = \max\{d(u_0,v), d(u_0,w)\}$.

\medskip
\noindent
{\bf Case 1. $d(u_0,v)<d(u_0,w).$}
\smallskip

Since $\leq$ is a BFS order (Lemma~\ref{l7.1}(1)),
we have $v \leq w$.
By Lemma~\ref{rem-quad-wbrid}, $v$
and $w'$ either coincide or are adjacent. In the first case we are done because $v$ and therefore $w'$ are
adjacent to $f(v)$. If $v$ and $w'$ are adjacent, since
$i=d(u,v)=d(u,w'),$ the vertices $v'$ and $f(w')$ coincide or are
adjacent by the induction assumption. Again, if $v'=f(w')$, we are
done. Now suppose that $v'$ and $f(w')$ are adjacent. Since $w'=f(w)$, we have
$w' \leq v$ (by Lemma~\ref{l7.1}(4)), and by the induction hypothesis, $v' \sim w'$.
This concludes the analysis of Case 1.

\medskip
\noindent
{\bf Case 2. $d(u_0,v)=d(u_0,w)=i+1.$}
\smallskip

Suppose, without loss of generality
that $v \leq w$. If the vertices $v'$ and $w'$ coincide,
then we are done. If $v' \neq w'$, then $v' \leq w'$ because $v \leq w$,
and thus $v' \nsim w$.
If $v'$ and $w'$ are adjacent, then the vertices
$v,w,w',v'$ define a $4$-cycle. Since $G$ is weakly bridged, this
cycle cannot be induced and since $v' \nsim w$, we have $w' \sim v$.
So, assume by way of contradiction that the vertices $v'$ and $w'$ are
not adjacent in $G$. If $v \sim w'$, then $v',w' \in I(v,u_0)$ by
Lemma~\ref{rem-quad-wbrid}, and we get $v' \sim w'$, contrary to our
assumptions. Consequently, $v' \nsim w$ and $w' \nsim v$.

Since $G$ is weakly modular, by TC($u_0$), there exists $s \sim v,w$
such that $d(u_0,s) = i$. Denote by $S$ the set of all such vertices
$s$. For every $s \in S$, since $s,v' \in I(v,u_0)$ (respectively, $s,v' \in
I(w,u_0)$) and since $G$ is weakly bridged, $s \sim v'$ (respectively, $s \sim
w'$).  For every $s \in S$, since $f(v)=v'$, $v' \leq s$ and thus
$f(v') \leq f(s)$ and $f(s) \sim v'$ by the induction hypothesis. For
the same reasons, for every $s \in S$, we have $f(s) \sim w'$.
For every $p \sim v',w'$ and any vertex $s \in S$, the cycle $(p,v',s,w',p)$
cannot be induced and thus $p \sim s$.

\begin{claim}\label{claim-Lvs}
For every $s \in S$, $L_{(v')}(v') \neq L_{(v')}(s)$.
\end{claim}

\begin{proof}[Proof of Claim~\ref{claim-Lvs}]
Suppose $L_{(v')}(v') = L_{(v')}(s)$. If $L_{(v')}(w') =
L_{(v')}(v')$, then we obtain $L_{v'}(s) = L_{(v')}(s) \cup\{v'\} = L_{(v')}(v')
\cup\{v'\}$ (since $L_{(v')}(v') \cup \{v'\}$ is a clique and $v' \sim
s$) and
$L_{v'}(w') = L_{(v')}(v')$ (since $v' \nsim w'$). Consequently, $v' =
\min_\leq \{ L_{v'}(w') \Delta L_{v'}(s)\}=\min_\leq \{ L_{(w')}(w')
\Delta L_{(w')}(s)\}$ and thus $s < w'$, a contradiction. Otherwise, if
$L_{(v')}(w') \neq L_{(v')}(v')$, let $p = \min_\leq \{ L_{(v')}(w') \Delta
L_{(v')}(s)\}$.  Since $L_{(v')}(v') = L_{(v')}(s)$, we conclude that  $p \in
L_{(v')}(s)$.  Consequently, $p = \min_\leq \{ L_{(w')}(w') \Delta
L_{(w')}(s)\}$ and thus $s < w'$, a contradiction.
\end{proof}

\begin{claim}\label{claim-s0}
 Let $s_0$ be the least vertex of $(S,\leq)$ and let $p = \min_{\leq}
 \{ L_{(v')}(v') \Delta L_{(v')}(s_0)\}$. Then for every $s \in S$ we have
 $L_{(p)}(v') = L_{(p)} (w') = L_{(p)}(s)$ and $p \nsim s$.
\end{claim}

\begin{proof}[Proof of Claim~\ref{claim-s0}]
By the definition of $p$, $L_{(p)}(v') =
L_{(p)}(s_0)$. If $L_{(p)}(w') \neq L_{(p)}(v')$, then  $q = \min_\leq
\{ L_{(p)}(v')\Delta L_{(p)}(w')\} \in L_{(p)}(v')$ since $v' \leq
w'$. Consequently, $\min_\leq \{ L_{(p)}(s_0)\Delta L_{(p)}(w')\} \in
L_{(p)}(s_0)$, and hence $s_0 \leq w'$, a contradiction.

Thus $L_{(p)}(w') = L_{(p)}(v') = L_{(p)}(s_0)$. For every $s \in S$ and
$p' \in L_{(p)}(v') = L_{(p)}(w')$ we have $p' \sim s$. Since
$L_{(p)}(v')$ is a clique, we get $L_{(p)}(v') \subseteq
L_{(p)}(s)$. Moreover, since $v' < s$, we have $L_{(p)}(v') =
L_{(p)}(s)$. Since $p \notin L_p(s_0)$ and $L_{p}(v) =
L_{(p)}(s_0) \cup \{p\}$ is a clique, we conclude that $p \nsim s_0$.  If there exists
$s_1 \in S$ such that $p \sim s_1$, then $L_{p}(s_1) = L_{(p)}(s_1) \cup
\{p\} = L_{p}(v)$. In this case, $p = \min_\leq (L_{p}(s_1) \Delta
L_{p}(s_0)) \in L_{p}(s_1)$, and thus $s_1 \leq s_0$, contrary to the
choice of $s_0$.
\end{proof}

Let $s_0$ be the least vertex of $(S,\leq)$ and let $s' = f(s_0)$. By
the induction assumption, we know that $s' \sim v', w'$ because  $v' < w'
< s_0$. Moreover, since $d(s',u_0) = i-1$, we have $s' \nsim v,w$.

Let $p = \min_\leq \{ L_{(v')}(v') \Delta L_{(v')}(s_0)\}$.  From
Claim~\ref{claim-s0}, $s_0 \nsim p$. Since $p \leq s_0$, $d(u_0,p)
\leq i$, and thus, if $p \sim v$ (respectively, $p \sim w$), then $s_0,p \in
I(u_0,v)$ (respectively, $s_0,p \in I(u_0,w)$). By
Lemma~\ref{rem-quad-wbrid}, $s_0\sim p$, a contradiction. If $p \sim
w'$, then $p \sim v',w'$, and thus $p \sim s_0$, a contradiction.

Since $s_0 \nsim p$, we conclude that $s' = f(s_0) \neq p$. If $s' < p$, then $f(v')
\leq s' < p$ and thus $f(v') = \min_\leq L_p(v') = \min_\leq L_p(s_0)
= f(s_0) = s'$; consequently, $s',p \in L_{p}(v')$ and thus $p \sim
s'$.  If $p < s'$, then $p = f(v')$ and by the induction assumption,
$s' \sim p$.

Consequently, $v,w,v',w',s_0,s'$ and $p$ induce in $G$ a $\widehat W_5$.
Thus, by the $\widehat W_5$-condition, there exists a vertex $t \sim
v,w,v',w',s_0,s',p$. Hence $t \in S$, and $t \sim p$, contradicting
Claim~\ref{claim-s0}. This finishes the analysis of Case 2 and concludes the
proof of the first assertion of the lemma.

Finally, we claim that the mapping $f$ is a mooring of $G$ onto $u_0$. Indeed,
for every $v \neq u_0$, we have $v \sim f(v)$ and $d(f(v),u_0) = d(v,u_0)-1$.
Moreover, for any edge $vw$ of $G$, from the first assertion it follows that either
$f(v) = f(w)$ or $f(v)\sim f(w)$, i.e., $f$ is indeed a mooring.
\end{proof}

This concludes the proof of Proposition \ref{prop-moorable}.
\end{proof}


\end{document}